\documentclass[11pt]{article}
\usepackage{graphicx}
\usepackage{amsmath}
\usepackage{amssymb}
\usepackage{amsfonts}
\usepackage{amsthm}
\usepackage[left=2.1cm,top=1.5cm,right=2.1cm, bottom=2.2cm,letterpaper]{geometry}
\usepackage{latexsym}
\usepackage{tcolorbox}
\usepackage{hyperref}
\usepackage{url}
\usepackage{enumitem}
\usepackage[toc]{appendix}
\usepackage[export]{adjustbox}
\usepackage[final]{showkeys}
\usepackage{mathtools}
\mathtoolsset{showonlyrefs}

\hypersetup{
	colorlinks=true,
	linkcolor=red,
	filecolor=magenta,
	urlcolor=cyan,
}

\newtheorem{theorem}{Theorem}[section]
\theoremstyle{plain}

\newtheorem{corollary}[theorem]{Corollary}
\newtheorem{lemma}[theorem]{Lemma}

\newtheorem{problem}{Open Problem}
\newtheorem{proposition}[theorem]{Proposition}
\theoremstyle{definition}
\newtheorem{definition}[theorem]{Definition}
\newtheorem{example}[theorem]{Example}
\newtheorem{remark}[theorem]{Remark}
\numberwithin{equation}{section}

\newcommand{\R}{\mathbb{R}}

\newcommand{\N}{\mathbb{N}}

\newcommand{\B}{{\bf B}}

\renewcommand{\phi}{\varphi}

\begin{document}

\title{A unified approach to the divergence equation\\
and related functional inequalities}

\author{Filippo Gazzola -- Hans-Christoph Grunau -- Gianmarco Sperone}

\date{}

\maketitle

\begin{abstract}
The huge amount of literature about the divergence equation in bounded Lipschitz domains of $\R^n$ ($n\ge2$) is fairly disconnected and even apparently
simple problems remain unsolved. We go several steps further in the knowledge of this equation and of some related inequalities.
We prove that among its infinitely many solutions there exists a special one obeying  elliptic regularity theory. This solution simplifies
the definition of the Bogovskii constant $C_B$ and allows us to prove its attainment in smooth domains. We then obtain a universal lower bound for
$C_B$ in any Lipschitz domain as well as a non-minimality criterion. As expected, balls are minimisers as the domain varies, although
no symmetrisation technique is used. We also analyse ellipsoids and annuli: for the first we improve the (so far) best
asymptotic inequality for thinning domains. Finally, we introduce higher-order Bogovskii constants, which lead to a polyharmonic Stokes problem. Not only regularity
theory applies without smoothness of the domain
when the source has some vanishing traces, but we also prove that balls are again minimisers among Lipschitz domains with the very same Bogovskii
constant. Three main challenging open problems are suggested.\par\noindent
{\bf Mathematics Subject Classification:} 35F05, 35J58, 47A75, 76D07.\par\noindent
{\bf Keywords and phrases:} Bogovskii constant, Schur complement, Cosserat eigenvalues, Ne\v{c}as inequality, higher-order Stokes equations,
optimality in functional inequalities.
\end{abstract}

{
	\hypersetup{linkcolor=black}
	\tableofcontents
}

\eject

\section{Introduction} \label{intro}

Which is the simplest boundary-value problem for a linear ODE on a bounded interval $(a,b)\subset\R$? We believe it is
$u'(x)=f(x)$ with $u(a)=u(b)=0$. This problem is solvable if and only if $f$ has zero mean over $(a,b)$ and, in this case,
the solution is unique and explicitly given by $u(x)=\int_{a}^{x}f(s)ds$.\par
Moving to higher space dimensions $n\ge2$, with $\Omega \subset \mathbb{R}^{n}$ being a bounded domain, this becomes a celebrated problem in
Fluid Mechanics, namely to find a vector field $u\in H^1_0(\Omega, \mathbb{R}^{n})$ such that
\begin{equation}\label{Bogdom}
\nabla\cdot u=f \quad \mbox{in} \quad \Omega\, ,\qquad u=0 \quad \mbox{on} \quad \partial\Omega
\end{equation}
for a given $f\in L^2_0(\Omega, \mathbb{R})$ (zero mean): the existence of a solution to \eqref{Bogdom} is by far nontrivial,
see e.g.\ \cite[Theorem III.3.1]{galdi2011introduction}. If $\overline{u}$ solves \eqref{Bogdom}, then taking
\begin{equation}\label{phiPhi}
\varphi\in \mathcal{C}^2_c(\Omega, \mathbb{R})\, ,\qquad \Phi=(\varphi_{x_2}, -\varphi_{x_1},0, \ldots,0)^{T}\in \mathcal{C}^1_c(\Omega, \mathbb{R}^{n})\, ,
\end{equation}
one has that $\nabla\cdot\Phi=0$ and $\overline{u}+\Phi$ also solves \eqref{Bogdom}. Clearly, one can ``play'' with other partial derivatives
of $\varphi$ or, in dimension $n=3$, simply take $\Phi=$curl$(\psi)$ for any $\psi\in \mathcal{C}^1_c(\Omega, \mathbb{R}^{3})$. By using this trick, one can construct a sequence
$\{u_k\}\subset H^1_0(\Omega, \mathbb{R}^{n})$ of solutions to \eqref{Bogdom} such that $\|\nabla u_k\|_{L^2(\Omega)}\to\infty$ as $k\to\infty$, whatever
$f\in L^2_0(\Omega, \mathbb{R})$ is (including $f\equiv0$!). Therefore, uniqueness and a priori bounds for solutions to \eqref{Bogdom} are
a delicate matter. In view of the twisty story retraced at the end of this section, we will refer to \eqref{Bogdom} as the
{\bf divergence equation}. Moreover, following \cite[(2.3)]{costabel} (see also \cite{gazspefra}), the ``optimal constant'' for a priori bounds of solutions to \eqref{Bogdom} is characterised by
\begin{equation}\label{bogo11}
C_B(\Omega) \doteq   \sup_ {f\in L^2_0(\Omega, \mathbb{R}) \atop \|f\|_{L^2(\Omega)}=1}\quad\inf\Big\{\|\nabla v\|_{L^2(\Omega)}^2\, ;\ v\in H^1_0(\Omega, \mathbb{R}^{n})\, ,
\ \nabla\cdot v=f\text{ in }\Omega\Big\}\, ,
\end{equation}
with $C_B(\Omega) < \infty$ for every bounded Lipschitz domain $\Omega \subset \mathbb{R}^n$, see again
\cite[Theorem III.3.1]{galdi2011introduction}. In spite of earlier contributions in different fields, following the Fluid Mechanics community,
we will refer to \eqref{bogo11} as the {\bf Bogovskii constant}.\par
When we got interested in \eqref{Bogdom} we were  surprised, on the one hand of the huge number of different mathematical groups
working on it and, on the other hand, on the lack of interactions between some of them. The bibliographies of the related papers are not exhaustive and often
relevant previous works are not mentioned. Below we try to reconstruct the history of this mysterious working parallel to each other.
It is our main purpose to connect all together
these contributions, to complement them with new results, to analyse in depth \eqref{Bogdom} and \eqref{bogo11},
and to suggest a unified approach to their analysis.\par
Unless otherwise stated, throughout the paper $\Omega\subset\R^n$ ($n\ge2$) is a bounded Lipschitz domain. As customary for PDEs,
the first step is to set up the problem in some functional-analytic framework and to analyse possible invariances of the equation:
this is done in Section \ref{sec:basics}. In Theorem \ref{minBog} we prove the existence of a ``privileged'' solution to \eqref{Bogdom}:
not only it achieves the (inner) infimum in \eqref{bogo11} and is then subject to a priori bounds, but it also benefits of elliptic
regularity theory, although \eqref{Bogdom} is not elliptic. In fact, this privileged solution is a weak solution to a generalised Stokes system in $\Omega$, see \eqref{gstokes}.

For general $\Omega\subset\R^n$, to find the exact value of
$C_B (\Omega)$ appears to be a challenging task. In Theorem \ref{lowerbound} we prove the sharp universal lower bound $C_B (\Omega)\ge n$
and we provide a criterion for $\Omega$ not to reach this value: this criterion is applied in Section \ref{applications}.\par
In Section \ref{VEL}, we set the analysis in different functional settings such as the Velte-Crouzeix spaces: the
lack of connections between different communities somehow forces us to link together several results related to the \textit{Schur complement
of the Stokes operator}, to the {\em Cosserat eigenvalue problem}, and to the {\em Ne\v{c}as inequality}. We summarise these connections
in Propositions \ref{lem:comm_1.6} and \ref{prop:comm_1.3}. These results are complemented with Theorem \ref{bogoatttheo}, where we
prove that also the (outer) supremum in \eqref{bogo11} is attained in {\em smooth} domains $\Omega$; for general Lipschitz domains, see Open Problem \ref{smooth_to_Lip}.\par
Section \ref{sec:4_1} is devoted to the analysis of \eqref{Bogdom} in balls: by combining Theorem \ref{lowerbound} with earlier contributions leading to \eqref{balln}, we prove that balls minimise the
Bogovskii constant among bounded Lipschitz domains; we also report the proof of \eqref{balln} by describing the whole spectrum of the Cosserat eigenvalue problem.
The natural question is then whether balls are the only minimisers of the Bogovskii constant, see Open Problem \ref{shapeopt}.
In the Appendix (Section \ref{sec:examples}) we complement the results in balls by showing that the privileged solution can be obtained from any solution by applying a projection operator.
Its nonlocal and symmetry breaking behaviour is illustrated by two explicit examples.\par
We then analyse in detail the cases of ellipsoids (Section \ref{sec:ellipsoids}) and annuli (Section \ref{sec:annuli}). In ellipsoids,
we determine explicitly the privileged solution to \eqref{Bogdom} when $f$ belongs to a suitable $n$-dimensional space and
we find upper/lower bounds for $C_B$ (depending on the axes) more precise than those available in literature, see \eqref{CBE} and Open Problem \ref{star}:
as expected, we show that no proper ellipsoid attains $C_B=n$. Also in annuli we are able to prove such result which, perhaps, is less expected,
see Theorem \ref{thm:comm_4.1} where we also show that thinning annulli have a diverging Bogovskii constant in {\em any} dimension $n\ge2$
(this was known only for $n=2$ with a fairly complicated proof \cite{Chizhonkov_Olshanskii_2000}). Our proof also reveals connections with
the (planar) {\em Zhukovsky transform} \cite[Chapter 4]{acheson1990elementary} and with the Friedrichs inequality \cite{friedrichs1937certain}.\par
In Section \ref{polyBogo} we introduce a {\em higher-order version} of the Bogovskii constant, see $C_{B}^{(m)}(\Omega)$ in \eqref{bogok}.
In Theorem \ref{minBoghigh} we prove that, also within this framework, among infinitely many solutions to \eqref{Bogdom} there exists
a privileged one, different from that for \eqref{bogo11}, which achieves the (inner) infimum in \eqref{bogok}, is subject to a priori bounds,
and benefits of elliptic regularity theory, as it solves a \textit{higher-order generalised Stokes system} in $\Omega$, see \eqref{gstokesm} and previous work in \cite{amrouche1992problemes}. In Theorem \ref{lowerboundHO} we prove the sharp universal lower bound
$C_{B}^{(m)}(\Omega)\ge n$ (the same as in Theorem \ref{lowerbound}!) and a criterion for $\Omega$ not to reach this value.
In Proposition \ref{lem:comm_1.6ho} we prove the {\em higher-order version of the Ne\v{c}as inequality} and,
after introducing the \textit{higher-order Schur complement of the Stokes operator}, in Theorem \ref{bogoatttheoho} we show that the
(outer) supremum in \eqref{bogok} is attained on smooth domains $\Omega$. Theorem \ref{finalresult} states that, in any ball, also the higher-order
Bogovskii constant $C_{B}^{(m)}$ in \eqref{bogok} equals the lower bound $n$. While the proof of \eqref{balln} is obtained by exploiting the knowledge
of the Cosserat spectrum, our proof of Theorem \ref{finalresult} is obtained after determining explicitly the Schur complement of the polyharmonic Stokes problem.
\par\medskip\noindent
\textbf{Historical facts and bibliographical notes.}

In a seminal paper, Friedrichs \cite{friedrichs1937certain} proved an inequality between the $L^2$-norm of conjugate harmonic functions (with
zero mean) in planar bounded domains. Ten years later \cite{friedrichs}, he connected his result with the Korn inequality \cite{korn} in elasticity, see also \cite{horganzamp}.
The Friedrichs inequality has natural extensions to
higher space dimensions. In 3D, pairs of conjugate harmonic functions were replaced by Velte \cite[(1)]{velte1998inequalities} with
pairs of vector fields satisfying the so-called Moisil-Teodorescu equations. In {\em any} space dimension $n\ge2$, a generalised Friedrichs
inequality was obtained by Costabel \cite[(1.10)]{costabel2017inequalities} who introduced suitable differential forms.\par
The first a priori bound for \eqref{Bogdom} is due to Cattabriga \cite{cattabriga1961problema} in 3D and for $f\in L^p_0(\Omega,\R)$
($1<p<\infty$), assuming that $u$ also solves a generalised Stokes problem. Several years later, Babu\v{s}ka \& Aziz \cite{babuska} found a priori bounds
for \eqref{Bogdom} in planar domains for numerical purposes, see also subsequent results by Ladyzhenskaya \& Solonnikov
\cite{ladyzhenskaya1978some}. Related to a result by Pileckas \cite{pileckas} is the work by Bogovskii \cite{bogovskii1979solution,bogovskii1980}:
based on the explicit representation formula in star-shaped domains due to Sobolev \cite[Formula (7.9)]{sobolev1963}, he proved existence and a priori bounds for a solution to \eqref{Bogdom} in the scale of Sobolev spaces of positive order.
Without giving references, a priori bounds for \eqref{Bogdom} are named after Ladyzhenskaya-Babu\v{s}ka-Brezzi (LBB) in (e.g.)
\cite{Bernardi,Chizhonkov_Olshanskii_2000}.

In a celebrated paper, Horgan \& Payne \cite{horgan} were able to prove the following simple formula connecting the optimal  Friedrichs (or Korn)
constant $\Gamma(\Omega)$ with $C_B(\Omega)$
\begin{equation}\label{connection}
	C_B(\Omega)=\Gamma(\Omega)+1\, ;
\end{equation}
they name $C_B(\Omega)$ the \textit{Babu\v{s}ka-Aziz constant} because they only consider 2D simply connected domains $\Omega$ as in \cite{babuska}.
By using the contributions in \cite{costabel2017inequalities,velte1998inequalities}, \eqref{connection} has been discussed in higher
space dimensions and for more general domains in \cite{acosta2017divergence,costabel,Costa-McIntosh,duran2012elementary,geissert2006equation,Guzman,ladyzhenskaya1978some,Payne2007,zsuppan2016,zsuppan2018,zsuppan2020}.\par

The functional setting that we use in Section \ref{VEL} is an extension, to any dimension $n\ge2$, of the one employed by Velte \cite{Velte_LNM1431_1990}, although the author himself writes that
{\em our considerations rest upon results which are known (or nearly known)}, see \cite[p.159]{Velte_LNM1431_1990}.
Our Proposition \ref{lem:comm_1.6} follows the lines of \cite[Proposition 1'']{Velte_LNM1431_1990} (written only for $n=3$)
and the variational characterisation \eqref{eq:comm_1.3} is related to the general Ne\v{c}as inequality
\cite{Necas_ineq}, see also \cite[Lemme 3.7.1]{nevcasmethodes}, \cite[Exercise III.3.4]{galdi2011introduction} and earlier results by Cattabriga \cite[pp.312-313]{cattabriga1961problema}.
In fact, \eqref{eq:comm_1.3} may be seen as a ``dual characterisation'' of $C_B$, in the spirit of Fichera \cite{fichera}.
Proposition \ref{lem:comm_1.6} also contains the proof of \cite[Theorem 1]{Velte_LNM1431_1990} which, as mentioned there,
was probably already known, although no written proof seems to exist in literature. It was orally exposed by Velte during an Obwerwolfach
meeting in 1998 and Sohr commented that he was already aware of it, see \cite[Remark, p.162]{Velte_LNM1431_1990}.\par
Finally, let us try to retrace the history of \eqref{balln}. It is proved, e.g., in  \cite[Section 6]{horgan} for $n=2$ and in
\cite[Theorem 2]{Velte_LNM1431_1990} for $n=3$. Moreover, there are indications in the literature that the general statement had been known
for some time. In a talk given in Budapest in March 2016, Zsupp\'an \cite[p.4]{zsuppanslides} attempts to reconstruct the history of
\eqref{Bogdom} overlooking, however, part of the vast literature; in particular, on \cite[p.14]{zsuppanslides} he attributes to Cosserat
\& Cosserat \cite{Cosserat_18986} the first proof of \eqref{balln} but we were not able to find the explicit proof therein.
In the notes on her web page, Dauge \cite[p.2]{Daugeweb} claims that it was Crouzeix \cite{Crouzeix_1997} who first proved \eqref{balln}.
This paper is also mentioned on \cite[p.4]{zsuppanslides} but is not easy to find on the web: Crouzeix himself calls it
``an article introuvable'' (an unobtainable paper), which is downloadable from his webpage. It is indeed true that \eqref{balln} is proved in \cite{Crouzeix_1997}, but... it is very hidden. The paper by Crouzeix \cite{Crouzeix_1997} starts by clarifying that the considered domains are
only 2D and 3D, which is repeated in the titles of Sections 3 and 4. However, at the bottom of p.5, Crouzeix manages to prove \eqref{balln}
without writing this explicitly. His proof is obtained as a consequence of results in ellipsoids which, in turn, are related to much
earlier results by the Cosserat brothers \cite{Cosserat_18987,Cosserat_18986} and Mihlin \cite{Mihlin} in elasticity. We are then back to
the staring point! Having these results in mind (see also \cite{Simader_vonWahl_2006}) one may deduce \eqref{balln} directly from
\cite[Theorem 2]{Velte_LNM1431_1990}, which is again formulated only for $n=3$. Moreover, in view of Proposition \ref{lem:comm_1.6}, a proof
of \eqref{balln} is also given in \cite[Corollary 6]{gaultier1996spectral} (using the Schur complement) where the authors do not cite previous works, neither
\cite{Cosserat_18986} nor \cite{Velte_LNM1431_1990}, although in parts their reasoning is pretty much along those lines.\par\bigskip

\section{The Bogovskii constant}\label{sec:varkiational}

\subsection{Preliminaries}\label{sec:basics}

We collect here some results and remarks which will prove to be useful in what follows.\par
We begin with a simple observation that has several connections with Section \ref{VEL}.
A refinement of the celebrated Helmholtz-Weyl decomposition (see, e.g., \cite[Theorem I.1.5]{temam} or \cite[Theorem~III.1.1]{galdi2011introduction}) states that the spaces
\begin{align}
\mathbb{G}_1 &\doteq \{w\in L^2(\Omega,\R^n);\ \nabla\cdot w=0,\ \gamma_\nu w=0\}\, ,\\
\mathbb{G}_2 &\doteq \{w\in L^2(\Omega,\R^n);\ \nabla\cdot w=0,\ \exists g\in H^1(\Omega,\R),\ w=\nabla g\}\, , \label{scomp}\\
\mathbb{G}_3 &\doteq \{w\in L^2(\Omega,\R^n);\ \exists g\in H^1_0(\Omega,\R),\ w=\nabla g\}\, ,
\end{align}
are mutually orthogonal in $L^2(\Omega,\R^n)$ and $L^2(\Omega,\R^n)=\mathbb{G}_1\oplus\mathbb{G}_2\oplus\mathbb{G}_3$;
in \eqref{scomp} $\gamma_\nu$ is the normal trace operator.
This means that every vector field $w\in L^2(\Omega,\R^n)$ can be uniquely written as the sum of three vectors $w^i\in\mathbb{G}_i$ for $i=1,2,3$.
As mentioned in the introduction, the problem \eqref{Bogdom} admits infinitely many solutions for any $f\in L^2_0(\Omega,\R)$.
For any such solution $u\in H^1_0(\Omega,\R^n)\subset L^2(\Omega,\R^n)$, let $u=u^1+u^2+u^3$ be its orthogonal decomposition in the spaces appearing in \eqref{scomp}. Then
$$f=\nabla\cdot u=\nabla\cdot u^3\, .$$
Since $u^3\in\mathbb{G}_3$, there exists $g\in H^1_0(\Omega,\R)$ such that $u^3=\nabla g$. Inserted into the previous equation,
this shows that $g$ weakly solves the problem $\Delta g=f$ in $\Omega$.
Then, $g$ is uniquely determined and so is $u^3=\nabla g$. Therefore,
\begin{equation}\label{G3}
\mbox{although \eqref{Bogdom} admits infinitely many solutions, all of them have the same projection over $\mathbb{G}_3$.}
\end{equation}

The next statement says that the inner infimum in \eqref{bogo11} is always bounded from below by $1$. On the contrary, the Bogovskii constant can be arbitrarily large
in ``thin'' domains, see Theorem~\ref{ellipsoid} and Remark \ref{rem:thin_rectangles} below.

\begin{lemma}\label{lem:general_lower_bound}
For any $u\in H^1_0(\Omega, \mathbb{R}^{n})$ we have
	\begin{equation}\label{lowBog}
	\| \nabla \cdot u\|_{L^2(\Omega )}\le \| \nabla u\|_{L^2(\Omega )}.
	\end{equation}
\end{lemma}

\noindent
\begin{proof} By density, it suffices to consider $u\in \mathcal{C}^\infty_c(\Omega, \mathbb{R}^{n})$.
Integrating twice by parts and applying the Cauchy-Schwarz inequality yields:
\begin{align*}
\| \nabla \cdot u\|^2_{L^2(\Omega )}
=&\int_\Omega \left(\sum_{i=1}^n \partial_i u_{i}\right)^2 \, dx
=\sum_{i,j=1}^n \int_\Omega \left( \partial_i u_{i}\right) \, \left( \partial_j u_{j}\right)\, dx
=\sum_{i,j=1}^n \int_\Omega \left( \partial_i u_{j}\right) \, \left( \partial_j u_{i}\right)\, dx\\
\le&  \int_\Omega \left( \sum_{i,j=1}^n(\partial_i u_{j})^2\right)^{1/2} \, \left(\sum_{i,j=1}^n (\partial_j u_{i})^2\right)^{1/2}\, dx
= \int_\Omega |\nabla u |^2 \, dx= \| \nabla u\|^2_{L^2(\Omega )},
\end{align*}
and \eqref{lowBog} follows.	
\end{proof}

Next, we observe that the Bogovskii constant \eqref{bogo11} is invariant under translation, dilation and rotation of domains; therefore,
there is no monotonicity of $C_B$ with respect to domain inclusions. Otherwise $C_B(\Omega)$ would be the same for all $\Omega\subset \R^n$; in Section~\ref{symmBog} we shall show that this is not the case.

\begin{lemma}\label{invariance}
	For  any $h\in\R^n$, any $k>0$, and any $\mathcal{A}\in O(n)$, we have that
	$$
	C_B(\mathcal{A}(\Omega))=C_B(k\Omega)=C_B(\Omega+h)=C_B(\Omega)\, .
	$$
\end{lemma}

\noindent
\begin{proof}We only prove the invariance under rotation. Let $\mathcal{A}\in O(n)$. For any $f\in L^2_0 (\mathcal{A}(\Omega), \mathbb{R})$ and $u\in H^1_0 (\mathcal{A}(\Omega),\R^n)$
	solving \eqref{Bogdom} in $\mathcal{A}(\Omega)$, we define
	$$
	f_\mathcal{A}\in L^2_0(\Omega,\R),\quad f_\mathcal{A}(x)\doteq f(\mathcal{A}x),\quad u_\mathcal{A}\in H^1_0 (\Omega,\R^n),\quad u_\mathcal{A}(x)\doteq \mathcal{A}^T \cdot u(\mathcal{A}x).
	$$
	Then, obviously $\|f_\mathcal{A} \|_{L^2(\Omega )}=\|f \|_{L^2(\mathcal{A}(\Omega) )}$, and further
	\begin{align*}
	\nabla u_\mathcal{A}(x) =& \mathcal{A}^T \cdot (\nabla u)(\mathcal{A}x)\cdot \mathcal{A},\\
	|\nabla u_\mathcal{A}(x)|^2=& \operatorname{tr} \left( \left( \mathcal{A}^T \cdot (\nabla u)(\mathcal{A}x)\cdot \mathcal{A}\right)^T\cdot \left(\mathcal{A}^T \cdot (\nabla u)(\mathcal{A}x)\cdot \mathcal{A}\right)\right)\\
	=& \operatorname{tr} \left(\mathcal{A}^T \cdot  (\nabla u)^T(\mathcal{A}x)\cdot(\nabla u)(\mathcal{A}x)\cdot \mathcal{A} \right)
	=\operatorname{tr} \left(\mathcal{A}^T \cdot \mathcal{A}\cdot  (\nabla u)^T(\mathcal{A}x)\cdot(\nabla u)(\mathcal{A}x) \right) \\
	=&\operatorname{tr} \left(  (\nabla u)^T(\mathcal{A}x)\cdot(\nabla u)(\mathcal{A}x) \right)=|(\nabla u)(\mathcal{A}x)|^2,\\
	\|\nabla u_\mathcal{A} \|_{L^2(\Omega )}=&\|\nabla u \|_{L^2(\mathcal{A}(\Omega) )}.
	\end{align*}
	Finally, we have a.e. in $\Omega$:
	\begin{align*}
	\nabla \cdot u_\mathcal{A} (x) =&\operatorname{tr} \left(\nabla  u_\mathcal{A} (x)\right)
	= \operatorname{tr} \left(  \mathcal{A}^T \cdot (\nabla u)(\mathcal{A}x)\cdot \mathcal{A}\right)
	= \operatorname{tr} \left(  \mathcal{A}^T \cdot \mathcal{A}\cdot (\nabla u)(\mathcal{A}x)\right)\\
	=&\operatorname{tr} \left(  (\nabla u)(\mathcal{A}x)\right)
	=(\nabla \cdot u)(\mathcal{A}x)=f(\mathcal{A}x) =f_\mathcal{A}(x),
	\end{align*}
	so that $f_\mathcal{A},u_\mathcal{A}$ solve \eqref{Bogdom} in $\Omega$.
	We conclude that $C_B(\mathcal{A}(\Omega))\ge C_B(\Omega)$. The same argument with $\mathcal{A}(\Omega)$ instead of $\Omega$ and $\mathcal{A}^T$ instead of $\mathcal{A}$ gives
	$C_B(\Omega)=C_B(\mathcal{A}^T(\mathcal{A}(\Omega)))\ge C_B(\mathcal{A}(\Omega))$, which is the reverse inequality. Equality then follows.	
\end{proof}

We conclude this section with a simple observation on divergence-free vector fields.

\begin{lemma}\label{lem:lemma1}
For any vector field
\begin{equation}\label{H10s}
v \in H^{1}_{0,\sigma}(\Omega,\R^n)\doteq\{w\in H^1_0(\Omega,\R^n);\, \nabla \cdot w=0\mbox{ in }\Omega\},
\end{equation}
any $j\in\{1,\ldots,n\}$ and any $g\in\mathcal{C}^1(\R,\R)$, one has
\begin{equation} \label{exlemma}
\int_\Omega g'(x_j )v_j(x)\, dx=0.
\end{equation}	
\end{lemma}
\noindent
\begin{proof}
Let $\Omega$, $v$, and $g$ be as in the lemma. Then for  	
arbitrary $\varphi\in \mathcal{C}^1(\overline{\Omega}, \mathbb{R})$,  integration by parts yields:
\begin{align*}
0=& \int_\Omega \varphi (x) (\nabla \cdot v) (x) \, dx= -\int_\Omega \nabla\varphi(x) \cdot v(x)\, dx.
\end{align*}
Choosing $j\in\{1,\ldots,n\}$ and $\varphi(x)= g(x_j)$, we find \eqref{exlemma}.\end{proof}

\subsection{A variational characterisation and a sharp lower bound}

Our first result states that the (inner) infimum in \eqref{bogo11} is, in fact, a minimum and that it is achieved by a unique vector field; cf.
\cite{costabel2017inequalities,costabel,horgan,ladyzhenskaya1978some,velte1998inequalities,zsuppan2020}. This means that, among all solutions of \eqref{Bogdom} in
$H^1_0 (\Omega, \mathbb{R}^{n})$, there is a privileged one.

\begin{theorem}\label{minBog}
For any $f\in L^2_0(\Omega,\R)$ there exists a unique solution
$u_f\in H^1_0(\Omega,\R^n)$ to \eqref{Bogdom} such that
\begin{equation}\label{inf}
\|\nabla u_f\|_{L^2(\Omega)}^2=\inf\Big\{\|\nabla v\|_{L^2(\Omega)}^2\, ;\ v\in H^1_0(\Omega,\R^n)\, ,\ \nabla\cdot v=f\text{ in }\Omega\Big\}\, .
\end{equation}
Moreover, with $H^1_{0,\sigma}(\Omega,\R^n)$ as in \eqref{H10s}, satisfying the following Euler-Lagrange equation
\begin{equation}\label{ELeq}
\int_\Omega \nabla u : \nabla v\, dx=0\qquad\forall v\in H^1_{0,\sigma}(\Omega,\R^n)
\end{equation}
is necessary and sufficient for a solution $u\in H^1_0(\Omega,\R^n)$ of \eqref{Bogdom} to minimise \eqref{inf} so that $u\equiv u_f$. Equivalently,
there exists a unique $p_{f} \in L^2_0(\Omega,\R)$ such that the pair $(u_{f},p_{f})$ weakly solves the generalised Stokes system
\begin{equation} \label{gstokes}
\left\{
\begin{aligned}
	& - \Delta u_{f} + \nabla p_{f} = 0 \, , \quad \nabla \cdot u_{f} = f \ \ \mbox{ in } \ \ \Omega \, , \\[5pt]
	& u_{f}=0 \ \ \mbox{ on } \ \ \partial \Omega  \, .
\end{aligned}
\right.
\end{equation}
Moreover, if $\Omega$ has a boundary of class $\mathcal{C}^{m+1,1}$ and $f\in H^{m}(\Omega, \mathbb{R}) \cap L^2_0(\Omega, \mathbb{R})$, for some $m \in \mathbb{N}$, then $u_f\in H^{m+1}(\Omega, \mathbb{R}^{n}) \cap H^1_0(\Omega, \mathbb{R}^{n})$.
\end{theorem}
\noindent
\begin{proof} If $f=0$ the result is trivial and $u_f=0$.\par
If $f\neq0$, by linearity of \eqref{Bogdom} we may assume that $\|f\|_{L^2(\Omega)}=1$. Fixed any such $f$, consider a minimising sequence,
that is,
a sequence of solutions $\{u^k\}\subset H_0^1(\Omega,\R^n)$ to \eqref{Bogdom} whose norms $\|\nabla u^k\|_{L^2(\Omega)}^2$ tend to the infimum in \eqref{inf}.
Then $\{u^k\}\subset H_0^1(\Omega,\R^n)$ is bounded and, up to a subsequence, it converges weakly to some $u_f\in H_0^1(\Omega,\R^n)$. In turn,
$\nabla\cdot u^k\rightharpoonup \nabla\cdot u_f$ in $L^2(\Omega,\R)$ so that $\int_\Omega \phi\nabla\cdot u^k\, dx\to\int_\Omega \phi\nabla\cdot u_f\, dx$ as
$k\to\infty$ for all $\phi\in L^2(\Omega,\R)$, which shows that $u_f$ solves \eqref{Bogdom}. Moreover, by lower semicontinuity of the norm with respect
to weak convergence, we have
$$
\|\nabla u_f\|_{L^2(\Omega)}^2 \le \liminf_{k\to\infty}\|\nabla u^k\|_{L^2(\Omega)}^2
= \inf\Big\{\|\nabla v\|_{L^2(\Omega)}^2\, ;\ v\in H^1_0(\Omega, \R^n)\, ,\ \nabla\cdot v=f\text{ in }\Omega\Big\}=:M
$$
which shows that $u_f\in H_0^1(\Omega,\R^n)$ achieves the minimum in \eqref{inf}.\par
In order to prove uniqueness, assume that two vector fields $u,v\in H^1_0(\Omega,\R^n)$ solve
\eqref{Bogdom} with $\|\nabla u\|_{L^2(\Omega)}^2=\|\nabla v\|_{L^2(\Omega)}^2=M$. Then, by linearity,
any convex combination $w_\alpha=\alpha u+(1-\alpha)v$ also solves \eqref{Bogdom}. Moreover, by the H\"older inequality,
\begin{align*}
\|\nabla w_\alpha\|_{L^2(\Omega)}^2 &= \|\alpha\nabla u+(1-\alpha)\nabla v\|_{L^2(\Omega)}^2\\
& = \alpha^2\|\nabla u\|_{L^2(\Omega)}^2
+(1-\alpha)^2\|\nabla v\|_{L^2(\Omega)}^2+2\alpha(1-\alpha)\int_\Omega\nabla u : \nabla v\, dx\\
& \le \alpha^2 M+(1-\alpha)^2M+2\alpha(1-\alpha)\|\nabla u\|_{L^2(\Omega)}\|\nabla v\|_{L^2(\Omega)}=M\, .
\end{align*}
Since $M$ is the minimum, the above inequality is an equality, that is,
$$
\int_\Omega\nabla u : \nabla v\, dx=\|\nabla u\|_{L^2(\Omega)}\|\nabla v\|_{L^2(\Omega)}\, .
$$
But the H\"older inequality becomes an equality if and only if the two involved functions are proportional, namely
$\nabla v\equiv\gamma \nabla u$ for some $\gamma>0$. Since $\|\nabla u\|_{L^2(\Omega)}^2=\|\nabla v\|_{L^2(\Omega)}^2=M$, we necessarily
have $\gamma=1$ so that $\nabla v\equiv\nabla u$ and $v\equiv u$.
This also proves that \eqref{ELeq} is satisfied for any $v\in H^1_{0,\sigma}(\Omega,\R^n)$.\par
Conversely, the linearity of the subspace of $H^1_0(\Omega,\R^n)$
of solutions  of \eqref{Bogdom} and the strict convexity of the functional $v\mapsto \| \nabla v\|^2_{L^2(\Omega)}$ on this subspace prove
also the sufficiency of \eqref{ELeq}.
\par
Making now use of the decomposition \eqref{scomp} of $L^2(\Omega,\R^n)$, we see that \eqref{ELeq} is the weak formulation of \eqref{gstokes}.
We obtain a unique pressure $p_f\in L^2_0(\Omega,\R)$ corresponding to the privileged solution  $u_f$ of the divergence equation \eqref{Bogdom} such that
$$
(\Delta u_f)_j = \partial_{x_{j}} p_f \qquad \forall j \in \{1,...,n\} \, .
$$
This yields system \eqref{gstokes}.

Finally, the higher-order regularity statement follows directly from \cite[Theorem IV.5.8]{boyer2012mathematical}.
\end{proof}

Theorem \ref{minBog} deserves several comments. First, it enables us to slightly modify the variational characterisation of $C_B(\Omega)$
in \eqref{bogo11}:
\begin{equation}\label{bogo}
C_B(\Omega) \doteq \sup_{f\in L^2_0(\Omega,\R) \setminus \{0\}} \ \dfrac{\| \nabla u_{f} \|_{L^2(\Omega)}^2}{\| f \|_{L^2(\Omega)}^2} \, ,
\end{equation}
with $u_{f} \in H_{0}^{1}(\Omega,\R^n)$ being the unique vector field in $H_{0}^{1}(\Omega,\R^n)$ satisfying \eqref{ELeq}.\par
Second, the regularity statement in Theorem \ref{minBog} {\em only} refers to the special solution $u_f$. A simple counter-example is
as follows: let $\B\subset\R^2$ be the unit disc and let
$$
f\equiv0\mbox{ in }\B\, ,\qquad u(x)=\left(\frac{x_2(1-|x|^2)}{|x|^{2/3}}\, ,\, \frac{x_1(|x|^2-1)}{|x|^{2/3}}\right)\, ,
$$
so that $\nabla\cdot u=f\in C^\infty(\overline{\B},\R)$ and $u\in H^1_0\setminus H^2(\B,\R^2)$, while $u_f\equiv0$.
The (obvious) reason of this failure is that elliptic regularity does not apply to \eqref{Bogdom} while the special solution $u_f$
also solves the elliptic problem \eqref{gstokes}. We come back to this issue in Section \ref{polyBogo}.\par
Third, and most important, by using Theorem \ref{minBog}, we also provide a sharp lower bound for the Bogovskii constant in
bounded  Lipschitz domains $\Omega\subset\R^n$ and a non-minimality criterion that will be applied in Section \ref{applications}.
Given any index $i \in \{1,...,n\}$, let
\begin{equation}\label{eq:def_K_i}
\mathcal{K}_{i}(\Omega, \mathbb{R}) \doteq \{ g : \Omega \to \mathbb{R} \ ; \ g(x)=g(x_i) \quad \forall x \in \Omega \ \} \, 
\end{equation}
denote the vector space of scalar functions on $\Omega$ that depend, at most, on the $i$-th variable (in particular, $\mathcal{K}_i(\Omega, \mathbb{R})$ contains constant functions).
In the next result we show that the Bogovskii constant is bounded from below by $n$, thereby improving \eqref{lowBog}; moreover, this lower bound is sharp.
Indeed, in Section \ref{sec:4_1} below we shall see that it is attained on any ball of $\R^n$.

\begin{theorem}\label{lowerbound}
	For any bounded Lipschitz domain $\Omega\subset\R^n$ one has
	$$
	C_B(\Omega)\ge n \, .
	$$
Moreover, $C_B(\Omega)>n$ whenever one of the following two facts occurs:\\
$\bullet$ there exist $i \in \{1,...,n\}$ and $g\in\mathcal{K}_{i}(\Omega, \mathbb{R}) \cap L^{2}(\Omega, \mathbb{R})$ such that the unique solution $w \in H^{1}_{0}(\Omega, \mathbb{R})$ of
\begin{equation}\label{generalTorsion}
- \Delta w = g\mbox{ in }\Omega\, , \qquad w = 0\mbox{ on }\partial \Omega  \, ,
	\end{equation}
	satisfies
	\begin{equation}\label{condition}
	\int_\Omega \left| \dfrac{\partial w}{\partial x_{i}} \right|^2 \, dx<\frac{1}{n}\int_\Omega|\nabla w|^2\, dx\, ,
	\end{equation}
$\bullet$ there exists $h\in \mathcal{C}(\R, \mathbb{R})$ such that the unique solution to
\begin{equation}\label{hsum}
-\Delta w=h\left(\sum_{i=1}^nx_i\right)\mbox{ in }\Omega\, ,\qquad w=0\mbox{ on }\partial\Omega\, ,
\end{equation}
satisfies
\begin{equation}\label{negativepd}
\sum_{1=i<j}^{n}\int_\Omega\partial_{x_i}w\, \partial_{x_j}w\, dx<0\, .
\end{equation}
\end{theorem}
\noindent\begin{proof} We start by noticing the elementary inequality for scalar functions
	\begin{equation}\label{elementary}
	\int_\Omega|\nabla w|^2 \, dx=\sum_{i=1}^{n} \left( \int_\Omega \left| \dfrac{\partial w}{\partial x_{i}} \right|^2 \, dx \right) \ge n \min_{k \in \{1,...,n \} }\ \int_\Omega \left| \dfrac{\partial w}{\partial x_{k}} \right|^2\, dx \qquad
	\forall w\in H^1_0(\Omega, \mathbb{R})\, .
	\end{equation}
	We apply \eqref{elementary} to the weak solution $w\in H^1_0(\Omega, \mathbb{R})$ of the {\em torsion problem}
\begin{equation}\label{torsionpb}
- \Delta w = 1\mbox{ in }\Omega \, ,\qquad w = 0\mbox{ on }\partial \Omega \, ,
\end{equation}
that is,
\begin{equation}\label{partialk0}
\int_\Omega \nabla w \cdot \nabla \varphi \, dx = \int_\Omega \varphi \, dx\qquad \forall \varphi \in H^1_0(\Omega, \mathbb{R}) \, .
\end{equation}
Let  $k\in\{1,...,n\}$ denote  the index of the partial derivative of $w$ attaining the minimum in the right hand side of \eqref{elementary};
	clearly, there may be more than one such index, e.g.\ in domains with some symmetry property. Then, \eqref{elementary} becomes
	\begin{equation}\label{partialk}
	\int_\Omega|\nabla w|^2 \, dx\ge n\, \int_\Omega \left| \dfrac{\partial w}{\partial x_{k}} \right|^2 \, dx \, .
	\end{equation}
	Take the vector field $u=(u_1,...,u_n)\in H^1_0(\Omega, \mathbb{R}^{n})$ such that $u_k=w$ and $u_i\equiv0$ for $i\neq k$, and define $f \in L_{0}^{2}(\Omega, \R)$ by
	$$
	f\doteq\nabla\cdot u=\dfrac{\partial w}{\partial x_{k}} \quad \text{in} \ \ \Omega \, .
	$$
	We claim that $u = u_f$ as in Theorem \ref{minBog}; given any $v = (v_{1},...,v_{n}) \in H^{1}_{0,\sigma}(\Omega, \mathbb{R}^{n})$, notice that
	$$
	\int_{\Omega} \nabla u : \nabla v \, dx= \int_{\Omega} \nabla w \cdot \nabla v_{k}\, dx = \int_{\Omega} v_{k}\, dx = 0 \, ,
	$$
	owing to Lemma \ref{lem:lemma1} and \eqref{partialk0}, so that the Euler-Lagrange equation \eqref{ELeq} is fulfilled.
	From Theorem \ref{minBog} and \eqref{bogo}+\eqref{partialk} we then deduce
	$$
	C_B(\Omega)=\sup_ {h \in L^2_0(\Omega) \setminus \{0\}}\frac{\|\nabla u_h\|_{L^2(\Omega)}^2}{\|h\|_{L^2(\Omega)}^2} \geq \frac{\|\nabla u_f\|_{L^2(\Omega)}^2}{\|f\|_{L^2(\Omega)}^2}=\frac{\|\nabla w\|_{L^2(\Omega)}^2}{\|\partial_{x_k}w\|_{L^2(\Omega)}^2} \ge n \, ,
	$$
	thus providing the stated lower bound for $C_B(\Omega)$.\par
	In order to prove the first non-minimality criterion, consider $i \in \{1,...,n\}$, $g\in\mathcal{K}_{i}(\Omega, \mathbb{R}) \cap L^{2}(\Omega, \mathbb{R})$ and $w\in H^1_0(\Omega, \mathbb{R})$  satisfying \eqref{generalTorsion}-\eqref{condition}, meaning that
	\begin{equation}\label{partialk1}
	\int_\Omega \nabla w \cdot \nabla \varphi  \, dx = \int_\Omega g \, \varphi \, dx\qquad \forall \varphi \in H^1_0(\Omega, \mathbb{R}) \, .
	\end{equation}
	Take the vector field
	$u=(u_1,...,u_n)\in H^1_0(\Omega, \mathbb{R}^{n})$ such that $u_i=w$ and $u_j\equiv0$ for $j \neq i$, and define $h \in L_{0}^{2}(\Omega, \mathbb{R})$ by
	$$
	h\doteq\nabla\cdot u=\dfrac{\partial w}{\partial x_{i}} \quad \text{in} \ \ \Omega \, .
	$$
	As before, we claim that $u = u_h$; given any vector field $v = (v_{1},...,v_{n}) \in H^{1}_{0,\sigma}(\Omega, \mathbb{R}^{n})$, notice that
	$$
	\int_{\Omega} \nabla u : \nabla v \, dx= \int_{\Omega} \nabla w \cdot \nabla v_{i} \, dx= \int_{\Omega} g \, v_{i}\, dx = 0 \, ,
	$$
	owing to Lemma \ref{lem:lemma1} and \eqref{partialk1}, so that the Euler-Lagrange equation \eqref{ELeq} is fulfilled.  Then, by using \eqref{bogo} and \eqref{condition}, we find
	$$
	C_B(\Omega)\ge\frac{\|\nabla u_{h}\|_{L^2(\Omega)}^2}{\|h\|_{L^2(\Omega)}^2}=\frac{\|\nabla w\|_{L^2(\Omega)}^2}{\|\partial_{x_i}w\|_{L^2(\Omega)}^2}>n\, ,
	$$
	which completes the proof of the first criterion.\par
For the second criterion, let $w\in H^1_0(\Omega, \mathbb{R})$ satisfy both \eqref{hsum} and \eqref{negativepd}.
Take the vector field $u=(u_i,...,u_n)=(w,...,w)\in H^1_0(\Omega, \mathbb{R}^{n})$. Then \eqref{ELeq} is satisfied and
$$
|\nabla u|^2=n|\nabla w|^2\, ,\qquad|\nabla\cdot u|^2=|\nabla w|^2+2\sum_{1=i<j}^{n}\partial_{x_i}w\, \partial_{x_j}w\, .
$$
Therefore, in view of \eqref{negativepd}, we have
$$
\int_\Omega|\nabla u|^2\, dx=n\int_\Omega|\nabla\cdot u|^2\, dx-2n\sum_{1=i<j}^{n}\int_\Omega\partial_{x_i}w\, \partial_{x_j}w\, dx>n\int_\Omega|\nabla\cdot u|^2\, dx\, ,
$$
which proves again that $C_B(\Omega)>n$. We point out that this second criterion is equivalent to the first one after a careful
(and lengthy) change of variables. We provided here a direct proof.\end{proof}

\subsection{Applications of the non-minimality criterion in 2D symmetric domains}\label{applications}

The non-minimality criterion in Theorem \ref{lowerbound} appears to be useful in symmetric domains where the torsion problem \eqref{torsionpb}
may give equal $L^2$-norms for all the partial derivatives of the solution.
As a first application, we show that a square in $\R^2$ does not minimise $C_B$.

\begin{example}\label{Example_3.1}
The unique solution to
$$
-\Delta w=\sin(x_1)\quad\mbox{in }(0,\pi)^2\, ,\qquad w=0\quad\mbox{on }\partial(0,\pi)^2\, ,
$$
is explicitly given by
$$
w(x_1,x_2)=\sin(x_1)\left(1-\frac{\sinh(x_2)+\sinh(\pi-x_2)}{\sinh(\pi)}\right)\, .
$$
Then we compute
$$
\partial_{x_1}w=\cos(x_1)\left(1-\frac{\sinh(x_2)+\sinh(\pi-x_2)}{\sinh(\pi)}\right)\, ,\quad
\partial_{x_2}w=\sin(x_1)\, \frac{\cosh(\pi-x_2)-\cosh(x_2)}{\sinh(\pi)}\, ,
$$
and notice that
\begin{align}
\int_{(0,\pi)^2}|\partial_{x_1}w|^2 \, dx&= 
\frac{\pi}{2\sinh^2(\pi)}\int_{0}^{\pi}\left[\sinh(\pi)-\sinh(x_2)-\sinh(\pi-x_2)\right]^2dx_2 \notag\\
&= \frac{\pi(\cosh(\pi)-1)}{2\sinh^2(\pi)}\left( \pi\cosh(\pi)+2\pi -3\sinh(\pi) \right)\notag\\
&\stackrel{(*)}{<}\frac{\pi(\cosh(\pi)-1)}{2\sinh^2(\pi)}\left( \sinh(\pi)-\pi \right)\notag\\
&= \frac{\pi}{2\sinh^2(\pi)}\int_{0}^{\pi}\left[\cosh(\pi-x_2)-\cosh(x_2)\right]^2dx_2=\int_{(0,\pi)^2}|\partial_{x_2}w|^2\, dx \label{square}
\end{align}
so that \eqref{condition} is satisfied  and
\begin{equation}\label{eq:bogovskii_rectangle}
C_B\big((0,\pi)^2\big)\ge1+\frac{\int_{(0,\pi)^2}|\partial_{x_2}w|^2\, dx}{\int_{(0,\pi)^2}|\partial_{x_1}w|^2 \, dx}=
\frac{ \pi\cosh(\pi)+\pi -2\sinh(\pi)}{ \pi\cosh(\pi)+2\pi -3\sinh(\pi)}>2.
\end{equation}
In order to see that the inequality $(*)$ --and hence \eqref{eq:bogovskii_rectangle}-- is correct, one has to discuss the function
\begin{equation*}
f:[0,\infty]\to\R,\qquad f(s)\doteq 4\sinh(s) -3s-s\cosh(s)
\end{equation*}
and to prove that $f(\pi)>0$. In fact, $f(0)=f'(0)=f''(0)=0$, $f$ is strictly convex close to $s=0$, has a unique flex point, is eventually strictly concave and $\lim_{s\to\infty}f(s)=-\infty $: hence $f$ has a unique zero in $(0,\infty)$. Since $f(3.2) > 0.05$ we have that $f|_{(0,3.2]}>0$ and, in particular, that $f(\pi)>0$.
\end{example}

\begin{remark}
Horgan-Payne \cite[(6.37)]{horgan} proved that $C_B\big((0,\pi)^2\big)\le4+2\sqrt2 \approx 6.82$ and conjectured \cite[(6.39)]{horgan} that
$C_B\big((0,\pi)^2\big)=7/2$.
Later, Costabel-Dauge \cite[(7.3)]{costabel} proved that $C_B\big((0,\pi)^2\big)\ge2\pi/(\pi-2)\approx5.5$. A numerical evaluation of
\eqref{eq:bogovskii_rectangle} only shows that $C_B\big((0,\pi)^2\big)\ge 2.0438\ldots$ But our purpose is just to show how the
non-minimality criterion in Theorem \ref{lowerbound} applies; see also the comments following Corollary~\ref{ballmin}.
\end{remark}

Let us now apply the non-minimality criterion in Theorem \ref{lowerbound} to 2D symmetric domains defined by epigraphs, with obvious
but heavy extensions to $n\ge3$.

\begin{corollary}\label{epigraphs}
Assume that
\begin{equation}\label{phiphi}
\phi\in H^2\cap H^1_0((-1,1),\mathbb{R})\mbox{ satisfies }\phi(x_1)>0\mbox{ in }(-1,1),\ \phi(-x_1)=\phi(x_1)\quad\forall|x_1|\le1 \, ,
\end{equation}
so that $\phi$ is positive and even, vanishing at the endpoints; assume moreover that
\begin{equation}\label{provided}
3\int_{0}^{1}\!x_1^2\phi(x_1)^3\phi'(x_1)^2\, dx_1<\int_{0}^{1}x_1^2\phi(x_1)^3\, dx_1\, .
\end{equation}
Then the open bounded symmetric planar domain
$$\Omega_\phi=\big\{(x_1,x_2)\in\R^2;\, -1<x_1<1,\ -\phi(x_1)<x_2<\phi(x_1)\big\}$$
satisfies $C_B(\Omega_\phi)>2$.
\end{corollary}
\noindent
\begin{proof} Take the function $w(x_1,x_2)=x_1\big(x_2^2-\phi(x_1)^2\big)$ that satisfies $w\in H^{2}\cap H^1_0(\Omega_\phi, \mathbb{R})$ and
$$
\begin{array}{c}
\partial_{x_1}w(x_1,x_2)=x_2^2-\phi(x_1)^2-2x_1\phi(x_1)\phi'(x_1)\, ,\quad
\partial_{x_2}w(x_1,x_2)=2x_1x_2\, ,\\
-\Delta w=2\Big(2\phi(x_1)\phi'(x_1)+x_1\phi'(x_1)^2+x_1\phi(x_1)\phi''(x_1)-x_1\Big)\doteq g(x_1)\mbox{ a.e.}\, ,
\end{array}
$$
so that \eqref{generalTorsion} is fulfilled. Exploiting the double
symmetry of $\Omega_\phi$, we find
{\small
\begin{align*}
\int_{\Omega_\phi}\!|\partial_{x_1}w|^2 &= 4\int\limits_{0}^{1}\!\int\limits_{0}^{\phi(x_1)}\!\Big[x_2^4\!+\!\phi(x_1)^4\!+\!4x_1^2
\phi(x_1)^2\phi'(x_1)^2\!-\!2\phi(x_1)^2x_2^2\!-\!4x_1\phi(x_1)\phi'(x_1)x_2^2\!+\!4x_1\phi(x_1)^3\phi'(x_1)\Big]dx_2dx_1\\
 &=4\int_{0}^{1}\!\Big[\frac{\phi(x_1)^5}{5}\!+\!\phi(x_1)^5\!+\!4x_1^2
\phi(x_1)^3\phi'(x_1)^2\!-\!\frac{2\phi(x_1)^5}{3}\!-\!\frac{4x_1\phi(x_1)^4\phi'(x_1)}{3}\!+\!4x_1\phi(x_1)^4\phi'(x_1)\Big]dx_1\\
\mbox{(by parts) }&=16\int_{0}^{1}\!x_1^2\phi(x_1)^3\phi'(x_1)^2\, dx_1\, ,
\end{align*}
}
$$\int_{\Omega_\phi}|\partial_{x_2}w|^2=16\int_{0}^{1}x_1^2\int_{0}^{\phi(x_1)}x_2^2\, dx_2dx_1=
\frac{16}3 \int_{0}^{1}x_1^2\phi(x_1)^3\, dx_1\, ,
$$
so that \eqref{condition} is fulfilled since \eqref{provided} holds. Hence $C_{B}(\Omega_\phi) > 2$ by Theorem \ref{lowerbound}.
\end{proof}

It is clear that \eqref{provided} is satisfied whenever $\|\phi'\|_\infty$ is sufficiently small.

\begin{example}
The function $\phi(x_1)=(1-x_1^2)/2\sqrt3$ satisfies \eqref{phiphi} and $|\phi'(x_1)|<1/\sqrt3$ for $x_1\in[0,1)$.
Hence, \eqref{provided} is fulfilled.\par\vskip-1mm\noindent
\begin{minipage}{110mm}
Therefore, by Corollary \ref{epigraphs}, the domain
$$\Omega_\phi=\big\{(x_1,x_2)\in\R^2;\, -1<x_1<1,\ \tfrac{x_1^2-1}{2\sqrt3}<x_2<\tfrac{1-x_1^2}{2\sqrt3}\big\}$$
(see the picture on the right) satisfies $C_B(\Omega_\phi)>2$.
\end{minipage}\qquad\qquad
\begin{minipage}{50mm}
\includegraphics[width=45mm]{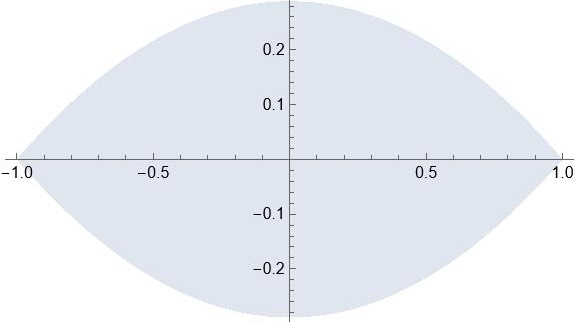}
\end{minipage}
\end{example}

In fact, the non-minimality criterion also applies to nonconvex domains.

\begin{example}
The function $\phi(x_1)=(x_1^4 + 0.1)(1-x_1^2)$ satisfies \eqref{phiphi} and fulfills \eqref{provided}, the integrands being polynomials
and computable by hand.\par\vskip-1mm\noindent
\begin{minipage}{110mm}
Therefore, by Corollary \ref{epigraphs}, the domain
$$\Omega_\phi=\big\{(x_1,x_2)\in\R^2;\, -1<x_1<1,\ -\phi(x_1)<x_2<\phi(x_1)\big\}$$
(see the picture on the right) satisfies $C_B(\Omega_\phi)>2$.
\end{minipage}\qquad\qquad
\begin{minipage}{50mm}
\includegraphics[width=45mm]{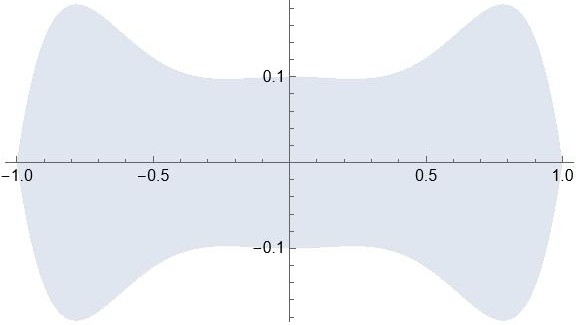}
\end{minipage}
\end{example}

\subsection{Different functional settings}\label{VEL}

While \eqref{scomp} is obtained through an orthogonal decomposition of $L^2(\Omega,\R^n)$, we consider here an orthogonal decomposition
in $H^1_0(\Omega,\R^n)$ with respect to the scalar product
\begin{equation}\label{scalarH10}
\langle u,v\rangle_{H^1_0}\doteq\int_\Omega \nabla u : \nabla v\, dx\, .
\end{equation}
Let $H^1_{0,\sigma}(\Omega,\R^n)$ be as in \eqref{H10s} and consider the extended (for $n\neq3$) Velte-Crouzeix spaces
\begin{equation}\label{spaces}
W\doteq\{u\in H^1_0(\Omega,\R^n):\ \nabla u \equiv (\nabla u)^{T} \},\quad
H\doteq(H^1_{0,\sigma}(\Omega,\R^n) \oplus W)^\perp=H^1_{0,\sigma}(\Omega,\R^n)^\perp\cap W^\perp\, ,
\end{equation}
with orthogonality intended with respect to the scalar product \eqref{scalarH10}.
Thanks to the homogeneous boundary conditions, $W$ consists of suitable gradient fields, irrespective of whether $\Omega$ is simply connected or not: when $n=3$ and $\Omega$ is simply connected, $W$ is the kernel of the curl operator \cite{velte2}.
In order to characterise $H$, we first notice that \eqref{ELeq} in Theorem \ref{minBog} states that $u_f\in H^1_{0,\sigma}(\Omega,\R^n)^\perp$.
Since \eqref{ELeq} is the weak formulation of \eqref{gstokes}, for any $p\in L^2_0 (\Omega,\R)$ there exists a unique $u\in H^1_{0,\sigma}(\Omega,\R^n)^\perp$ such that
\begin{equation}\label{eq:comm_1.2}
-\Delta u+\nabla p=0\quad \mbox{in}\quad  \Omega, \qquad u=0\quad \mbox{on}\quad \partial \Omega,
\end{equation}
which, in weak form, reads
\begin{equation}\label{eq:comm_1.1}
	\exists ! u\in H^1_{0,\sigma}(\Omega,\R^n)^\perp\subset H^1_0(\Omega,\R^n)\quad\mbox{s.t.}\quad
	\langle  u,w\rangle_{H^1_0} =\langle p,\nabla \cdot w \rangle_{L^2}\quad \forall w\in  H^1_0(\Omega,\R^n).
\end{equation}
Indeed, existence and uniqueness of $u\in H^1_0(\Omega,\R^n)$ follow from the Riesz Representation Theorem, while inserting
$w\in H^1_{0,\sigma}(\Omega,\R^n)$ yields $u\in H^1_{0,\sigma}(\Omega,\R^n)^\perp$. Conversely, as in the proof of Theorem~\ref{minBog}, one sees that for any
$u\in H^1_{0,\sigma}(\Omega,\R^n)^\perp$ there exists a unique $p\in L^2_0 (\Omega,\R)$ such that \eqref{eq:comm_1.1} is satisfied.
This establishes a bijective correspondence
$$
H^1_{0,\sigma}(\Omega,\R^n)^\perp\ni u \mapsto p_u\in  L^2_0 (\Omega,\R)
$$
via \eqref{eq:comm_1.1}. As in \cite[Theorem 1]{Velte_LNM1431_1990} we claim that
\begin{equation}\label{charH}	
H=\{u\in H^1_{0,\sigma}(\Omega,\R^n)^\perp;\ p_u \ \mbox{is (weakly) harmonic}\}.
\end{equation}

\noindent
\begin{proof} We first take $u\in H^1_{0,\sigma}(\Omega,\R^n)^\perp\cap W^\perp$ so that $\langle u,w\rangle_{H^1_0}=0$ for all $w\in H^1_{0,\sigma}(\Omega,\R^n)$ and, then, by \eqref{eq:comm_1.1}, $p_u\in L^2_0 (\Omega,\R)$ is such that
\begin{equation}\label{eq:comm_1.1bis}
\langle  u,w\rangle_{H^1_0} =\langle p_u,\nabla \cdot w \rangle_{L^2}\qquad\forall w\in  H^1_0(\Omega,\R^n).
\end{equation}
For all $\varphi\in \mathcal{C}^\infty_c(\Omega,\R)$ we have $\nabla \varphi \in W$ and, then, we infer $\langle u,\nabla \varphi\rangle_{H^1_0}=0$
after recalling that $u\in W^\perp$ by \eqref{spaces}. Therefore, \eqref{eq:comm_1.1bis} implies
\begin{equation}\label{weakharm}
\langle p_u,\Delta \varphi\rangle_{L^2}=\langle u,\nabla \varphi\rangle_{H^1_0}=0\quad\forall \varphi\in C^\infty_c(\Omega,\R)
\end{equation}
and shows that $p_u$ is (weakly) harmonic.
For the converse inclusion, assume that $p_u$ is weakly harmonic, that is, \eqref{weakharm} holds.
By density we infer that $\langle u, w\rangle_{H^1_0}=0$ for all $w\in W$, i.e., $u\in W^\perp$. Since $u\in H^1_{0,\sigma}(\Omega,\R^n)^\perp$
by assumption, $u\in H$ follows. These two-ways inclusions prove \eqref{charH}.
\end{proof}

We may now give  alternative variational characterisations of the Bogovskii constant $C_B$ in
\eqref{bogo11}, in a different functional setting.

\begin{proposition}\label{lem:comm_1.6}
Let $C_B(\Omega)$ be as in \eqref{bogo11}. Then
\begin{equation}\label{eq:comm_1.3}
C_B(\Omega)=\sup_{u\in H^1_{0,\sigma}(\Omega,\R^n)^\perp\setminus \{0\}} \frac{\|\nabla u\|_{L^2(\Omega)}^2}{\| \nabla \cdot u\|_{L^2(\Omega)}^2}\qquad\mbox{and}\qquad
C_B(\Omega)=\sup_{f\in L^2_0(\Omega,\R)\setminus \{0\}}\ \frac{\|  f\|_{L^2(\Omega)}^2}{\| \nabla f\|_{H^{-1}(\Omega)}^2} \, .
\end{equation}
\end{proposition}
\noindent
\begin{proof} By using the spaces in \eqref{spaces}, we may rephrase \eqref{ELeq} in Theorem \ref{minBog} as
$$
\mbox{for all $f\in L^2_0 (\Omega,\R)$, a solution $u\in H^1_{0}(\Omega,\R^n)$ to \eqref{Bogdom} coincides with $u_f$ if and only if }
u\in H^1_{0,\sigma}(\Omega,\R^n)^\perp.
$$
From \eqref{bogo} and this observation we deduce \eqref{eq:comm_1.3}$_1$.

Now, let us define
\begin{equation} \label{necasconstant}
C_{\mathcal{N}}(\Omega) \doteq \sup_{f\in L^2_0(\Omega,\R)\setminus \{0\}}\ \frac{\|  f\|_{L^2(\Omega)}^2}{\| \nabla f\|_{H^{-1}(\Omega)}^2} \, ,
\end{equation}
so that $C_{\mathcal{N}}(\Omega) < +\infty$, in view the Ne\v{c}as inequality \cite{Necas_ineq}. For any $f\in L^2_0 (\Omega,\R)\setminus \{0\}$, employing the orthogonal decomposition $H^1_0(\Omega,\R^n)=H^1_{0,\sigma}(\Omega,\R^n)\oplus H^1_{0,\sigma}(\Omega,\R^n)^\perp$, we find that
\begin{align*}
\| \nabla f\|_{H^{-1}(\Omega)}=&\sup_{u\in H^1_0(\Omega,\R^n)\setminus \{0\}}\
\frac{\int_\Omega f\cdot(\nabla \cdot u) \, dx}{\|\nabla u\|_{L^2(\Omega)}}
	=\sup_{u\in H^1_{0,\sigma}(\Omega,\R^n)^\perp\setminus \{0\}}\frac{\int_\Omega f\cdot(\nabla \cdot u) \, dx}{\|\nabla u\|_{L^2(\Omega)}}\\[6pt]
	=& \sup_{u\in H^1_{0,\sigma}(\Omega,\R^n)^\perp\setminus \{0\}}\frac{\int_\Omega (\nabla\cdot u_f)\cdot(\nabla \cdot u) \, dx}
{\|\nabla u\|_{L^2(\Omega)}}
	\geq \frac{\|\nabla\cdot u_f\|^2_{L^2(\Omega)}}{\|\nabla u_f\|_{L^2(\Omega)}}.
\end{align*}
We remark that in the literature here often an equality sign is put, which we think is not justified.
From this we conclude with the help of \eqref{eq:comm_1.3}$_1$:
\begin{align*}
\sup_{f\in L^2_0 (\Omega,\R)\setminus \{0\}} \ \frac{\|  f\|_{L^2(\Omega)}^2}{\| \nabla f\|_{H^{-1}(\Omega)}^2}=&
\sup_{f\in L^2_0 (\Omega,\R)\setminus \{0\}} \ \frac{\|  \nabla\cdot u_f\|_{L^2(\Omega)}^2}{\| \nabla f\|_{H^{-1}(\Omega)}^2}
	\leq \sup_{f\in L^2_0 (\Omega,\R)\setminus \{0\}} \ \frac{\|  \nabla\cdot u_f\|_{L^2(\Omega)}^2  \|\nabla u_f\|_{L^2(\Omega)}^2}{\| \nabla \cdot u_f\|_{L^2(\Omega)}^4}\\[6pt]
	=&\sup_{u\in H^1_{0,\sigma}(\Omega,\R^n)^\perp\setminus \{0\}} \ \frac{\|\nabla u\|_{L^2(\Omega)}^2}{\| \nabla \cdot u\|_{L^2(\Omega)}^2}
	=C_B(\Omega) \, ,
\end{align*}
i.e., $C_{\mathcal{N}}(\Omega) \leq C_B(\Omega)$. On the other hand, take any $f\in L^2_0 (\Omega,\R) \setminus \{0\}$. Recalling Theorem \ref{minBog}, consider the unique pair $(u_{f}, p_{f}) \in H^1_{0}(\Omega,\R^n) \times L^{2}_{0}(\Omega,\R)$ satisfying the generalised Stokes system \eqref{gstokes}, so that
\begin{equation} \label{pressureandvel}
\| \nabla p_{f} \|_{H^{-1}(\Omega)} = \| \Delta u_{f} \|_{H^{-1}(\Omega)} = \| \nabla u_{f} \|_{L^{2}(\Omega)} \, .
\end{equation}
Moreover, by testing \eqref{gstokes}$_1$ with $u_{f}$, applying the H\"older inequality and \eqref{necasconstant}, we obtain
$$
\| \nabla u_{f} \|^{2}_{L^{2}(\Omega)} = \int_{\Omega} p_{f} \, f \leq \|p_{f}\|_{L^2(\Omega)} \|f\|_{L^2(\Omega)} \leq \sqrt{C_{\mathcal{N}}(\Omega)} \, \| \nabla p_{f} \|_{H^{-1}(\Omega)} \|f\|_{L^2(\Omega)} \, ,
$$
which, due to \eqref{pressureandvel}, simplifies to
$$
\| \nabla u_{f} \|_{L^{2}(\Omega)} \leq \sqrt{C_{\mathcal{N}}(\Omega)} \, \|f\|_{L^2(\Omega)} \qquad \forall f\in L^2_0 (\Omega,\R) \setminus \{0\} \, .
$$
Owing to \eqref{bogo}, the last inequality implies $C_B(\Omega) \leq C_{\mathcal{N}}(\Omega)$, thereby proving \eqref{eq:comm_1.3}$_2$.
\end{proof}

\begin{remark}
	The first part of the previous proof gives the interpolation-type inequality
	 $$
	 \|\nabla\cdot u_f\|_{L^2(\Omega)}^2 \le 
	\| \nabla f\|_{H^{-1}(\Omega)} \|\nabla u_f\|_{L^2(\Omega)}
	\qquad\forall f\in L^2_0 (\Omega,\R),
	$$
cf. also Lemma~\ref{lem:general_lower_bound}.	
\end{remark}

Next, we define the spaces
$$
A(\Omega,\R) \doteq \{ p \in L_{0}^{2}(\Omega,\R) \, ;\, \exists \psi \in H_{0}^{2}(\Omega,\R), \ p = \Delta \psi \ \ \text{in} \ \ \Omega \, \} \, ,
$$
so that the following orthogonal decomposition holds:
\begin{equation} \label{decompl02}
L_{0}^{2}(\Omega,\R) = A(\Omega,\R) \oplus A(\Omega,\R)^{\perp} \, ,
\end{equation}
cf. \cite[Section 2]{Crouzeix_1997}. Notice that $A(\Omega,\R)^{\perp}$ is the subspace of $L_{0}^{2}(\Omega,\R)$ comprising weakly harmonic functions. Then, recall that the \textit{Schur complement of the Stokes operator} in $\Omega$ is the bounded linear
operator $\mathcal{S} : L_{0}^{2}(\Omega,\R) \longrightarrow L_{0}^{2}(\Omega,\R)$ obeying the formula
\begin{equation}\label{Schur_complement}
\mathcal{S}(q) \doteq \nabla \cdot V_{q} \qquad \forall q \in L_{0}^{2}(\Omega,\R) \, ,
\end{equation}
where $V_{q} \in H_{0}^{1}(\Omega,\R^n)$ is the unique vector field such that $\Delta V_{q} = \nabla q$ in weak sense, that is:
\begin{equation}\label{qVq}
\int_{\Omega} \nabla V_{q} : \nabla v \, dx = \int_{\Omega} q(\nabla \cdot v) \, dx \qquad \forall v \in H_{0}^{1}(\Omega,\R^{n}) \, .
\end{equation}
It can be seen, as in \cite[Theorem 11.1]{weyers2006q}, that
\begin{equation} \label{sondecomp}
\mathcal{S}(q) = q \qquad \forall q \in A(\Omega,\R) \, ; \qquad \mathcal{S}(q) \in A(\Omega,\R)^{\perp} \qquad \forall q \in A(\Omega,\R)^{\perp} \, .
\end{equation}
As proved in \cite[Theorem 1]{gaultier1996spectral}, $\mathcal{S}$ is self-adjoint, positive definite, has unit operator norm, and $1/C_{B}(\Omega)$ is the minimal spectral value of $\mathcal{S}$ by employing \eqref{eq:comm_1.3}$_2$.
More precisely, letting $\sigma(\mathcal{S})$ denote the spectrum of $\mathcal{S}$, one has
\begin{equation} \label{spectrum1}
\left\{ \dfrac{1}{C_{B}(\Omega)},1 \right\} \subset \sigma(\mathcal{S}) \subset \left[ \dfrac{1}{C_{B}(\Omega)},1 \right] \, .
\end{equation}

We now use the spaces defined in \eqref{spaces} to analyse
the so-called {\em Cosserat eigenvalue problem} that consists in finding $(u,\lambda)\in [H^1_{0}(\Omega,\R^n)\setminus \{0\}]\times\R$ such that
$$
\Delta u=\lambda \nabla (\nabla \cdot u)\quad \mbox{in}\quad \Omega,\qquad u=0\quad \mbox{on}\quad \partial\Omega,
$$
whose weak formulation reads
\begin{equation}\label{eq:comm_1.5}
\langle  u,w\rangle_{H^1_0} =\lambda \langle \nabla\cdot u,\nabla \cdot w \rangle_{L^2}\qquad \forall w\in  H^1_{0}(\Omega,\R^n).
\end{equation}
One may observe that this is the Euler-Lagrange equation for extremal functions in \eqref{eq:comm_1.3}$_1$.

Assuming $n=3$, parts of the  following result were obtained in \cite[Proposition 3]{Velte_LNM1431_1990}. With the generalised characterisation of $W$ in \eqref{spaces}, we extend the proof to any $n\ge2$ and supplement the original statement by connecting it with the Schur complement.

\begin{proposition}\label{prop:comm_1.3}
Consider the spaces defined in \eqref{spaces}. Then the eigenvalue problem \eqref{eq:comm_1.5} has the following properties:
\begin{itemize}
\item[(a)] $\lambda\in\R$ is an eigenvalue of \eqref{eq:comm_1.5} if and only if $\lambda^{-1}$ is an eigenvalue of $\mathcal{S}$ with, respectively,
associated eigenvectors $V_q\in H^1_0(\Omega,\R^n)$ and $q\in L^2_0(\Omega,\R)$ connected through \eqref{qVq};
\item[(b)] $W$ is the eigenspace for the eigenvalue $\lambda=1$ of infinite multiplicity;
\item[(c)] $H^1_{0,\sigma}(\Omega,\R^n)$ is the eigenspace for the eigenvalue $\lambda=\infty$ of infinite multiplicity;
\item[(d)] any other eigenvalue $\lambda\not\in\{1,\infty\}$ is contained in the interval $(1,C_B(\Omega)]$ and the corresponding eigenvectors belong
to $H$.
\end{itemize}
\end{proposition}\noindent
\begin{proof}
(a) Assume that $\lambda^{-1}$ is an eigenvalue of $\mathcal{S}$ with eigenvector $q\in L^2_0(\Omega,\R)$, that is, $q=\lambda\mathcal{S}(q)=\lambda(\nabla\cdot V_q)$ with $V_{q} \in H_{0}^{1}(\Omega,\R^n)$
satisfying \eqref{qVq}; then
$$
\int_{\Omega} \nabla V_{q} : \nabla v \, dx\ \stackrel{\eqref{qVq}}{=}\ \int_{\Omega} q(\nabla \cdot v) \, dx=
\lambda\int_{\Omega} (\nabla \cdot V_q)(\nabla \cdot v) \, dx \qquad \forall v \in H_{0}^{1}(\Omega,\R^n) \, ,
$$
proving that $V_q\in H^1_0(\Omega,\R^n)$ is an eigenvector of \eqref{eq:comm_1.5} with eigenvalue $\lambda$.\par
Conversely, let $\lambda\in\R$ be an eigenvalue of \eqref{eq:comm_1.5} with corresponding eigenvector $V\in H^1_0(\Omega,\R^n)$; then,
by restricting \eqref{eq:comm_1.5} to special test functions as in \eqref{ELeq}, we obtain
$$
\langle V,v\rangle_{H^1_0} =0\qquad \forall v\in  H^1_{0,\sigma}(\Omega,\R^n).
$$
This means that there exists a unique $q\in L^2_0(\Omega,\R)$ (unique because of the zero mean condition) such that
$\Delta V = \nabla q$ in weak sense in $\Omega$, that is,
$$
\int_{\Omega} \nabla V : \nabla v \, dx = \int_{\Omega} q(\nabla \cdot v) \, dx \qquad \forall v \in H_{0}^{1}(\Omega,\R^n) \, .
$$
Hence,
$$
\int_{\Omega} q(\nabla \cdot v) \, dx=\langle V,v\rangle_{H^1_0}\ \stackrel{\eqref{eq:comm_1.5}}{=}\ \lambda \langle \nabla\cdot V,\nabla \cdot v \rangle_{L^2}
\ \stackrel{\eqref{Schur_complement}}{=}\ \lambda \langle \mathcal{S}(q),\nabla \cdot v \rangle_{L^2}\qquad \forall v\in  H^1_{0}(\Omega,\R^n)
$$
which implies that
$$
0=\int_{\Omega}(q-\lambda\mathcal{S}(q))(\nabla \cdot v) \, dx=\langle\langle\nabla(\lambda\mathcal{S}(q)-q),v\rangle\rangle\qquad \forall v\in  H^1_{0}(\Omega,\R^n) \, ,
$$
where $\langle\langle\cdot,\cdot\rangle\rangle$ is the duality product between $H^{-1}(\Omega,\R^n)$ and $H^1_{0}(\Omega,\R^n)$.
By the Riesz Theorem, this shows that $\lambda\mathcal{S}(q)-q=\gamma\in\R$ a.e.\ in $\Omega$. But $q\in L^2_0(\Omega,\R)$ and the Divergence Theorem imply
$$
\gamma|\Omega|=\int_\Omega(\lambda\mathcal{S}(q)-q)dx=\lambda\int_\Omega(\nabla\cdot V)dx=0\, ,
$$
proving that $\gamma=0$ and, in turn, $\lambda\mathcal{S}(q)=q$ a.e.\ in $\Omega$. So, also the converse implication holds.
This connection between \eqref{eq:comm_1.5} and the Schur complement $\mathcal{S}$ in \eqref{Schur_complement} was not emphasised in
\cite{Velte_LNM1431_1990}. On the other hand, this statement was obtained in \cite[Theorem 1.4]{weyers2006q} in smooth domains.\par
(b) Let $u\in W$. An integration by parts shows that, for all $w\in  H^1_{0}(\Omega,\R^n)$, we have
	\begin{align*}
		\langle  u,w\rangle_{H^1_0} =&\int_\Omega \nabla u : \nabla v\, dx
		=\sum^n_{i,j=1}\int_\Omega (\partial_i u_j)\cdot(\partial_i w_j) \, dx
		\stackrel{u\in W}{=}
		\sum^n_{i,j=1}\int_\Omega (\partial_j u_i)\cdot(\partial_i w_j) \, dx\\
		=&\sum^n_{i,j=1}\int_\Omega (\partial_i u_i)\cdot(\partial_j w_j) \, dx
		=\int_\Omega (\nabla\cdot  u)\cdot(\nabla \cdot w) \, dx,
	\end{align*}
i.e., \eqref{eq:comm_1.5} is satisfied with $\lambda=1$.
Conversely, let $u\in H^1_{0}(\Omega,\R^n)$ be an eigenfunction for $\lambda=1$. Then, choosing $w=u$ in \eqref{eq:comm_1.5}, the computations from the proof of Lemma~\ref{lowBog} yield:
	\begin{align*}
	\| \nabla u\|^2_{L^2(\Omega )}	=&\langle  u,u\rangle_{H^1_0}
			\stackrel{\eqref{eq:comm_1.5}}{=}
			\int_\Omega (\nabla\cdot  u)^2\, dx =\sum^n_{i,j=1}\int_\Omega (\partial_i u_i)(\partial_j u_j) \, dx=\sum^n_{i,j=1}\int_\Omega (\partial_j u_i)(\partial_i u_j) \, dx\\
		\le&  \int_\Omega \left( \sum_{i,j=1}^n|\partial_i u_{j}|^2\right)^{1/2} \, \left(\sum_{i,j=1}^n |\partial_j u_{i}|^2\right)^{1/2}\, dx
			= \int_\Omega |\nabla u |^2 \, dx= \| \nabla u\|^2_{L^2(\Omega )},
\end{align*}
i.e., we have equality when applying the Cauchy-Schwarz inequality. This means that $\nabla u(\,.\,)$ and $\nabla u(\,.\,)^T$
are proportional and, hence, equal. This shows that $u\in W$.\par
(c)	This is obvious.\par
(d) Let $u\in H^1_{0}(\Omega,\R^n)\setminus \{0\}$ be an eigenfunction of \eqref{eq:comm_1.5}. By combining \eqref{spectrum1} with the
just proved Item (a), we infer that any eigenvalue $\lambda\not\in\{1,\infty\}$ is contained in the interval $(1,C_B(\Omega)]$,
which is {\em strictly smaller} than the interval obtained in \cite[Proposition 3]{Velte_LNM1431_1990} for the sole case $n=3$.
Next, we choose arbitrary $w\in W $ in \eqref{eq:comm_1.5}  and find after integrating by parts:
	\begin{align*}
		\langle  u,w\rangle_{H^1_0} =&\lambda \int_\Omega (\nabla\cdot  u)\cdot(\nabla \cdot w) \, dx
		=\lambda \sum^n_{i,j=1}\int_\Omega (\partial_i u_i)\cdot(\partial_j w_j) \, dx
		=\lambda 	\sum^n_{i,j=1}\int_\Omega (\partial_j u_i)\cdot(\partial_i w_j) \, dx\\
		\stackrel{w\in W}{=}&\lambda 	\sum^n_{i,j=1}\int_\Omega (\partial_j u_i)\cdot(\partial_j w_i) \, dx
		=\lambda \langle  u,w\rangle_{H^1_0}.
	\end{align*}
	Since $\lambda\not=1$ by assumption, we conclude that $u\in W^\perp$; hence, $u\in W^\perp \cap  H^1_{0,\sigma}(\Omega,\R^n)^\perp=H$.
\end{proof}

\begin{remark}\label{0infty} We point out that even if $\infty$ is an eigenvalue for \eqref{eq:comm_1.5}, $0$ is not an eigenvalue
for $\mathcal{S} : L_{0}^{2}(\Omega,\R) \longrightarrow L_{0}^{2}(\Omega,\R)$, see \eqref{spectrum1} and Proposition \ref{prop:comm_1.3}.
\end{remark}

\subsection{Attainment of the Bogovskii constant}

We show here that, at least in sufficiently smooth domains, the supremum in \eqref{bogo} is, in fact, a maximum.

While $\lambda =1$ is an eigenvalue of $\mathcal{S}$ of infinite multiplicity (due to \eqref{sondecomp}), it is not clear, whether $1/C_{B}(\Omega)$ is also an eigenvalue of $\mathcal{S}$.
To investigate this fact, we notice that Theorem \ref{minBog} and \eqref{bogo} define the bounded linear operator
$$\mathcal{B} : L_{0}^{2}(\Omega,\R) \longrightarrow H_0^1(\Omega,\R^n)\, ,\qquad \mathcal{B}(f) \doteq u_{f}\quad\forall f\in L^2_0(\Omega,\R)$$
where boundedness (and continuity) follow from \eqref{bogo11}:
$
\| \mathcal{B} \| = \sqrt{C_B(\Omega)}<\infty\, .
$ Moreover we have:

\begin{theorem}\label{bogoatttheo}
Let $\Omega\subset\R^n$ be a bounded Lipschitz domain. Suppose that $f_{*} \in A(\Omega,\R)^{\perp} \setminus \{0\}$ is an eigenfunction of the Schur complement $\mathcal{S}$, associated with the eigenvalue $1/C_{B}(\Omega)$. Then
\begin{equation}\label{bogoatt}
	C_B(\Omega) = \dfrac{\| \nabla \mathcal{B}(f_{*}) \|_{L^2(\Omega)}^2}{\| f_{*} \|_{L^2(\Omega)}^2} \, .
\end{equation}
Moreover, if $\Omega$ is of class $\mathcal{C}^{4}$ and $C_{B}(\Omega) > 2$, then there exists an eigenfunction $f_{*} \in A(\Omega,\R)^{\perp} \setminus \{0\}$ of the Schur complement $\mathcal{S}$
associated to the eigenvalue $1/C_{B}(\Omega)$. In such case, the supremum
in both \eqref{eq:comm_1.3}$_1$ and \eqref{eq:comm_1.3}$_2$ is attained, respectively, by some $u\in H^1_{0,\sigma}(\Omega,\R^n)^\perp\setminus\{0\}$ and $f\in L^2_0(\Omega,\R)\setminus\{0\}$.
\end{theorem}
\noindent
\begin{proof} Since $f_{*} \in A(\Omega,\R)^{\perp} \setminus \{0\}$ is an eigenfunction of $\mathcal{S}$ (and in particular weakly harmonic), associated with the eigenvalue $1/C_{B}(\Omega)$, by Theorem \ref{minBog} there exists a vector field $v_{*} \in H_{0}^{1}(\Omega,\R^n) \setminus \{0\}$
	and $p\in L_{0}^{2}(\Omega,\R) $ satisfying the following generalised Stokes system in $\Omega$:
	\begin{equation*}
		\left\{
		\begin{aligned}
			& - \Delta v_{*} + \nabla p= 0 \, , \quad \nabla \cdot v_{*} = f_{*} \ \ \mbox{ in } \ \ \Omega \, , \\[5pt]
			& v_{*}=0 \ \ \mbox{ on } \ \ \partial \Omega  \, .
		\end{aligned}
		\right.
	\end{equation*}
	Since
	$$
	C_{B}(\Omega) \mathcal{S} (f_*)=f_*=\nabla\cdot v_*=\nabla\cdot \left( \Delta^{-1}\nabla p\right)=\mathcal{S}(p)
	\quad \Longrightarrow \quad \mathcal{S} \big(p-C_{B}(\Omega) f_{*} \big) = 0 \, ,
	$$
	we conclude by the injectivity of $ \mathcal{S}$ that $C_{B}(\Omega)f_*=p$ and obtain
\begin{equation} \label{gstokesschur}
	\left\{
	\begin{aligned}
		& - \Delta v_{*} + C_{B}(\Omega) \nabla f_{*} = 0 \, , \quad \nabla \cdot v_{*} = f_{*} \ \ \mbox{ in } \ \ \Omega \, , \\[5pt]
		& v_{*}=0 \ \ \mbox{ on } \ \ \partial \Omega  \, .
	\end{aligned}
	\right.
\end{equation}
From \eqref{gstokesschur} we deduce that $\mathcal{B}(f_{*}) = v_{*}$, as well as
$$
\int_{\Omega} | \nabla v_{*} |^{2} \, dx = C_{B}(\Omega) \int_{\Omega} | f_{*} |^{2} \, dx \, ,
$$
thereby proving \eqref{bogoatt}.\par
If $\Omega$ is of class $\mathcal{C}^{4}$ and $C_{B}(\Omega) > 2$, from \cite[Theorem 2.1]{riedl2013cosserat} (see \cite[Section 9.2]{riedl2010cosserat} for a detailed proof) we know that
\begin{equation}\label{defS-}
\mathcal{S} - \dfrac{1}{2} I_{d} : A(\Omega,\R)^{\perp} \longrightarrow A(\Omega,\R)^{\perp}
\end{equation}
is a compact, self-adjoint and positive operator. This result is based upon \cite[Corollary 4]{Crouzeix_1997}, cf. also \cite[p. 90]{weyers2006q}. Letting $\Sigma$ denote its spectrum, from \eqref{spectrum1} we infer
\begin{equation} \label{spectrum2}
	\left\{ \dfrac{1}{C_{B}(\Omega)} - \dfrac{1}{2}, \dfrac{1}{2} \right\} \subset \Sigma\subset\left[\dfrac{1}{C_{B}(\Omega)} - \dfrac{1}{2}, \dfrac{1}{2} \right]\, .
\end{equation}
Since the operator in \eqref{defS-} is compact, every non-zero element of its spectrum is an eigenvalue. By Theorem \ref{lowerbound} we know that $C_{B}(\Omega) \geq n$; therefore,
if $n \geq 3$ or $n=2$ and $C_{B}(\Omega) > 2$, there exists $f_{*} \in A(\Omega,\R)^{\perp} \setminus \{0\}$ such that
$$
\left(\mathcal{S} - \dfrac{1}{2} I_{d} \right)(f_{*}) = \left( \dfrac{1}{C_{B}(\Omega)} - \dfrac{1}{2} \right) f_{*} \qquad \Longrightarrow \qquad \mathcal{S}(f_{*}) = \dfrac{1}{C_{B}(\Omega)} f_{*} \, .
$$
In view of the first part of the statement, this concludes the proof, with Proposition \ref{lem:comm_1.6} showing the equivalence between
the two problems.\end{proof}

\begin{remark}
Theorem \ref{lowerbound} plays a crucial role in this proof. The only case not covered by the second statement in Theorem \ref{bogoatttheo}
is for $\mathcal{C}^{4}$-domains $\Omega\subset\mathbb{R}^2$ with $C_{B}(\Omega)=2$. Nevertheless, if $\B \subset \mathbb{R}^{2}$ is a disk,
then $C_{B}(\B)=2$, see \eqref{balln} below, and the supremum in \eqref{bogo} is attained, see \cite[Theorem 3]{gaultier1996spectral}
or Proposition~\ref{thm:comm_1.4} below.
\end{remark}

Theorem \ref{bogoatttheo} enables us to improve \eqref{bogo} with a more precise variational characterisation of the Bogovskii constant whenever $\Omega\subset\R^n$, $n \geq 2$, is a bounded domain of class $\mathcal{C}^{4}$ such that $C_{B}(\Omega) > 2$:
$$C_B(\Omega) \doteq \max_{f\in L^2_0(\Omega,\R) \setminus \{0\}} \ \dfrac{\| \nabla u_{f} \|_{L^2(\Omega)}^2}{\| f \|_{L^2(\Omega)}^2} \, ,$$
where $u_f\in H^1_0(\Omega,\R^n)$ is uniquely determined by Theorem \ref{minBog}.
By linearity, the functions $f_*$ satisfying \eqref{bogoatt} form a vector space which may 
be higher dimensional and may be in the case $n=2$ and $\Omega =\B$ even infinite dimensional, see
Proposition~\ref{thm:comm_1.4} below.

\begin{problem}\label{smooth_to_Lip}
It appears challenging to extend Theorem \ref{bogoatttheo} to any Lipschitz domain, with no regularity restriction. A good starting point for the interested reader could be
a careful reading of \cite{Bernardi}.
\end{problem}

\section{Some symmetric domains}\label{symmBog}

\subsection{Balls as minimisers  of the Bogovskii constant}\label{sec:4_1}

In general bounded Lipschitz domains it appears difficult to find more refined results than those collected in Section
\ref{sec:basics}. For this reason, we consider here balls, namely, the most symmetric domains where much more can be shown. Let $\B\subset\R^n$ ($n\ge2$) be the unit ball which, by Lemma \ref{invariance}, represents {\em any ball} in $\R^n$. Although it is difficult to say who first proved it, see the Introduction, it is known that
\begin{equation}\label{balln}
	C_B(\B)=n\, .
\end{equation}
For the sake of completeness, below we give a detailed proof of \eqref{balln} by putting together several contributions
in the literature. Before doing this, several comments are in order.\par
Combined with Theorem \ref{lowerbound}, \eqref{balln} has the following  fundamental consequence.
Recall the definition \eqref{eq:def_K_i} of $\mathcal{K}_{i}$.

\begin{corollary}\label{ballmin}
	For any bounded Lipschitz domain $\Omega\subset\R^n$ we have $C_B(\Omega)\ge C_B(\B)=n$.\par
	Moreover, for any $i \in \{1,...,n\}$ and $g\in\mathcal{K}_{i}(\B) \cap L^{2}(\B,\R)$, the following implication holds:
	\begin{equation}\label{generalTorsionball}
		- \Delta w = g\mbox{ in }\B\, ,\quad w = 0\mbox{ on }\partial \B\qquad\Longrightarrow\qquad
		\int_\B \left| \dfrac{\partial w}{\partial x_{i}} \right|^2\, dx \geq \frac{1}{n}\int_\B|\nabla w|^2\, dx\, .
	\end{equation}
	
	Furthermore, for any $h\in \mathcal{C}(\R,\R)$, the following implication holds:
	\begin{equation}\label{implication2}
		-\Delta w_h=h\left(\sum_{i=1}^nx_i\right)\mbox{ in }\B\, ,\quad w_h=0\mbox{ on }\partial\B\qquad\Longrightarrow\qquad
		\sum_{1=i<j}^{n}\int_\B \partial_{x_i}w_h\, \partial_{x_j}w_h\, dx\ge0\, .
	\end{equation}
\end{corollary}

As far as we are aware, \eqref{generalTorsionball} and \eqref{implication2} are new and have their own independent interest.
On the one hand, they appear surprising and unlinked to the main focus of the present paper. On the other hand, it is
quite instructive to build an explicit example in a disc and to compare it with Example~\ref{Example_3.1} in the square $(0,\pi)^2$. The function
$$w(x)=x_1(|x|^2-1)\qquad x\in\B\subset\R^2$$
satisfies
$$
\Delta w=8x_1\quad\mbox{in }\B\, ,\qquad w=0\quad\mbox{on }\partial\B\, ;
$$
some computations then show that
$$
\int_{\B}|\partial_{x_1}w|^2\, dx=\frac{\pi}{2}>\frac{\pi}{6}=\int_{\B}|\partial_{x_2}w|^2\, dx\, ,
$$
in line with \eqref{generalTorsionball} and contrary to \eqref{square}.

\begin{problem}\label{shapeopt}
Corollary \ref{ballmin} states that balls minimise $C_B$ in \eqref{eq:comm_1.3} among Lipschitz domains. A natural question is whether balls are the only minimisers which is in our opinion the most intuitive conjecture. This is a degenerate shape optimisation
problem due to scaling invariance (see Lemma \ref{invariance}) and several tools need to be set up. From \cite[Theorem 4.4]{Bernardi} we know that the dependence
of $C_B(\cdot)$ with respect to Lipschitz convergence of domains is Lipschitzian, which gives some hope to prove that $\B$ is a {\em local} minimiser of $C_B$ through shape variation.
\end{problem}

 However, \eqref{balln} may be obtained
{\em without using symmetrisation techniques}, as we now show.\par\medskip

\noindent
\emph{Proof of \eqref{balln}.} Let us first summarise the findings of Cosserat-Cosserat \cite{Cosserat_18987,Cosserat_18986} (see also \cite{Cosserat_1896}) concerning the eigenvalue problem \eqref{eq:comm_1.5} when $\Omega=\B$, along the lines of \cite{Simader_vonWahl_2006}.
To this end, we recall some basic facts about spherical harmonics from \cite[Chapter 5, Section 5]{Sauvigny_PDE_1}. For $k\in \mathbb{N}_0$ we consider
$$
{\mathcal M}_k\doteq\{f:\mathbb{S}^{n-1}\to \R;\ f\mbox{ is a spherical harmonic of degree }k\},
$$	
so that $f\in{\mathcal M}_k$ if and only if $x\mapsto |x|^k f(x/|x|)$ is harmonic in $\mathbb{R}^n$. We know from \cite[Theorem 5.5]{Sauvigny_PDE_1} that these spaces are finite dimensional
and that their dimensions
$$
N(k,n)\doteq\operatorname{dim}{\mathcal M}_k <\infty
$$
are given by the coefficients of the generating analytic function
$$
(-1,1)\ni t\mapsto \frac{1+t}{(1-t)^{n-1}}=\sum^\infty_{k=0}N(k,n)t^k.
$$	
In particular we have $N(0,n)=1$ and $N(1,n)=n$.
We introduce the orthonormal bases
\begin{equation}\label{basis}
	\{H_{k,1},\ldots,H_{k,N(k,n)}\}
\end{equation}
of ${\mathcal M}_k$ with respect to the canonical scalar product in $L^2(\mathbb{S}^{n-1})$. From \cite[Theorem 5.6]{Sauvigny_PDE_1} we know that
$$
\bigcup_{k=0}^\infty \{H_{k,1},\ldots,H_{k,N(k,n)}\}
$$
forms a complete orthonormal system in $L^2(\mathbb{S}^{n-1})$.\par
Then we consider the space of homogeneous harmonic polynomials  of degree $k$
$$
\widetilde{\mathcal M}_k\doteq\{x\mapsto |x|^k f(x/|x|);\ f\in {\mathcal M}_k\},\quad k\in \N_0\, .
$$
Let $H_k$ denote a nontrivial element of $\widetilde{\mathcal M}_k$, so that
$$
H_k(x) = \sum_{j=1}^{N(k,n)} \alpha_{k,j} |x|^k H_{k,j}(x/|x|)
$$
for suitable $\alpha_{k,j}\in\R$, where at least one is $\not=0$. For all $k\in \N$ we then define
\begin{equation}\label{eq:comm_1.6}
	U_k\in H^1_0 (\B,\R^n),\qquad U_k(x)\doteq(|x|^2-1)\nabla H_k(x)
\end{equation}
which will turn out to be the generic eigenfunction of \eqref{eq:comm_1.5} corresponding to the eigenvalue
\begin{equation}\label{eq:comm_1.7}
	\lambda_k =2+\frac{n-2}{k}\quad (k\in\N)\quad \mbox{of multiplicity}\quad N(k,n).
\end{equation}
One may observe that all eigenvalues coincide if $n=2$. Using the well-known Euler formula for $k$-homogeneous functions $x\cdot \nabla H_k (x) = kH_k(x)$, we compute:
$$
\begin{array}{cc}
	\nabla\cdot U_k(x)=2k H_k(x)\, ,\qquad\nabla \left( \nabla \cdot U_k \right)(x) = 2k \nabla H_k(x)\\
	\Delta U_k(x)=2n \nabla H_k(x) +4(x\cdot \nabla) \nabla H_k(x)
	=2n \nabla H_k(x) +4(k-1) \nabla H_k(x)=\lambda_k \nabla \left( \nabla \cdot U_k \right)(x)\, .
\end{array}
$$
In particular, \eqref{eq:comm_1.7} follows from these identities and the following result,
which generalises \cite[Theorem 2]{Velte_LNM1431_1990} to any dimension.

\begin{proposition}\label{thm:comm_1.4}
	Taking orthonormal bases $\left( H_{k,j}\right)_{k\in \N ,\ j=1,\ldots,N(k,n)}$ of $\widetilde{\mathcal M}_k$ with respect to the
	$L^2$-scalar product, the corresponding system of eigenfunctions $\left( U_{k,j}\right)_{k\in \N ,\ j=1,\ldots,N(k,n)}$  is complete in $H$ with respect to the $H^1_0$-scalar product \eqref{scalarH10}.
	Hence, $C_B(\B)=\lambda_1=n$ and, if $n\ge 3$,
	$C_B(\B)$ is attained on $U_1(x)=(|x|^2-1)\nabla H_1(x)$ with $H_1$ and $\nabla \cdot U_1$ being any linear function on $\R^n$. If $n=2$, $C_B(\B)$ is attained on the whole space $H$, i.e. on any $\tilde U(x)=(|x|^2-1)\nabla \tilde H(x)$ with any harmonic polynomial $\tilde H\in H$.
\end{proposition}
\noindent
\begin{proof}
	We first prove  orthogonality of the system $\left( U_{k,j}\right)_{k\in \N ,\ j=1,\ldots,N(k,n)}$. Let $k\not=\ell$ or $i\not=j$ so that
	$$
	\int_\B H_{k,i} H_{\ell,j}\, dx =0.
	$$
	It follows that
	\begin{align*}
		\langle U_{k,i} , U_{\ell,j} \rangle_{H^1_0(\B)}=&-\int_\B U_{k,i} (\Delta U_{\ell,j})\, dx =-2\ell\lambda_\ell \int_\B U_{k,i} (\nabla H_{\ell,j})\, dx\\
		=& -2\ell\lambda_\ell \int_\B (|x|^2-1)(\nabla H_{k,i})\cdot (\nabla H_{\ell,j})\, dx
		=4 \ell\lambda_\ell \int_\B (x\cdot \nabla H_{k,i})\cdot H_{\ell,j}\, dx\\
		=& 4 k\ell\lambda_\ell \int_\B  H_{k,i}\cdot H_{\ell,j}\, dx=0.
	\end{align*}
	
	Next, we show that, for all $k\in \N$ and $j=1,\ldots,N(k,n)$, we have $U_{k,j}\in H$. To this end, we first observe that $U_{k,j}\in H^1_{0,\sigma}(\B,\R^n)^\perp$ by \eqref{eq:comm_1.5}. Then we recall that $\Delta U_{k,j}=2k\lambda_k \nabla H_{k,j}$ and  \eqref{charH} yields $U_{k,j}\in H$ because $\Delta H_{k,j}=0$ in $\B$.\par
	We now turn to proving the completeness. Take any $v\in H$ such that
	$$
	\forall k\in \N ,\ \forall j=1,\ldots,N(k,n):\quad \langle U_{k,j}, v\rangle_{H^1_0(\B)} =0.
	$$
	We need to prove that $v=0$. Indeed, let $p_v\in L^2_0(\B,\R)$ correspond to $v$ according to \eqref{eq:comm_1.2}.
	Then \eqref{charH}   shows that $\Delta p_v=0$ in $\B$.
	It follows that
	$$
	\forall k\in \N ,\ \forall j=1,\ldots,N(k,n):\quad 0=\langle U_{k,j}, v\rangle_{H^1_0(\B)} =\langle \nabla \cdot U_{k,j}, p_v\rangle_{L^2(\B)} =2k\langle H_{k,j}, p_v\rangle_{L^2(\B)}\, .
	$$
	Since $p_v$ is harmonic, it follows that $p_v=0$; see \cite[Theorem 5.5.III and Proposition 5.5.7]{Sauvigny_PDE_1}.
	This finally yields that $v=0$.
	
	Combining the previous reasoning with Proposition \ref{prop:comm_1.3} and \eqref{eq:comm_1.3}, we see
	in the case $n>2$ that $C_B(\B)$ is the largest eigenvalue $\lambda_1=n$ of \eqref{eq:comm_1.5} on $H^1_{0,\sigma}(\B,\R^n)^\perp=W\oplus H$ of multiplicity $N(1,n)=n$ and that $C_B(\B)$ is attained precisely
on any $U_1$.
If $n=2$, $C_B(\B)$ is attained on the whole infinite dimensional space $H$, since all Cosserat eigenvalues are there equal to $2$.
\end{proof}

\subsection{Ellipsoids}\label{sec:ellipsoids}

From P{\'o}lya \& Szeg{\"o} \cite[Section 1.29]{polya1953isoperimetric}, we learn that a method due to Lord Rayleigh and Ritz consists in modifying an original
variational problem by restricting the space of admissible functions. One possibility is to decide {\em a priori} the shape of their level surfaces. The disadvantage of this approach is that it only gives
an upper bound of the defined variational constant. We focus our attention on {\em similar surfaces} \cite[Section 3.5]{polya1953isoperimetric} that, when $\Omega$ is a ball, restricts the problem to
radial functions.

\begin{theorem}\label{ellipsoid}
	Let $a_j>0$ for $j=1,...,n$ and consider the ellipsoid
	$$
	E\doteq\left\{(x_1,...,x_n)\in\R^n;\ \sum_{j=1}^na_j^2x_j^2<1\right\}.
	$$
	The problem \eqref{Bogdom} admits a solution $u\in H^1_0(E,\R^n)$ of the form
	$$
	u(x)=u\left(\sqrt{\sum_{j=1}^na_j^2x_j^2}\right)
	$$
	if and only if $f\in L^2_0(E,\R)$ has the form
	\begin{equation}\label{formf}
	f(x)=\sum_{i=1}^n f_i\left(\sqrt{\sum_{j=1}^na_j^2x_j^2}\right)x_i\quad\mbox{ for a.e. }x\in E.
	\end{equation}
	In such case, the solution $u$ of \eqref{Bogdom} is unique in this class and its components are given by
	$$
	u_i\left(\sqrt{\sum_{j=1}^na_j^2x_j^2}\right)=-\int_{\sqrt{\sum_{j=1}^na_j^2x_j^2}}^1sf_i(s)\, ds \qquad \forall x \in E,
	\ \ i \in \{1,...,n\}\, .
	$$
	Moreover, this solution of \eqref{Bogdom} coincides with $u_f$, given by Theorem \ref{minBog}, provided that $f\not=0$, $f(x)=\sum^n_{i=1}\gamma_i x_i$, $\gamma_i\in \R$.\par
	Concerning the Bogovskii constant, we have
	\begin{equation}\label{CBE}
	\left( \frac{\max_ia_i}{\min_ia_i}\right)^2<\frac{1}{\min_ia_i^2}\, \sum_{j=1}^{n}a_j^2\le C_B(E)\le n\left(\frac{\max_ia_i}{\min_ia_i}\right)^{2}\, .
	\end{equation}
\end{theorem}
\noindent
\begin{proof} For $x \in E$, let
	$$r=r(x)=\sqrt{\sum_{j=1}^na_j^2x_j^2}\, .$$
	Assume that $u=u(r)=\big(u_1(r),...,u_n(r)\big)\in H^1_0(E,\R^n)$, then
	$$
	\nabla\cdot u=\sum_{i=1}^{n}\frac{u_i'(r)}{r}\, a_i^2x_i \quad \mbox{in} \quad E\, .
	$$
	This shows that if $u=u(r)$ solves \eqref{Bogdom} for some $f\in L^2_0(E,\R)$, then $f$ has the form \eqref{formf} with
	\begin{equation}\label{explicitui2}
	f_i(r) = \frac{a_i^2 u_i'(r)}{r} \quad \Longrightarrow \quad u_i(r)=-\frac{1}{a_i^2}\int_{r}^1sf_i(s)\, ds \qquad \forall r \in (0,1],
	\ \ i\in\{1,...,n \},
	\end{equation}
	since $u_i(1)=0$. Conversely, if $f\in L^2_0(E,\R)$ has the form \eqref{formf} for some $f_1,...,f_n$, then the vector field
	\begin{equation} \label{solradial2}
	u(r) \doteq \left( -\frac{1}{a_1^2}\int_{r}^1sf_1(s) \, ds, ..., -\frac{1}{a_n^2}\int_{r}^1sf_n(s) \, ds  \right)^T \qquad
	\forall r \in(0,1],
	\end{equation}
	belongs to $H^1_0(E,\R^n)$ and solves \eqref{Bogdom}. In particular, \eqref{explicitui2} proves the uniqueness part of the statement.\par
	If $f(x)=x_i$ for some $i$ (that is, $f_i(r)\equiv1$ and $f_j(r)\equiv0$ for $j\neq i$), then \eqref{explicitui2} gives
	\begin{equation}\label{uifi}
	u_i(r)=\frac{\sum_{j=1}^{n}a_j^2x_j^2-1}{2a_i^2}\, ,\quad u_j(r)\equiv0\mbox{ for }j\neq i\, ,\qquad \forall r \in (0,1]\, ,
	\end{equation}
	which imply that all $\Delta u_i$ are constants and, hence, $\partial_{x_j}\Delta u_i=0$ for any $i,j$ and $u=u_f$. By linearity of
\eqref{Bogdom} and \eqref{ELeq} the same holds for any $f\not=0$ of the form $f(x)=\sum^n_{i=1}\gamma_i x_i$, $\gamma_i\in \R$.

	In order to prove the bounds for $C_B(E)$, we begin with showing the lower bound. This is obtained by taking $f(x)=x_i$ for some $i\in\{1,\ldots,n\}$ and the resulting solution
	\eqref{uifi}. The change of variables $y_j=a_jx_j$ for all $j$, gives the Jacobian $J(E)=\prod_ja_j$ and we obtain
	\begin{align*}
		\int_Ef(x)^2\, dx &= \int_Ex_i^2\, dx=\int_E\frac{a_i^2x_i^2}{a_i^2}\, dx=\frac{1}{a_i^2\, J(E)}\int_\B y_i^2\, dy=
		\frac{1}{na_i^2\, J(E)}\int_\B|y|^2\, dy\\
		&= \frac{\omega_n }{na_i^2\, J(E)}\int_0^1\rho^{n+1}\, d\rho=\frac{\omega_n }{n(n+2)a_i^2\, J(E)}\, ,
	\end{align*}
		where
	$$
	\omega_n  \doteq \dfrac{2 \pi^{n/2}}{\Gamma \left( \dfrac{n}{2} \right)}
	$$
	is the $(n-1)$-dimensional measure of the  unit sphere $\mathbb{S}^{n-1}=\partial \B \subset \R^n$. 
	On the other hand, for the corresponding solution $u$ in \eqref{uifi}, we have
	$$
	\partial_{x_j}u_i(x)=\frac{a_j^2\, x_j}{a_i^2}\mbox{ for }j=1,\ldots,n\ \Longrightarrow\ |\nabla u(x)|^2=
	\frac{1}{a_i^4}\sum_{j=1}^{n}a_j^4x_j^2\, .
	$$
	Hence, through the same change of variables as before,
	$$
	\int_E|\nabla u(x)|^2\, dx =\frac{1}{a_i^4}\sum_{j=1}^{n}a_j^2\int_Ea_j^2x_j^2\, dx
	= \frac{\omega_n }{n(n+2)J(E)\, a_i^4}\, \sum_{j=1}^{n}a_j^2.
	$$
	Therefore,
	\begin{equation}\label{bogratioaffine}
\frac{\int_E|\nabla u|^2\, dx}{\int_E f^2\, dx}=\frac{1}{a_i^2}\, \sum_{j=1}^{n}a_j^2\qquad\forall i \in \{1,...,n\}\, .	
	\end{equation}
	By choosing as index $i$ the one corresponding to the minimum $a_i$, we obtain
	\begin{equation}\label{eq:ellipse_2}
	C_B(E)\ge\frac{1}{\min_ia_i^2}\, \sum_{j=1}^{n}a_j^2\, ,
	\end{equation}
	that is, the stated lower bound.\par
For the upper bound, we make use of $C_B(\B)=n$ and the homothetic change of variables as before. This (and only this) will permit to directly transform also problem \eqref{Bogdom}.
	Take any
	$f\in L^2_0 (E,\R)$ and transform it to the unit ball:
	$$
	\widetilde f\in L^2_0 (\B,\R), \quad \widetilde f(x)\doteq f\left( \frac{x_1}{a_1},\ldots,\frac{x_n}{a_n}\right).
	$$
	We then consider the privileged solution $u_{\widetilde f}\in H^1_0(\B,\R^n)$
	of \eqref{Bogdom}, see Theorem~\ref{minBog}, and \eqref{balln} shows that
	\begin{equation}\label{eq:ellipse_1}
	\| \nabla u_{\widetilde f} \|^2_{L^2(\B)}\le n \| {\widetilde f} \|^2_{L^2(\B)}.
	\end{equation}
	We pull $ u_{\widetilde f}$ back to the ellipsoid by defining
	$$
	u\in H^1_0(E,\R^n), \quad u(x)\doteq \begin{pmatrix}
	1/a_1 & \dots & 0\\
	\vdots & \ddots & \vdots\\
	0 & \dots & 1/a_n
	\end{pmatrix}\cdot u_{\widetilde f} (a_1 x_1,\ldots,a_n x_n)
	$$
	which solves
	$$
	\nabla \cdot u =f\quad \text{in}\quad E,\qquad  u=0 \quad \text{on}\quad \partial E.
	$$
	We observe for $j=1,\ldots,n$ that
	$$
	\nabla u_j(x)= \left(\frac{a_1}{a_j} \frac{\partial u_{\widetilde f,j}}{\partial x_1}(a_1 x_1,\ldots,a_n x_n), \ldots, \frac{a_n}{a_j} \frac{\partial u_{\widetilde f,j}}{\partial x_n}(a_1 x_1,\ldots,a_n x_n)\right) .
	$$
	By the usual change of variables we obtain from \eqref{eq:ellipse_1},
	using $J(E)=\prod_{j=1}^n a_j$ as above
	\begin{align*}
	\| \nabla u \|^2_{L^2(E)}\le& \left( \max_{i,j=1,\ldots,n}\frac{a_i}{a_j}\right)^2
	\int_E |\nabla u_{\widetilde f} (a_1 x_1,\ldots,a_n x_n)|^2\, dx \\
	=&\frac{1}{J(E)}\left( \max_{i,j=1,\ldots,n}\frac{a_i}{a_j}\right)^2
	\| \nabla u_{\widetilde f} \|^2_{L^2(\B)}
	\le n\cdot \frac{1}{J(E)}\left( \max_{i,j=1,\ldots,n}\frac{a_i}{a_j}\right)^2 \| {\widetilde f} \|^2_{L^2(\B)}\\
	=& n \left( \frac{\max_{i=1,\ldots,n}a_i}{\min_{i=1,\ldots,n}a_i} \right)^2 \| { f} \|^2_{L^2(E)}.
	\end{align*}
	Passing to the respective privileged $u_f$ (observe that this will in general be different from $u$) according to Theorem~\ref{minBog} yields
	$$
	\| \nabla u_f \|^2_{L^2(E)}\le n \left( \frac{\max_{i=1,\ldots,n}a_i}{\min_{i=1,\ldots,n}a_i} \right)^2 \| { f} \|^2_{L^2(E)}.
	$$
	Since this holds for any $f\in L^2_0 (E,\R)$ we end up with
	$$
	C_B(E)\le  n \left( \frac{\max_{i=1,\ldots,n}a_i}{\min_{i=1,\ldots,n}a_i} \right)^2.
	$$
	Together with \eqref{eq:ellipse_2}, the claim \eqref{CBE} follows.
\end{proof}

\begin{problem}\label{star}
In order to obtain an upper bound for $C_B(E)$, one could use \cite[(III.3.4)]{galdi2011introduction} and would end up with
$$
C_B(E)\le c(n)\left(\frac{\max_ia_i}{\min_ia_i}\right)^{n+1}\, ,
$$
for some $c(n)>0$ depending only on the dimension $n$, see also \cite[Theorem III.3.1]{galdi2011introduction} for more general domains.
However, we would not obtain the right power of $\frac{\max_ia_i}{\min_ia_i}$, due to the generality of this bound.
In our opinion, the main strength of \eqref{CBE} is that the power is independent of the dimension $n$. We then wonder whether the same holds,
at least for some special domains (e.g., domains starshaped with respect to a ball or the more restricted class of convex domains).
\end{problem}

For $n=2$, it is known \cite[Section 5.1.1]{costabel} that
$$
C_B(E)=\frac{a_1^2+a_2^2}{\min\{a_1^2,a_2^2\}}\, .
$$
This suggests not only that the inequality in \eqref{eq:ellipse_2} may in fact be  an equality in any dimension $n\ge2$, but also yields the following result,
related to Theorem \ref{bogoatttheo}.

\begin{corollary}
	Let $0<a_1\le a_2$ and consider the (planar) ellipse
	$$
	E\doteq\left\{(x_1,x_2)\in\R^2;\ a_1^2x_1^2+a_2^2x_2^2<1\right\}.
	$$
	Then the supremum in \eqref{bogo} is a maximum and it is attained by real multiples of the function
	$$f(x)= x_1.$$
\end{corollary}

\begin{remark}\label{rem:thin_rectangles}
The lower bound in \eqref{CBE} shows that a sequence $E_m$ of ellipsoids with axes $a_i^m$ satisfying $\max_ia_i^m/\min_ia_i^m\to\infty$,
is such that $C_B(E_m)\to\infty$. This suggests that ``thin domains'' $\Omega\subset\R^n$ (with small inradius and large outradius) might have
a large Bogovskii constant $C_B(\Omega)$.
However, in ``thin ellipses''  \eqref{bogratioaffine} also shows (by choosing $a_i$ maximal) that for suitable pairs $(f,u_f)$ one may still have ``Bogovskii ratios'' close to $1$.\par
The function $w(x_1,x_2)=\sin(x_1)\left(1-\frac{\sinh(x_2)+\sinh(h-x_2)}{\sinh(h)}\right)$
is the unique solution to
$$
-\Delta w=\sin(x_1)\quad\mbox{in }(0,\pi)\times(0,h)\, ,\qquad w=0\quad\mbox{on }\partial\Big((0,\pi)\times(0,h)\Big)\, .
$$
By arguing as for deducing \eqref{eq:bogovskii_rectangle} in Example \ref{Example_3.1}, we find for $h\in (0,\pi)$ that

\begin{equation}\label{eq:asymptotics_rectangle}
C_B\big((0,\pi)\times (0,h)\big)\ge
\frac{ h\cosh(h)+h -2\sinh(h)}{ h\cosh(h)+2h -3\sinh(h)}
\stackrel{h\to0}{=}\frac{10}{h^2}(1+O(h^2))\to\infty,
\end{equation}
thereby supporting the conjecture that the Bogovskii constant of ``thinning'' domains tends to infinity.
Chizhonkov-Olshanskii  prove in \cite[Section 2, Equation
(30)]{Chizhonkov_Olshanskii_2000}:
\begin{equation*}
\frac{12}{h^2}\le C_B\big((0,\pi)\times (0,h)\big)\le \frac{60\pi^2}{h^2}.
\end{equation*}
These bounds were obtained by means of the Schur complement of the Stokes operator \eqref{Schur_complement} and the authors themselves call their approach ``rather technical''.
Our lower bound \eqref{eq:asymptotics_rectangle} is derived in a more elementary way: it is not as precise but it has the same order of magnitude.
\end{remark}

\subsection{Annuli}\label{sec:annuli}

Given $0 < r < R$, we consider the annular region $\mathcal{A}_r^R=\{x\in\R^n;\, r<|x|<R\}$.
Due to its multiple symmetries, we cannot use the torsion problem \eqref{torsionpb} in order to show that $C_B\big(\mathcal{A}_r^R\big)>n$, namely that an annulus does not achieve the minimal Bogovskii constant.
Also finding functions $g$ or $h$ satisfying \eqref{generalTorsion} or \eqref{hsum} appears to be a difficult task. This is why we follow a different procedure. Since the planar annulus ($n=2$) is the only non-simply connected $n$-dimensional annulus and since in 2D we may exploit the celebrated
Friedrichs inequality \cite{friedrichs1937certain}, we give several variants of the proof. We end up with the following statement.

\begin{theorem}\label{thm:comm_4.1}
Let $n\ge2$ and $\mathcal{A}_r^R=\{x\in\R^n;\, r<|x|<R\}$. For all $0<r<R<\infty$ we have
\begin{equation}\label{eq:comm_4.3}
C_B(\mathcal{A}_r^R)>n\, .
\end{equation}
Moreover, $\displaystyle\lim_{R\downarrow r}C_B(\mathcal{A}_r^R)=\infty$ with the following lower bound
\begin{equation}\label{asymptoticannulus}
\liminf_{R\downarrow r}\left(\frac{R}{r}-1\right)^2\, C_B(\mathcal{A}_r^R)\ge\frac{3}{n-1}\, .
\end{equation}	
\end{theorem}
\noindent
\begin{proof} We give several different proofs, depending on the dimension $n$ and on the tools that we use.\par
\underline{First proof of \eqref{eq:comm_4.3} when $n=2$}. From \cite{friedrichs1937certain} we know that there exists
$C_{F}(\mathcal{A}_r^R) > 0$ such that
$$
\int_{\mathcal{A}_r^R} |u|^2 \leq C_{F}(\mathcal{A}_r^R) \int_{\mathcal{A}_r^R} |v|^2 \, ,
$$
for every pair of conjugate harmonic functions $u, v \in L^2_0(\mathcal{A}_r^R,\R)$. Moreover, from \cite[Theorem 2.1]{costabel} we know that
\begin{equation}\label{fried}
C_{F}(\mathcal{A}_r^R) = C_{B}(\mathcal{A}_r^R) - 1 \, .
\end{equation}
Hence, \eqref{eq:comm_4.3} follows if we prove that $C_{F}(\mathcal{A}_r^R)>1$.
Since from \eqref{fried} and Theorem \ref{lowerbound} we know that $C_{F}(\mathcal{A}_r^R)\geq1$, we can argue by contradiction assuming that $C_{F}(\mathcal{A}_r^R) = 1$, namely
\begin{equation} \label{contra}
\int_{\mathcal{A}_r^R} |u|^2 \leq \int_{\mathcal{A}_r^R} |v|^2 \, ,
\end{equation}
for every pair of conjugate harmonic functions $u, v \in L^2_0(\mathcal{A}_r^R,\R)$. Considering the holomorphic function
\begin{equation}\label{zhuk}
\Phi(z) \doteq z + \dfrac{1}{z} \qquad \forall z=(x_1,x_2)\in \overline{\mathcal{A}_r^R} \, ,\qquad\Phi=f+ig\, ,
\end{equation}
we deduce that $f,g\in L^2_0(\mathcal{A}_r^R,\R)$ defined by
\begin{equation}\label{realimaginary}
f(x)=\left( 1 + \dfrac{1}{|x|^2} \right) x_1 \qquad \text{and} \qquad g(x)=\left( 1 - \dfrac{1}{|x|^2} \right) x_2 \qquad \forall (x_1,x_2) \in \overline{\mathcal{A}_r^R} \, ,
\end{equation}
are conjugate harmonic functions. Then, a straightforward computation yields
\begin{equation}\label{eq:comm_2.1}\\
\int_{\mathcal{A}_r^R} \left( |f|^2 - |g|^2 \right) = 2 \pi (R^2 - r^2) > 0 \, ,
\end{equation}
which contradicts \eqref{contra} and proves \eqref{eq:comm_4.3} when $n=2$.\par
\underline{Second proof of \eqref{eq:comm_4.3} when $n=2$}. We give here an alternative proof of \eqref{eq:comm_4.3} {\em without}
using explicitly the results by Friedrichs \cite{friedrichs1937certain}. It has the advantage of introducing the tools to prove \eqref{asymptoticannulus},
see the next item. We consider again the real and imaginary parts \eqref{realimaginary} of \eqref{zhuk}
and the corresponding privileged solution $u_f \in H^1_{0,\sigma}(\mathcal{A}_r^R,\R^2)^\perp$ of
$$
\nabla \cdot u_f=f \quad \mbox{in}\quad \mathcal{A}_r^R, \qquad u_f=0
 \quad \mbox{on}\quad \partial \mathcal{A}_r^R.
$$
We compute
\begin{align}
\| g\|^2_{L^2(\mathcal{A}_r^R)}=&\int_{\mathcal{A}_r^R} \left(1-\frac{1}{|x|^2} \right)^2x_2^2\, dx
=\frac{1}{2}\int_{\mathcal{A}_r^R} \left(1-\frac{1}{|x|^2} \right)^2|x|^2\, dx\nonumber \\
=&\pi \int^R_r \left(\rho^3-2\rho+\frac{1}{\rho}\right)\, d\rho
=\frac{\pi}{4}(R^2-r^2)\left( R^2+r^2-4+4\frac{\log (R/r)}{R^2-r^2}\right)\label{eq:comm_2.2}
\end{align}
and, by making use of the Cauchy-Riemann differential equations for $f+ig$ and the Cauchy-Schwarz inequality, we estimate
\begin{align*}
\| f\|^2_{L^2(\mathcal{A}_r^R)}=& \int_{\mathcal{A}_r^R} f \cdot (\nabla \cdot u_f)\, dx
=-\int_{\mathcal{A}_r^R} \nabla f \cdot   u_f\, dx
=-\int_{\mathcal{A}_r^R} (\partial_2 g, -\partial_1 g)^T \cdot   u_f\, dx\\
=&\int_{\mathcal{A}_r^R} g(\partial_2 u_{f,1} -\partial_1 u_{f,2}) \, dx
\le \| g\|_{L^2(\mathcal{A}_r^R)} \, \| \partial_2 u_{f,1} -\partial_1 u_{f,2} \|_{L^2(\mathcal{A}_r^R)}.
\end{align*}
For the second factor, we calculate:
\begin{align*}
\| \partial_2 u_{f,1} -\partial_1 u_{f,2} \|^2_{L^2(\mathcal{A}_r^R)} =&
\int_{\mathcal{A}_r^R} \left( (\partial_2 u_{f,1})^2+(\partial_1 u_{f,2})^2-2(\partial_2 u_{f,1}\cdot \partial_1 u_{f,2})\right)\, dx\\
=&\int_{\mathcal{A}_r^R} | \nabla   u_f|^2\, dx
-\int_{\mathcal{A}_r^R} \left( (\partial_1 u_{f,1})^2+(\partial_2 u_{f,2})^2+2(\partial_1 u_{f,1}\cdot \partial_2 u_{f,2})\right)\, dx\\
=&\int_{\mathcal{A}_r^R} | \nabla   u_f|^2\, dx
-\int_{\mathcal{A}_r^R} \left(\nabla \cdot u_f\right)^2\, dx
\le (C_B(\mathcal{A}_r^R)-1)\| f\|_{L^2(\mathcal{A}_r^R)}^2.
\end{align*}
Combining the previous estimates yields:
$$
\| f\|^2_{L^2(\mathcal{A}_r^R)}\le \sqrt{(C_B(\mathcal{A}_r^R)-1)}\, \| f\|_{L^2(\mathcal{A}_r^R)}\| g\|_{L^2(\mathcal{A}_r^R)}\quad\Longrightarrow\quad \| f\|^2_{L^2(\mathcal{A}_r^R)}\le (C_B(\mathcal{A}_r^R)-1)\| g\|^2_{L^2(\mathcal{A}_r^R)}.
$$
Taking \eqref{eq:comm_2.1} into account we come up with
\begin{equation}\label{lowerbound2}
C_B(\mathcal{A}_r^R)\ge 2 +\frac{2\pi(R^2-r^2)}{\| g\|^2_{L^2(\mathcal{A}_r^R)}}
\end{equation}
which proves (again) \eqref{eq:comm_4.3} when $n=2$.\par
\underline{Proof of \eqref{asymptoticannulus} when $n=2$}. By combining \eqref{eq:comm_2.2} with \eqref{lowerbound2} yields
$$
C_B(\mathcal{A}_r^R)\ge 2 +\frac{8}{ R^2+r^2-4+4\frac{\log (R/r)}{R^2-r^2} }.
$$
While the left hand side is scaling invariant, the right hand side is not. Apparently, any choice of $R>r>0$ seems equivalent, but this is not
the case. The only choice being able to emphasise the singular behaviour of $C_B(\mathcal{A}_1^R)$ is fixing $r=1$; then, we obtain
\begin{equation}\label{eq:comm_2.4}
C_B(\mathcal{A}_1^R)\ge 2 +\frac{8}{ R^2-3+4\frac{\log (R)}{R^2-1} }.
\end{equation}
Making use of the inequality $\log(1+x)\le x-\tfrac{x^2}{2}+\tfrac{x^3}{3}$ (for all $x\ge0$), we then find
$$
C_B(\mathcal{A}_1^R)\ge 2 +\frac{8}{ R^2-3+2\frac{\log (1+(R^2-1))}{R^2-1} }\ge 2+\frac{12}{(R^2-1)^2}.
$$
At this point, we may use scaling invariance (Lemma \ref{invariance}) to obtain for any $0<r<R$:
\begin{equation}\label{eq:comm_2.3}
C_B(\mathcal{A}_r^R)=C_B(\mathcal{A}_1^{R/r})\ge 2+\frac{12}{\left(\frac{R^2}{r^2}-1\right)^2}=2+\frac{12}{\left(\frac{R}{r}+1\right)^2\left(\frac{R}{r}-1\right)^2}.
\end{equation}
By letting $R\downarrow r$, this completes the proof of \eqref{asymptoticannulus} when $n=2$.\par
\underline{The case $n\ge3$.}
Let $\rho=R/r>1$;
by Lemma \ref{invariance}, we know that $C_B(\mathcal{A}_r^R)=C_B(\mathcal{A}_1^\rho)$ and, hence, we may restrict the
attention to this case. Inspired by \eqref{realimaginary}, we then consider the harmonic function
$$
\tilde f\in L^2_0 (\mathcal{A}_1^\rho,\R),\quad
\tilde f(x)=\left(n-1+\frac{1}{|x|^n} \right)x_1,\quad
$$
and $u\in H^1_0 (\mathcal{A}_1^\rho,\R^n)$ as the unique solution to $\Delta u=\nabla \tilde f$,
which satisfies $u\in H\subset H^1_{0,\sigma}(\mathcal{A}_1^\rho,\R^n)^\perp$, see \eqref{eq:comm_1.2}-\eqref{eq:comm_1.1}.
After putting $f\doteq \nabla\cdot u$, we see that $u=u_{f}$, as in Theorem \ref{minBog}. We find
 (here we need $n\ge3$):
\begin{align}
u_1(x)\doteq& x_1^2 \left( \frac{1}{2|x|^n}-\frac{c}{|x|^{n+2}}-d\right)
+\frac{n-1}{2n}|x|^2+a|x|^{2-n}-b+\frac{c}{n|x|^n}+\frac{d}{n}|x|^2\label{eq:comm_4.1}\\
u_j(x)\doteq& x_1x_j \left( \frac{1}{2|x|^n}-\frac{c}{|x|^{n+2}}-d\right)
\qquad (j=2,\ldots,n)
\label{eq:comm_4.2}
\end{align}
with
\begin{align*}
a=\frac{n-1}{2n}\cdot \frac{\rho^n-\rho^{n-2}}{\rho^{n-2}-1}-\frac{1}{2n},\quad
b=\frac{n-1}{2n}\cdot \frac{\rho^n-1}{\rho^{n-2}-1},\quad c=\frac{\rho^{n+2}-\rho^2}{2(\rho^{n+2}-1)},\quad
d=\frac{\rho^2-1}{2(\rho^{n+2}-1)},
\end{align*}
so that the homogeneous boundary conditions are satisfied. We then compute
\begin{align*}
f(x)=\frac{n-1}{n}\left(1-\frac{n+2}{2}\cdot\frac{\rho^2-1}{\rho^{n+2}-1} \right)x_1
+\frac{n-1}{n}\left(1-\frac{n-2}{2}\cdot\frac{\rho^n-\rho^{n-2}}{\rho^{n-2}-1}\right)\frac{x_1}{|x|^n}=\frac{n-1}{n}\left(C+\frac{D}{|x|^n}\right) \! x_1,
\end{align*}
where
$$
C=1-\frac{n+2}{2}\cdot\frac{\rho^2-1}{\rho^{n+2}-1},\qquad D=1-\frac{n-2}{2}\cdot\frac{\rho^n-\rho^{n-2}}{\rho^{n-2}-1}.
$$
We now determine a lower bound for the Bogovskii constant of the annuli $\mathcal{A}_1^\rho$:
\begin{align}
C_B(\mathcal{A}_r^R)=\, &C_B(\mathcal{A}_1^\rho)\ge \frac{\| u_f \|^2_{H^1_0(\mathcal{A}_1^\rho)}}{\|f \|^2_{L^2(\mathcal{A}_1^\rho)}}=\frac{\langle f,\tilde f \rangle_{L^2(\mathcal{A}_1^\rho)}}{\|f \|^2_{L^2(\mathcal{A}_1^\rho)}}\nonumber\\
=\, &\frac{n}{n-1}\cdot \frac{
\frac{n-1}{n+2}C(\rho^{n+2}-1)
+\frac12\left(C+(n-1)D\right)(\rho^2-1)
+\frac{1}{n-2}D(1-\rho^{2-n})
}{
\frac{1}{n+2}C^2(\rho^{n+2}-1)
+C\,D\,(\rho^2-1)
+\frac{1}{n-2}D^2(1-\rho^{2-n})
}\nonumber\\
=&\frac{n}{n-1}\cdot \frac{
(n-1)(n-2)\rho^{n-2}(\rho^{n+2}-1)	+ (n+2)(\rho^{n-2}-1)
}{
(n-2)\rho^{n}(\rho^{n}-1)-(n+1)(n-2)\rho^{n-2}(\rho^{2}-1)+ (n+2)(\rho^{n-2}-1)}.\label{eq:comm_4.4}
\end{align}
The last term in \eqref{eq:comm_4.4} is strictly greater than $n$ if and only if $(n-1)\rho^n-n\rho^{n-2}+1>0$ which is always
true since $\rho>1$. This proves \eqref{eq:comm_4.3} when $n\ge3$.\par
To obtain \eqref{asymptoticannulus} we apply de L'H\^opital's
rule to \eqref{eq:comm_4.4} and we recall that $\rho=R/r$.\end{proof}

\begin{remark} \label{airfoil}
	The Zhukovsky transform \eqref{zhuk} can be seen as a complex-valued function, defined in $\mathbb{C} \setminus \{0\}$. It
	is a conformal function that maps circles into airfoils, see \cite[Chapter 4]{acheson1990elementary} for further details. Figure \ref{airfoil} illustrates how the unit circle,
	centered at the point $(1/5, 1/5)$, is transformed into an airfoil with trailing edge located at $(-2,0)$.
	\begin{figure}[h] \label{airfoilfig}
		\begin{center}
			\includegraphics[scale=0.46]{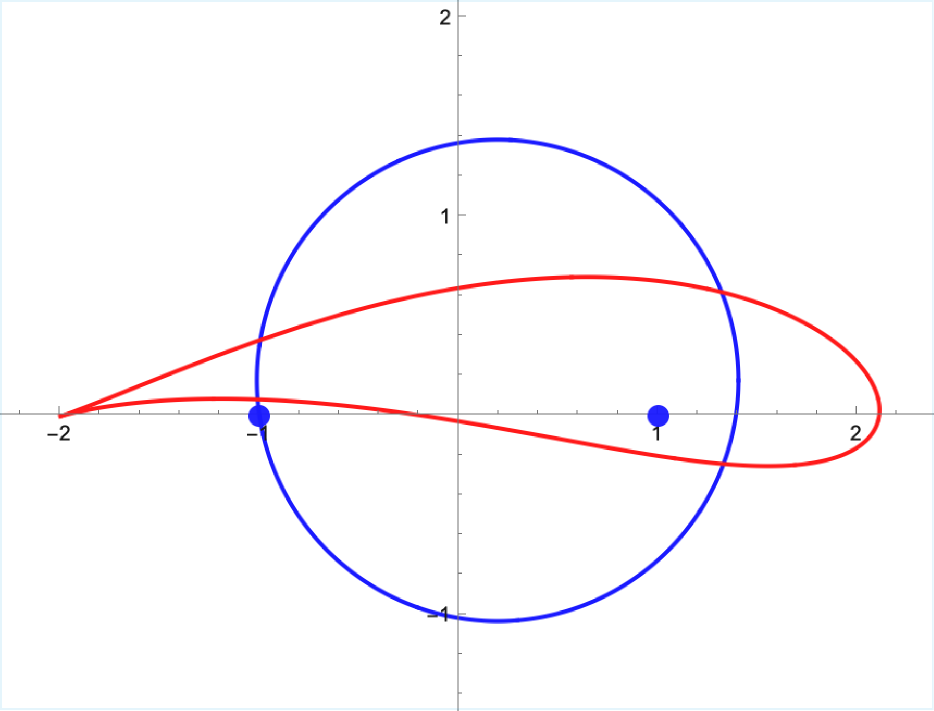}	
		\end{center}
		\vspace{-5mm}
		\caption{Geometric description of the Zhukovsky transform.}
	\end{figure}
\end{remark}

\begin{remark}\label{2nd}
Only for the 2D (non-simply connected) annulus, by using again (see Remark \ref{rem:thin_rectangles}) the eigenvalues of the Schur complement operator, Chizhonkov-Olshanskii
\cite[Corollary 3.3]{Chizhonkov_Olshanskii_2000} prove that
$$
C_B(\mathcal{A}_r^R)=\frac{1}{\mu(\mathcal{A}_r^R)^2}\ge \frac{24}{7}\frac{1}{\left(\frac{R}{r}-1\right)^2}\quad\forall R>r>0\, ,\qquad
\lim_{R\downarrow r}\left(\frac{R}{r}-1\right)^2\, C_B(\mathcal{A}_r^R)=12\, .
$$
Asymptotically, these estimates are qualitatively of the same order but quantitatively stronger and more precise than
\eqref{asymptoticannulus}. On the other hand, they require extensive calculations in polar coordinates and the strict lower bound
$2$ is by no means obvious. Not only our proof includes any dimension $n\ge2$ but also  appears to be more direct.
\end{remark}

\section{Higher-order regularity for the divergence equation}\label{polyBogo}

\subsection{Variational characterisation, non-minimality criterion and attainment}
Let $\Omega\subset\R^n$ be a bounded domain of class ${\mathcal C}^{m+1,1}$. From Theorem \ref{minBog} we know that
$$f\in H^m(\Omega, \mathbb{R})\cap L^2_0(\Omega, \mathbb{R})\ \Longrightarrow\ u_f\in H^{m+1}(\Omega, \mathbb{R}^{n})\cap H^1_0(\Omega, \mathbb{R}^{n})\, .$$
In the present section, we maintain the mere assumption of $\partial\Omega$ being Lipschitz and we gain regularity by increasing the number
of null traces of $f$. Of course, the resulting privileged solution $u_f$ will change.\par
To be precise, let $m\in\N$.  Given $u,v \in H_0^m(\Omega, \mathbb{R})$, the bilinear form
\begin{equation} \label{altranorma}
	(u,v)_{m,\Omega}\mapsto
	\left\{
	\begin{aligned}
		& \int_\Omega (\Delta^k u) (\Delta^k v) \, dx \quad \mbox{if} \ \ m=2k \, ,\\[6pt]
		& \int_\Omega\nabla(\Delta^k u)\cdot\nabla(\Delta^kv)\, dx\quad \mbox{if} \ \ m=2k+1 \, ,
	\end{aligned}
	\right.
\end{equation}
defines a scalar product on $H_0^m(\Omega, \mathbb{R})$ and this scalar product induces a norm equivalent to the $H^m(\Omega, \mathbb{R})$-norm, see  \cite[Theorem 2.2]{GGSbook};
in the sequel  $\|\cdot\|_{m,\Omega}$ denotes  this norm (in a completely similar way we can define the bilinear form \eqref{altranorma} on the space of vector fields in $H_0^m(\Omega, \mathbb{R}^{n})$).\par
From the work of Bogovskii \cite{bogovskii1979solution}, see also \cite[Theorem III.3.3]{galdi2011introduction} and \cite{geissert2006equation}, we
know that for any function $f\in H^m_0(\Omega; \mathbb{R}) \cap L^2_0(\Omega, \mathbb{R})$ there exists a vector field $u\in H^{m+1}_0(\Omega, \mathbb{R}^{n})$ solving \eqref{Bogdom}. Employing again \eqref{phiPhi},
one sees that such regular solution is not unique, which suggests to define also {\em higher-order Bogovskii constants}.

\begin{definition} \label{defboghigher}
Let $\Omega\subset\R^n$ be a bounded Lipschitz domain and let $m\in\N$. We put
	$$
	X_{m}(\Omega, \mathbb{R}) \doteq H_{0}^{m}(\Omega, \mathbb{R}) \cap L_{0}^{2}(\Omega, \mathbb{R}) \, ,
	$$
which is a closed subspace of $H_{0}^{m}(\Omega; \mathbb{R})$. The quantity
	\begin{equation}\label{bogok}
		C_{B}^{(m)}(\Omega) \doteq   \sup_ {f\in X_{m}(\Omega, \mathbb{R}) \atop \|f\|_{m,\Omega}=1}\quad\inf\Big\{\|v\|_{m+1,\Omega}^2\, ;
		\ v\in H^{m+1}_0(\Omega, \mathbb{R}^{n})\, ,\ \nabla\cdot v=f \text{ in } \Omega\Big\}
	\end{equation}
	is called the {\em $m$-th order Bogovskii constant} of $\Omega$. 
\end{definition}

Note that consistently $C_{B}^{(0)}(\Omega)=C_B(\Omega)$, as defined in \eqref{bogo11}.
The higher-order version of Theorem \ref{minBog}, implicitly including regularity, reads:

\begin{theorem}\label{minBoghigh}
	Let $\Omega\subset\R^n$ be a bounded Lipschitz domain and let $m\ge1$ be an integer. For any function $f\in X_{m}(\Omega, \mathbb{R})$ there exists a unique solution
	$u_f\in H^{m+1}_0(\Omega, \mathbb{R}^{n})$ of \eqref{Bogdom} such that
	\begin{equation}\label{infm}
		\|u_f\|_{m+1,\Omega}^2=\inf\Big\{\|v\|_{m+1,\Omega}^2\, ;\ v\in H^{m+1}_0(\Omega, \mathbb{R}^{n})\, ,\ \nabla\cdot v=f\text{ in }\Omega\Big\}\, .
	\end{equation}
	Moreover, with the scalar product in \eqref{altranorma}, the Euler-Lagrange equation
	\begin{equation}\label{ELeqm}
		(u_f,v)_{m+1, \Omega}=0\qquad\forall v\in H^{m+1}_{0, \sigma}(\Omega, \mathbb{R}^{n})\doteq\{w\in H^{m+1}_0(\Omega, \mathbb{R}^{n});\, \nabla \cdot w=0\mbox{ in }\Omega\}\,
	\end{equation}
	is a necessary and sufficient condition for a solution $u\in H^{m+1}_0(\Omega, \mathbb{R}^{n})$ of \eqref{Bogdom} to minimise \eqref{infm} so that $u\equiv u_f$.
	Finally, there exists a unique scalar function $p_{f} \in X_{m}(\Omega, \mathbb{R})$ such that the pair $(u_{f},p_{f})$ weakly solves the following higher-order Stokes system in $\Omega$:
	\begin{equation} \label{gstokesm}
		(-\Delta)^{m+1} u_{f} + \nabla ((-\Delta)^{m} p_{f}) = 0 \, , \quad \nabla \cdot u_{f} = f \ \ \mbox{ in } \ \ \Omega \, .
	\end{equation}
\end{theorem}
\noindent
\begin{proof} If $f=0$ the result is trivial and $u_f \equiv 0$.\par
	If $f\neq0$, by linearity of \eqref{Bogdom} we may assume that $\|f\|_{m,\Omega}=1$. Fixed any such $f$, consider a sequence of solutions $\{u_k\}\subset H_0^{m+1}(\Omega, \mathbb{R}^{n})$
of \eqref{Bogdom} whose norms $\|u_k\|_{m+1,\Omega}^2$ tend to the infimum
	in \eqref{infm}. Then $\{u_k\}_{k \in \mathbb{N}}$ is bounded  and, up to a subsequence, it converges weakly to some
	$u_f\in H_0^{m+1}(\Omega, \mathbb{R}^{n})$. In turn,
	$\nabla\cdot u_k\rightharpoonup \nabla\cdot u_f$ in $H^m(\Omega, \mathbb{R})$ as $k\to\infty$, which shows that $u_f$ solves \eqref{Bogdom}. Moreover,
	by lower semicontinuity of the norm with respect to weak convergence, we find
	$$
	\|u_f\|_{m+1,\Omega}^2 \le \liminf_{k\to\infty}\|u_k\|_{m+1,\Omega}^2
	=\inf\Big\{\|v\|_{m+1,\Omega}^2\, ;\ v\in H^{m+1}_0(\Omega, \mathbb{R}^{n})\, ,\ \nabla\cdot v=f\text{ in }\Omega\Big\} \doteq M \, ,
	$$
	which shows that $u_f\in H_0^{m+1}(\Omega, \mathbb{R}^{n})$ achieves the minimum in \eqref{infm}.\par
	In order to prove uniqueness, suppose that two vector fields $u,v\in H^{m+1}_0(\Omega,\mathbb{R}^{n})$ solve
	\eqref{Bogdom} with $\|u\|_{m+1,\Omega}^2=\|v\|_{m+1,\Omega}^2=M$. Then, by linearity,
	any convex combination $w_\alpha=\alpha u+(1-\alpha)v$, $\alpha \in \mathbb{R}$, also solves \eqref{Bogdom}. Moreover, by the H\"older inequality,
	\begin{align*}
		\|w_\alpha\|_{m+1,\Omega}^2 & = \alpha^2\|u\|_{m+1,\Omega}^2+(1-\alpha)^2\|v\|_{m+1,\Omega}^2+2\alpha(1-\alpha)(u,v)_{m+1,\Omega}\\[6pt]
		&\le \alpha^2 M+(1-\alpha)^2M+2\alpha(1-\alpha)\|u\|_{m+1,\Omega}\|v\|_{m+1,\Omega}=M\, .
	\end{align*}
	Since $M$ is the minimum, the above inequality is an equality, that is,
	$$
	(u,v)_{m+1,\Omega}=\|u\|_{m+1,\Omega}\|v\|_{m+1,\Omega}\, .
	$$
	But the H\"older inequality becomes an equality if and only if the two involved functions are proportional, namely
	$v\equiv\gamma u$ for some $\gamma>0$. Since $\|u\|_{m+1,\Omega}^2=\|v\|_{m+1,\Omega}^2=M$, we necessarily have $\gamma=1$ so that $v\equiv u$.
	This also proves that the Euler-Lagrange equation \eqref{ELeqm} is satisfied for any $v\in H^{m+1}_{0, \sigma}(\Omega, \mathbb{R}^{n})$.\par
	Letting  $H^{-m-1}(\Omega, \mathbb{R}^{n})$ denote the dual space of $H_{0}^{m+1}(\Omega, \mathbb{R}^{n})$, we introduce the linear mapping
	$\mathcal{L}_{f} : H_{0}^{m+1}(\Omega, \mathbb{R}^{n}) \longrightarrow \mathbb{R}$ by
	$$
	\begin{aligned}
		\mathcal{L}_{f}(v) \doteq (u_f,v)_{m+1, \Omega} \qquad \forall v \in H_{0}^{m+1}(\Omega, \mathbb{R}^{n}) \, .
	\end{aligned}
	$$
	We certainly have $\mathcal{L}_{f} \in H^{-m-1}(\Omega, \mathbb{R}^{n})$ and, due to \eqref{ELeqm}, that $\mathcal{L}_{f}(v)=0$ for every $v \in H^{m+1}_{0, \sigma}(\Omega, \mathbb{R}^{n})$.
	
	Next, given any $q \in X_{m}(\Omega, \mathbb{R})$, there exists a vector field $J_{q} \in H_{0}^{m+1}(\Omega, \mathbb{R}^{n})$ such that
	\begin{equation} \label{vecjeprebogm}
		\nabla \cdot J_{q} = q \ \ \mbox{in} \ \ \Omega \qquad \text{and} \qquad \| J_{q} \|_{m+1, \Omega} \leq \sqrt{C^{(m)}_{B}(\Omega)} \, \| q\|_{m, \Omega} \, .
	\end{equation}
	Define then the mapping $\widehat{\mathcal{L}}_{f} : X_{m}(\Omega, \mathbb{R}) \longrightarrow \mathbb{R}$ by
	\begin{equation} \label{vecjepre22bogm}
		\widehat{\mathcal{L}}_{f}(q) \doteq \mathcal{L}_{f}(J_{q}) \qquad \forall q \in X_{m}(\Omega, \mathbb{R}) \, .
	\end{equation}
	Notice that $\widehat{\mathcal{L}}_{f}$ is well-defined: given $q \in X_{m}(\Omega, \mathbb{R})$, if $J_{1}, J_{2} \in H_{0}^{m+1}(\Omega, \mathbb{R}^{n})$ are two different vector fields such that
	$$
	\nabla \cdot J_{1} = \nabla \cdot J_{2} = q \ \ \mbox{in} \ \ \Omega \, ,
	$$
	then $J_{1} - J_{2} \in  H^{m+1}_{0, \sigma}(\Omega, \mathbb{R}^{n})$. This means that $\mathcal{L}_{f}(J_{1} - J_{2}) = 0$, and, by linearity, that $\mathcal{L}_{f}(J_{1}) = \mathcal{L}_{f}(J_{2})$. Since $\widehat{\mathcal{L}}_{f}$ is linear, \eqref{vecjeprebogm}$_2$ ensures that $\widehat{\mathcal{L}}_{f}$ is a bounded linear functional on $X_{m}(\Omega, \mathbb{R})$. From the Riesz Representation Theorem we then deduce the existence of a unique $p_{f} \in X_{m}(\Omega, \mathbb{R})$ such that
	\begin{equation} \label{vecjepre33bogm}
		\widehat{\mathcal{L}}_{f}(q) = (p_f,q)_{m, \Omega}  \qquad \forall q \in X_{m}(\Omega, \mathbb{R}) \, .
	\end{equation}
	Given any $v \in H_{0}^{m+1}(\Omega, \mathbb{R}^{n})$, we infer from \eqref{vecjepre22bogm}-\eqref{vecjepre33bogm} that
	$$
	\mathcal{L}_{f}(v) = \widehat{\mathcal{L}}_{f}(\nabla \cdot v) = (p_f,\nabla \cdot v)_{m, \Omega}  \, .
	$$
	Thus, the vector field $u_{f} \in H_{0}^{m+1}(\Omega, \mathbb{R}^{n})$ has a unique associated ``pressure'' $p_{f} \in X_{m}(\Omega, \mathbb{R})$ such that
	\begin{equation} \label{vecjepre34bogm}
		(u_f,v)_{m+1, \Omega} - (p_f,\nabla \cdot v)_{m, \Omega} = 0 \qquad \forall v \in H_{0}^{m+1}(\Omega, \mathbb{R}^{n}) \, ,
	\end{equation}
	which corresponds to the weak formulation of the higher-order Stokes problem \eqref{gstokesm} in $\Omega$.
\end{proof}

One may find it surprising that the higher-order version of Theorem \ref{lowerbound} looks rather similar. We recall the definition \eqref{eq:def_K_i} of $\mathcal{K}_{i}$.

\begin{theorem}\label{lowerboundHO}
	Let $\Omega\subset\R^n$ bounded Lipschitz domain. For any $m \in \mathbb{N}$ one has
	$$
	C^{(m)}_{B}(\Omega)\ge n \, .
	$$
	If there exist $i \in \{1,...,n\}$ and $g\in\mathcal{K}_{i}(\Omega, \mathbb{R}) \cap L^{2}(\Omega, \mathbb{R})$ such that the unique solution $w \in H^{m+1}_{0}(\Omega, \mathbb{R})$ to
	\begin{equation}\label{generalTorsionho}
		(-\Delta)^{m+1} w = g \ \ \mbox{ in } \ \ \Omega \, ,
	\end{equation}
	satisfies
	\begin{equation}\label{conditionho}
		\left\| \dfrac{\partial w}{\partial x_{i}} \right\|^{2}_{m, \Omega} < \frac{1}{n} \| w \|^{2}_{m+1, \Omega} \, ,
	\end{equation}
	then $C_{B}^{(m)}(\Omega)>n$.\\
\end{theorem}
\noindent
\begin{proof}
Since the case $m=0$ was considered in the proof of Theorem \ref{lowerbound}, we assume $m \geq 1$.
Given any $w\in H^{m+1}_{0}(\Omega, \mathbb{R})$, from \cite[Formula (2.12)]{GGSbook} we know that
	\begin{equation}\label{elementarym}
		\| w \|^{2}_{m+1, \Omega} = \int_\Omega|D^{m+1} w|^2\, dx=\sum_{i=1}^{n} \left( \int_\Omega \left| D^{m}(\partial_{x_i} w) \right|^2 \right)
		\ge n\ \min_{i=1,...,n} \ \| \partial_{x_i} w \|^{2}_{m, \Omega} \, .
	\end{equation}
	Here it is essential to define
	$$
	|D^{m} w|^2:=\sum^n_{i_1,\ldots,i_m=1} (\partial_{x_{i_1}}\ldots \partial_{x_{i_m}} w)^2
	$$
	in order to have that  $\| w \|^{2}_{m, \Omega}=(w,w)_{m, \Omega}$ as defined in \eqref{altranorma}. The multiindex notation fails in this respect, if $m\ge 2$.
We apply \eqref{elementarym} to the unique weak solution $w\in H^{m+1}_{0}(\Omega, \mathbb{R})$ of the problem
$(-\Delta)^{m+1} w=1$ in $\Omega$ under homogeneous Dirichlet boundary conditions, that is,
	\begin{equation}\label{partialk0ho}
		(w, \varphi)_{m+1,\Omega} = \int_\Omega \varphi \, dx\qquad \forall \varphi \in H^{m+1}_{0}(\Omega, \mathbb{R}) \, .
	\end{equation}
Let $k\in\{1,...,n\}$ denote the index of one of the partial derivatives of $w$ attaining the minimum in the right hand side
of \eqref{elementarym} so that
	\begin{equation}\label{partialkm}
	\| w \|^{2}_{m+1, \Omega} \geq n \, \| \partial_{x_k} w \|^{2}_{m, \Omega} \, .
	\end{equation}
	Take the vector field $u=(u_1,...,u_n)\in H^{m+1}_{0}(\Omega, \mathbb{R}^{n})$ such that $u_k=w$ and $u_i\equiv0$ for $i\neq k$, and define $f \in X_{m}(\Omega, \mathbb{R})$ by
	$$
	f\doteq\nabla\cdot u=\dfrac{\partial w}{\partial x_{k}} \quad \text{in} \ \ \Omega \, .
	$$
	We claim that $u = u_f$; given any divergence-free vector field $v = (v_{1},...,v_{n}) \in H^{m+1}_{0, \sigma}(\Omega, \mathbb{R}^{n})$, notice that
	$$
	(u, v)_{m+1,\Omega} = (w, v_{k})_{m+1,\Omega} = \int_{\Omega} v_{k} \, dx = 0 \, ,
	$$
	owing to Lemma \ref{lem:lemma1} and \eqref{partialk0ho}, so that the Euler-Lagrange equation \eqref{ELeqm} is fulfilled. Then, Theorem \ref{minBoghigh} combined with \eqref{bogok} and \eqref{partialkm} gives
	$$
	C_B^{(m)}(\Omega)=\sup_ {h\in X_{m}(\Omega, \mathbb{R}) \setminus \{0\}}\frac{\|u_h\|_{m+1, \Omega}^{2}}{\|h\|_{m, \Omega}^{2}} \geq \frac{\|u_f\|_{m+1, \Omega}^{2}}{\|f\|_{m, \Omega}^{2}}=\frac{\|w\|_{m+1, \Omega}^{2}}{\|\partial_{x_k} w\|_{m, \Omega}^{2}}
	\ge n\, ,
	$$
	thus providing the stated lower bound for $C_B^{(m)}(\Omega)$.\par
	Now, consider $i \in \{1,...,n\}$, $g\in\mathcal{K}_{i}(\Omega, \mathbb{R}) \cap L^{2}(\Omega, \mathbb{R})$ and $w\in H^{m+1}_{0}(\Omega, \mathbb{R})$ satisfying \eqref{generalTorsionho}-\eqref{conditionho}, meaning  that
	\begin{equation}\label{partialk1ho}
		(w, \varphi)_{m+1,\Omega} = \int_\Omega g \, \varphi \, dx\qquad \forall \varphi \in H^{m+1}_{0}(\Omega, \mathbb{R}) \, .
	\end{equation}
	Take the vector field
	$u=(u_1,...,u_n)\in H^{m+1}_{0}(\Omega, \mathbb{R}^{n})$ such that $u_i=w$ and $u_j\equiv0$ for $j \neq i$, and define $h \in X_{m}(\Omega, \mathbb{R})$ by
	$$
	h\doteq\nabla\cdot u=\dfrac{\partial w}{\partial x_{i}} \quad \text{in} \ \ \Omega \, .
	$$
	Again, we claim that $u = u_h$; given any vector field $v = (v_{1},...,v_{n}) \in H^{m+1}_{0, \sigma}(\Omega, \mathbb{R}^{n})$, notice that
	$$
	(u, v)_{m+1,\Omega} = (w, v_{i})_{m+1,\Omega} = \int_{\Omega} g \, v_{i}\, dx = 0 \, ,
	$$
	owing to Lemma \ref{lem:lemma1} and \eqref{partialk1ho}, so that the Euler-Lagrange equation \eqref{ELeqm} is fulfilled. Then, by using \eqref{bogok} and \eqref{conditionho}, we find
	$$
	C_B^{(m)}(\Omega) \ge \frac{\|u_h\|_{m+1, \Omega}^{2}}{\|h\|_{m, \Omega}^{2}}=\frac{\|w\|_{m+1, \Omega}^{2}}{\|\partial_{x_i} w\|_{m, \Omega}^{2}}>n\, ,
	$$
	which completes the proof of the second statement.
\end{proof}

Let $m\ge1$ be an integer. Following the line of Section \ref{VEL}, the orthogonal decomposition
\begin{equation} \label{ortdecho}
	H_{0}^{m+1}(\Omega, \mathbb{R}^{n}) = H_{0, \sigma}^{m+1}(\Omega, \mathbb{R}^{n}) \oplus H_{0, \sigma}^{m+1}(\Omega, \mathbb{R}^{n})^{\perp}
\end{equation}
turns out to be useful in the current setting. In fact, given $p\in X_{m}(\Omega, \mathbb{R})$ there exists a unique vector field $u\in H_{0, \sigma}^{m+1}(\Omega, \mathbb{R}^{n})^{\perp}$ such that
\begin{equation}\label{eq:comm_1.2ho}
	(-\Delta)^{m+1} u + \nabla ((-\Delta)^{m} p) = 0 \quad \text{in the sense of} \ \ H^{-m-1}(\Omega, \mathbb{R}^{n}) \, ,
\end{equation}
that is,
\begin{equation}\label{eq:comm_1.1ho}
	(u,v)_{m+1, \Omega} - (p,\nabla \cdot v)_{m, \Omega} = 0 \qquad \forall v \in H_{0}^{m+1}(\Omega, \mathbb{R}^{n}) \, .
\end{equation}
Existence and uniqueness of a vector field $u \in H_{0}^{m+1}(\Omega; \mathbb{R}^{n})$ satisfying \eqref{eq:comm_1.2ho} follow directly from Riesz's Representation Theorem, while \eqref{eq:comm_1.1ho}
ensures that $u\in H_{0, \sigma}^{m+1}(\Omega, \mathbb{R}^{n})^{\perp}$. Conversely, as in the proof of Theorem \ref{minBoghigh} one sees that, given any
$u\in H_{0, \sigma}^{m+1}(\Omega, \mathbb{R}^{n})^{\perp}$, there exists a unique $p\in X_{m}(\Omega, \mathbb{R})$ such that \eqref{eq:comm_1.1ho} is satisfied. This establishes a bijective correspondence
$$
H_{0, \sigma}^{m+1}(\Omega, \mathbb{R}^{n})^{\perp} \ni u \mapsto p_u\in  X_{m}(\Omega, \mathbb{R})
$$
via the weak formulation \eqref{eq:comm_1.1ho}. In particular, it holds
\begin{equation} \label{equivnormsho}
	\|u\|_{m+1,\Omega} = \|(-\Delta)^{m+1} u\|_{H^{-m-1}(\Omega)} = \|\nabla ((-\Delta)^{m} p_{u})\|_{H^{-m-1}(\Omega)} \qquad \forall u \in H_{0, \sigma}^{m+1}(\Omega, \mathbb{R}^{n})^{\perp} \, .
\end{equation}

Theorem \ref{minBoghigh} enables us to slightly modify the variational characterisation of $C^{(m)}_B(\Omega)$ in \eqref{bogok}:
\begin{equation}\label{bogoho}
	C^{(m)}_{B}(\Omega) = \sup_ {f\in X_{m}(\Omega, \mathbb{R}) \setminus \{0\}} \ \dfrac{\|u_f\|_{m+1,\Omega}^2}{\|f\|_{m,\Omega}^2} \, ,
\end{equation}
with $u_{f} \in H_{0}^{m+1}(\Omega, \mathbb{R}^{n})$ being the unique vector field in $H_{0}^{m+1}(\Omega, \mathbb{R}^{n})$ satisfying the Euler-Lagrange equation \eqref{ELeqm}, or equivalently, the higher-order Stokes system \eqref{gstokesm}. Moreover, Theorem \ref{minBoghigh} and \eqref{bogoho} define the bounded linear operator
$$\mathcal{B}_{m} : X_{m}(\Omega, \mathbb{R}) \longrightarrow H_{0}^{m+1}(\Omega, \mathbb{R}^{n})\, ,\qquad \mathcal{B}(f) \doteq u_{f}\quad\forall f\in X_m(\Omega, \mathbb{R})$$
where boundedness (and continuity) follow from \eqref{bogok}: $\| \mathcal{B}_{m} \|^{2} = C^{(m)}_B(\Omega)<\infty$.

In analogy with Proposition \ref{lem:comm_1.6}, we also prove the higher-order version of the Ne\v{c}as inequality \eqref{eq:comm_1.3}$_2$:

\begin{proposition}\label{lem:comm_1.6ho}
	For any bounded Lipschitz domain $\Omega \subset \R^n$ and any integer $m  \geq 1$ one has:
	\begin{equation}\label{eq:comm_1.13ho}
		C^{(m)}_{B}(\Omega) = \sup_ {f\in X_{m}(\Omega) \setminus \{0\}} \ \dfrac{\|f\|_{m,\Omega}^2}{\|\nabla ((-\Delta)^{m} f)\|_{H^{-m-1}(\Omega)}^2} \, .
	\end{equation}	
\end{proposition}
\noindent
\begin{proof}
	Given any $f\in  X_{m}(\Omega, \mathbb{R})$, the Euler-Lagrange equation \eqref{ELeqm} implies that $u_{f} \in H^{m+1}_{0,\sigma}(\Omega, \mathbb{R}^{n})^\perp$. Therefore, for any solution $u \in H^{m+1}_{0}(\Omega, \mathbb{R}^{n})$
	of the divergence equation \eqref{Bogdom} one has $u=u_f$ if and only if $u\in H^{m+1}_{0,\sigma}(\Omega, \mathbb{R}^{n})^\perp$.
	Based upon this observation, from \eqref{bogoho} we deduce:
	\begin{equation}\label{eq:comm_1.3ho}
		C^{(m)}_{B}(\Omega) = \sup_ {u \in H^{m+1}_{0,\sigma}(\Omega, \mathbb{R}^{n})^\perp \setminus \{0\}} \ \dfrac{\|u\|_{m+1,\Omega}^2}{\|\nabla \cdot u\|_{m,\Omega}^2} \, .
	\end{equation}	
Recalling \eqref{ortdecho}, for any $f\in X_{m}(\Omega, \mathbb{R}) \setminus \{0\}$ we find
	\begin{align*}
		\|\nabla ((-\Delta)^{m} f)\|_{H^{-m-1}(\Omega)}=&\sup_{u\in H^{m+1}_{0}(\Omega, \mathbb{R}^{n})\setminus \{0\}} \ \dfrac{ | (f, \nabla \cdot u)_{m, \Omega} |}{\| u\|_{m+1, \Omega}}
		=\sup_{u\in H^{m+1}_{0,\sigma}(\Omega, \mathbb{R}^{n})^\perp\setminus \{0\}} \ \dfrac{|(f, \nabla \cdot u)_{m, \Omega}|}{\| u\|_{m+1, \Omega}} \\[6pt]
		=& \sup_{u\in H^{m+1}_{0,\sigma}(\Omega, \mathbb{R}^{n})^\perp\setminus \{0\}} \ \dfrac{|(\nabla \cdot u_{f}, \nabla \cdot u)_{m, \Omega}|}{\| u\|_{m+1, \Omega}} \geq \frac{\|\nabla\cdot u_f\|^2_{m, \Omega}}{\| u_f\|_{m+1,\Omega}} \, .
	\end{align*}
	From this we conclude, with the help of \eqref{eq:comm_1.3ho},
	\begin{align*}
		C^{(m)}_{\mathcal{N}}(\Omega) & \doteq \sup_ {f\in X_{m}(\Omega, \mathbb{R}) \setminus \{0\}} \ \dfrac{\|f\|_{m,\Omega}^2}{\|\nabla ((-\Delta)^{m} f)\|_{H^{-m-1}(\Omega)}^2} =
		\sup_{f\in X_{m}(\Omega, \mathbb{R}) \setminus \{0\}} \ \dfrac{\|\nabla \cdot u_{f}\|_{m,\Omega}^2}{\|\nabla ((-\Delta)^{m} f)\|_{H^{-m-1}(\Omega)}^2} \\[6pt]
		\leq &\sup_{f\in X_{m}(\Omega, \mathbb{R}) \setminus \{0\}} \frac{\| u_f\|^{2}_{m+1,\Omega}}{\|\nabla\cdot u_f\|^2_{m, \Omega}} = \sup_ {u \in H^{m+1}_{0,\sigma}(\Omega, \mathbb{R}^{n})^\perp \setminus \{0\}} \ \dfrac{\|u\|_{m+1,\Omega}^2}{\|\nabla \cdot u\|_{m,\Omega}^2}
		=C_B^{(m)}(\Omega) \, .
	\end{align*}
On the other hand, take any $f\in X_{m}(\Omega, \mathbb{R})  \setminus \{0\}$. Recalling Theorem \ref{minBoghigh}, consider the unique pair $(u_{f}, p_{f}) \in H^{m+1}_{0}(\Omega,\R^n) \times X_{m}(\Omega,\R)$ satisfying the higher-order Stokes system \eqref{gstokesm}, so that
\begin{equation} \label{pressureandvelho}
\|\nabla ((-\Delta)^{m} p_{f})\|_{H^{-m-1}(\Omega)} = \| (-\Delta)^{m+1} u_{f}\|_{H^{-m-1}(\Omega)} = \|u_{f}\|_{m+1,\Omega} \, .
\end{equation}
Moreover, by testing \eqref{gstokesm} with $u_{f}$ (see \eqref{vecjepre34bogm}) and applying the Cauchy-Schwarz inequality, we obtain
$$
\|u_{f}\|^{2}_{m+1,\Omega} = (p_{f} \,, f)_{m, \Omega} \leq \|p_{f}\|_{m, \Omega} \, \|f\|_{m, \Omega} \leq \sqrt{C^{(m)}_{\mathcal{N}}(\Omega)} \, \|\nabla ((-\Delta)^{m} p_{f})\|_{H^{-m-1}(\Omega)} \, \|f\|_{m, \Omega} \, ,
$$
which, due to \eqref{pressureandvelho}, simplifies to
$$
\|u_{f}\|_{m+1,\Omega}  \leq \sqrt{C^{(m)}_{\mathcal{N}}(\Omega)} \, \|f\|_{m, \Omega} \qquad \forall f\in X_{m}(\Omega,\R) \setminus \{0\} \, .
$$
Owing to \eqref{bogoho}, the last inequality implies $C^{(m)}_{B}(\Omega) \leq C^{(m)}_{\mathcal{N}}(\Omega)$, which concludes the proof.
\end{proof}

We next discuss whether the supremum in \eqref{bogoho} is attained. To approach this question, we define
$$
A_{m}(\Omega, \mathbb{R}) \doteq \{ p \in X_{m}(\Omega, \mathbb{R}) \, ; (\exists \psi \in H_{0}^{m+2}(\Omega, \mathbb{R})) \ p = \Delta \psi \ \ \text{in} \ \ \Omega \, \} \, ,
$$
so that the following orthogonal decomposition holds:
\begin{equation} \label{decompl02ho}
X_{m}(\Omega, \mathbb{R}) = A_{m}(\Omega, \mathbb{R}) \oplus A_{m}(\Omega, \mathbb{R})^{\perp} \, ,
\end{equation}
see \cite[Part II]{riedl2010cosserat}. Notice that $A_{m}(\Omega, \mathbb{R})^{\perp}$ is the subset of $X_{m}(\Omega, \mathbb{R})$ comprising weakly $(m+1)$-polyharmonic functions. Then, the \textit{higher-order Schur complement of the Stokes operator} in $\Omega$ is the linear mapping $\mathcal{S}_{m} : X_{m}(\Omega, \mathbb{R}) \longrightarrow X_{m}(\Omega, \mathbb{R})$ obeying the formula
$$
\mathcal{S}_{m}(q) \doteq \nabla \cdot V_{q} \qquad \forall q \in X_{m}(\Omega, \mathbb{R}) \, ,
$$
where $V_{q} \in H_{0}^{m+1}(\Omega, \mathbb{R}^{n})$ is the unique vector field such that
$(-\Delta)^{m+1} V_{q} + \nabla ((-\Delta)^{m} q) = 0$,
whose existence and uniqueness follows from the Riesz Representation Theorem: $V_{q} \in H_{0}^{m+1}(\Omega; \mathbb{R}^{n})$ is the unique vector field such that
\begin{equation} \label{schurm0}
	(V_{q},v)_{m+1, \Omega} - (q,\nabla \cdot v)_{m, \Omega} = 0 \qquad \forall v \in H_{0}^{m+1}(\Omega, \mathbb{R}^{n}) \, .
\end{equation}
As in \cite[Theorem 1.2 and the subsequent remark]{riedl2013cosserat}, it can be shown that
\begin{equation} \label{sondecompho}
	\mathcal{S}_{m}(q) = q \qquad \forall q \in A_{m}(\Omega, \mathbb{R}) \, ; \qquad \mathcal{S}_{m}(q) \in A_{m}(\Omega, \mathbb{R})^{\perp} \qquad \forall q \in A_{m}(\Omega, \mathbb{R})^{\perp} \, .
\end{equation}
Furthermore, we notice that the operator $\mathcal{S}_{m}$ is self-adjoint, since
$$
(\mathcal{S}_{m}(p),q)_{m, \Omega} = (\nabla \cdot V_{p},q)_{m, \Omega} = (V_{q},V_{p})_{m+1, \Omega} = (p,\nabla \cdot V_{q})_{m, \Omega} = (p, \mathcal{S}_{m}(q))_{m, \Omega} \qquad \forall p,q \in X_{m}(\Omega, \mathbb{R}) \, .
$$
In particular, due to \eqref{equivnormsho} and Proposition \ref{lem:comm_1.6ho}, we have, for every $p \in X_{m}(\Omega, \mathbb{R})$,
$$
\begin{aligned}
	(\mathcal{S}_{m}(p),p)_{m, \Omega} & = \|V_{p}\|_{m+1,\Omega}^2 = \|\nabla ((-\Delta)^{m} p)\|_{H^{-m-1}(\Omega)}^2 \\[6pt]
	& \geq \left[ \inf_{f\in X_{m}(\Omega, \mathbb{R}) \setminus \{0\}} \ \dfrac{\|\nabla ((-\Delta)^{m} f)\|_{H^{-(m+1)}(\Omega)}^2}{\|f\|_{m,\Omega}^2} \right] \|p\|_{m,\Omega}^2 = \dfrac{1}{C^{(m)}_{B}(\Omega)} \|p\|_{m,\Omega}^2 \, ,
\end{aligned}
$$
thereby proving that $\mathcal{S}_{m}$ is also a positive definite  operator, with $1/C^{(m)}_{B}(\Omega)$  its optimal coercivity constant. In other words, we have
$$
\inf_{p \in X_{m}(\Omega, \mathbb{R}) \setminus \{0\}} \dfrac{(\mathcal{S}_{m}(p),p)_{m, \Omega}}{\|p\|_{m,\Omega}^2} = \dfrac{1}{C^{(m)}_{B}(\Omega)} \, , \qquad \sup_{p \in X_{m}(\Omega, \mathbb{R}) \setminus \{0\}} \dfrac{(\mathcal{S}_{m}(p),p)_{m, \Omega}}{\|p\|_{m,\Omega}^2} = 1 \, .
$$
Therefore, letting $\sigma(\mathcal{S}_{m})$ denote the spectrum of $\mathcal{S}_{m}$, it holds
\begin{equation} \label{spectrum1ho}
	\left\{ \dfrac{1}{C^{(m)}_{B}(\Omega)},1 \right\} \subset \sigma(\mathcal{S}_{m}) \subset \left[ \dfrac{1}{C_{B}^{(m)}(\Omega)},1 \right] \, , \qquad \text{with} \  \ C_{B}^{(m)}(\Omega) \geq n \, .
\end{equation}
Nevertheless, while $\lambda =1$ is indeed an eigenvalue of $\mathcal{S}_{m}$ of infinite multiplicity (due to \eqref{sondecompho}), it is not clear, a priori, that $1/C_{B}^{(m)}(\Omega)$ is also an eigenvalue of $\mathcal{S}_{m}$. The following result states that if $1/C_{B}^{(m)}(\Omega)$ is an eigenvalue of $\mathcal{S}_{m}$, then $\mathcal{B}_{m}$ is a norm-attaining operator between $X_{m}(\Omega, \mathbb{R})$ and $H_{0}^{m+1}(\Omega, \mathbb{R}^{n})$. The corresponding proof is omitted, since it is entirely similar to that of Theorem \ref{bogoatttheo}, but we refer to \cite[Chapter I, Section 6]{riedl2010cosserat}, \cite[Section 2]{riedl2013cosserat} and \cite[Section 11]{weyers2006q} for further details.

\begin{theorem}\label{bogoatttheoho}
	Let $m \in \mathbb{N}$ and $\Omega\subset\R^n$ be a bounded Lipschitz domain. If $f_{*} \in A_{m}(\Omega, \mathbb{R})^{\perp} \setminus \{0\}$ is an eigenfunction of the Schur complement $\mathcal{S}_{m}$, associated to the eigenvalue $1/C^{(m)}_{B}(\Omega)$, then
	\begin{equation}\label{bogoattho}
		C^{(m)}_{B}(\Omega) = \dfrac{\| \nabla \mathcal{B}_{m}(f_{*}) \|_{m+1, \Omega}^2}{\| f_{*} \|_{m, \Omega}^2} \, .
	\end{equation}
Moreover, if $\Omega$ has a boundary of class $\mathcal{C}^{m+4}$ and $C^{(m)}_{B}(\Omega) > 2$, then there exists $f_{*} \in A_{m}(\Omega, \mathbb{R})^{\perp} \setminus \{0\}$ satisfying \eqref{bogoattho}.
\end{theorem}

\subsection{Higher-order Bogovskii constant of a ball}

In this section we prove
	
\begin{theorem}\label{finalresult}
	Let $m \in \mathbb{N}$ and $\B \subset\R^n$, $n \geq 2$, be the unit ball. Then,
	\begin{equation}\label{finalresultn}
		C^{(m)}_{B}(\B) = n \, .
	\end{equation}
\end{theorem}
\noindent
\begin{proof}
We first recall the notation of Section \ref{sec:4_1}: for every  $k\in \mathbb{N}$, $\widetilde{\mathcal M}_k$ denotes the space of
homogeneous harmonic polynomials of degree $k$ (with $\widetilde{\mathcal M}_0=\R$), having dimension $N(k,n)$ and
$$
\{\widetilde{H}_{k,1},\ldots,\widetilde{H}_{k,N(k,n)}\} \subset \widetilde{\mathcal M}_k \,
$$
as an orthonormal basis with respect to the $L^2(\B, \mathbb{R})$-scalar product. It is well-known that
$$
\bigcup_{k=1}^\infty \{\widetilde{H}_{k,1},\ldots,\widetilde{H}_{k,N(k,n)}\}
$$
forms a complete orthonormal system in $A(\B, \mathbb{R})^{\perp}$ with respect to the $L^2(\B, \mathbb{R})$-scalar product (recall \eqref{decompl02}), see \cite[Proposition 4]{gaultier1992quelques} and \cite[Theorem 5.5.III,  Proposition 5.5.7]{Sauvigny_PDE_1}. We next claim that
\begin{equation} \label{basisho}
\bigcup_{k=1}^\infty \{(1\!-\!|x|^{2})^{m} \widetilde{H}_{k,1},\ldots,(1\!-\!|x|^{2})^{m} \widetilde{H}_{k,N(k,n)}\}\mbox{ forms a complete orthogonal system in $A_{m}(\B, \mathbb{R})^{\perp}$.}
\end{equation}
This proof is lengthy. Since \eqref{basisho} is known to hold when $m=0$ (Proposition \ref{thm:comm_1.4}), we assume $m \geq 1$.
Fix any pair of integers $k \in \mathbb{N}$, $j \in \{1,\ldots,N(k,n)\}$, and put
$$
Q_{k,j}(x) \doteq (1 - |x|^{2})^{m} \widetilde{H}_{k,j}(x) = \sum_{\ell=0}^{m} {m \choose \ell} (-1)^{\ell} r^{2\ell} \widetilde{H}_{k,j}(x) \qquad \forall x \in \B \, ,
$$
with the radial variable $r \doteq |x|$ for every $x \in \mathbb{R}^{n}$. We clearly have $Q_{k,j} \in \mathcal{C}^{\infty}(\overline{\B}, \mathbb{R}) \cap H_{0}^{m}(\B, \mathbb{R})$. Furthermore, since $\widetilde{H}_{k,j} \in \widetilde{\mathcal M}_k$, from the Mean Value Property it holds
$$
0 = | \B | \widetilde{H}_{k,j}(0) =\int_{\B} \widetilde{H}_{k,j}(x) \, dx = \int_{0}^{1} \int_{\mathbb{S}^{n-1}} r^{n-1} \widetilde{H}_{k,j}(r \xi) \, d\xi \, dr = \dfrac{1}{n+k} \int_{\mathbb{S}^{n-1}} \widetilde{H}_{k,j}(\xi) \, d\xi \, ,
$$
that is,
\begin{equation} \label{zerointegral}
\int_{\mathbb{S}^{n-1}} \widetilde{H}_{k,j}(\xi) \, d\xi = 0 \, .
\end{equation}
Owing to \eqref{zerointegral}, we also infer
$$
\int_{\B} Q_{k,j}(x) \, dx = \left( \int_{0}^{1} r^{n+k-1} (1 - r^{2})^{m} \, dr \right) \left( \int_{\mathbb{S}^{n-1}} \widetilde{H}_{k,j}(\xi) \, d\xi \right) = 0 \, ,
$$
thereby implying that $Q_{k,j} \in L_{0}^{2}(\B, \mathbb{R})$ and, subsequently, that $Q_{k,j} \in X_{m}(\B, \mathbb{R})$.

Recall, for every integer $\ell \geq 1$, the identities
\begin{equation} \label{euler}
\begin{aligned}
\nabla r^{2\ell } = 2\ell  r^{2\ell -2} x \, , \quad \Delta r^{2\ell } = 2\ell  (2\ell  + n - 2) r^{2\ell-2} \, , \quad x \cdot \nabla \widetilde{H}_{k,j}(x) = k \widetilde{H}_{k,j}(x) \qquad \forall x \in \mathbb{R}^{n} \, .
\end{aligned}
\end{equation}
We then compute, for all $\ell \in \{1,\ldots,m\}$,
$$
\begin{aligned}
& \Delta (r^{2\ell} \widetilde{H}_{k,j}(x)) = \widetilde{H}_{k,j}(x) (\Delta r^{2\ell}) + 2 (\nabla r^{2\ell} \cdot \nabla \widetilde{H}_{k,j}(x)) = 2\ell (2\ell + n + 2k - 2)  r^{2\ell-2} \widetilde{H}_{k,j}(x) \, , \\[6pt]
& \Delta^{2} (r^{2\ell} \widetilde{H}_{k,j}(x)) = (2\ell)(2\ell-2)(2\ell + n + 2k - 2)(2\ell + n + 2k - 4) r^{2\ell-4} \widetilde{H}_{k,j}(x) \, , \\[6pt]
& \Delta^{3} (r^{2\ell} \widetilde{H}_{k,j}(x)) = (2\ell)(2\ell-2)(2\ell-4)(2\ell + n + 2k - 2)(2\ell + n + 2k - 4)(2\ell+n+2k-6) r^{2\ell-6} \widetilde{H}_{k,j}(x) \, .
\end{aligned}
$$
By induction we infer that
\begin{equation} \label{mlaplacianformulalem}
\Delta^{m} (r^{2\ell} \widetilde{H}_{k,j}(x)) \equiv 0 \quad \forall \ell \in \{1,\ldots,m-1\} \, , \quad
\Delta^{m} (r^{2m} \widetilde{H}_{k,j}(x)) = \dfrac{(2m)!! \, (2m+n+2k-2)!!}{(n+2k-2)!!} \widetilde{H}_{k,j}(x)
\end{equation}
for all $x \in \B$ and, consequently,
\begin{equation} \label{mlaplacianformulalem2}
\Delta^{m} Q_{k,j}(x) = (-1)^{m} \dfrac{(2m)!! \, (2m+n+2k-2)!!}{(n+2k-2)!!} \widetilde{H}_{k,j}(x)\, ,\qquad\Delta^{m+1} Q_{k,j}(x) = 0\, , \qquad \forall x \in \B \, .
\end{equation}
Here we use the notation for $\ell\in\N$:
$$
\ell!!=\begin{cases}
\Pi_{k=1}^{\ell/2}(2k),\quad \mbox{if $\ell$ is even,}\\[2mm]
\Pi_{k=1}^{(\ell+1)/2}(2k-1),\quad \mbox{if $\ell$ is odd.}
\end{cases}
$$
From \eqref{mlaplacianformulalem2}$_2$ we deduce that $Q_{k,j} \in A_{m}(\B, \mathbb{R})^{\perp}$ for every $k \in \mathbb{N}$ and $j \in \{1,\ldots,N(k,n)\}$.

We now proceed to show the orthogonality and completeness of the set \eqref{basisho} within $A_{m}(\B, \mathbb{R})^{\perp}$.

\underline{Orthogonality:} Take $(k,j) \in \mathbb{N} \times \{1,\ldots,N(k,n)\}$ and $(\ell,i) \in \mathbb{N} \times \{1,\ldots,N(\ell,n)\}$ such that $k \neq \ell$ or $j \neq i$, thereby implying
$$
0 = \int_{\B} \widetilde{H}_{k,j}(x) \widetilde{H}_{\ell,i}(x) \, dx = \int_{0}^{1} \int_{\mathbb{S}^{n-1}} r^{n-1} \widetilde{H}_{k,j}(r \xi) \widetilde{H}_{\ell,i}(r \xi) \, d\xi \, dr = \dfrac{1}{n + k + \ell} \int_{\mathbb{S}^{n-1}} \widetilde{H}_{k,j}(\xi) \widetilde{H}_{\ell,i}(\xi)  \, d\xi \, ,
$$
that is,
\begin{equation} \label{zerointegral1}
	\int_{\mathbb{S}^{n-1}} \widetilde{H}_{k,j}(\xi) \widetilde{H}_{\ell,i}(\xi) \, d\xi = 0 \, .
\end{equation}
Owing to \eqref{mlaplacianformulalem2}$_1$-\eqref{zerointegral1}, we can integrate by parts in $\B$ to infer
$$
\begin{aligned}
& (Q_{k,j}, Q_{\ell,i})_{m, \B} = (-1)^{m} \int_{\B} (\Delta^{m} Q_{k,j}(x)) Q_{\ell,i}(x) \, dx \\[6pt]
& = \dfrac{(2m)!! \, (2m+n+2k-2)!!}{(n+2k-2)!!} \int_{\B} (1 - |x|^{2})^{m} \widetilde{H}_{k,j}(x) \widetilde{H}_{\ell,i}(x) \, dx \\[6pt]
& = \dfrac{(2m)!! \, (2m+n+2k-2)!!}{(n+2k-2)!!} \left( \int_{0}^{1} r^{n+k+\ell-1} (1 - r^{2})^{m} \, dr \right) \left( \int_{\mathbb{S}^{n-1}} \widetilde{H}_{k,j}(\xi) \widetilde{H}_{\ell,i}(\xi) \, d\xi \right) = 0 \, .
\end{aligned}
$$

\underline{Completeness:} Take any $q \in A_{m}(\B, \mathbb{R})^{\perp}\subset X_m= H^m_0\cap L^2_0$ such that
\begin{equation} \label{zerointegral2}
	(q, Q_{k,j})_{m, \B} = 0 \qquad \forall k \in \mathbb{N} \, , \ \ \forall j \in \{1,\ldots,N(k,n)\} \, .
\end{equation}
We integrate by parts in $\B$ to infer
$$
	(q, Q_{k,j})_{m, \B} = (-1)^{m} \int_{\B} q(x) (\Delta^{m} Q_{k,j}(x)) \, dx = \dfrac{(2m)!! \, (2m+n+2k-2)!!}{(n+2k-2)!!} \int_{\B} q(x) \widetilde{H}_{k,j}(x) \, dx \, ,
$$
so that
\begin{equation} \label{zerointegral3}
	\int_{\B} q(x) \widetilde{H}_{k,j}(x) \, dx = 0 \qquad \forall k \in \mathbb{N} \, , \ \ \forall j \in \{1,\ldots,N(k,n)\} \, .
\end{equation}
Since $q \in L_{0}^{2}(\B, \mathbb{R})$, from the decomposition \eqref{decompl02} it follows the existence of a unique pair of functions $(\psi, \phi) \in H_{0}^{2}(\B, \mathbb{R}) \times A(\B, \mathbb{R})^{\perp}$ such that $q = \Delta \psi + \phi$ in $\B$. After noticing that
$$
\int_{\B} (\Delta \psi(x)) \widetilde{H}_{k,j}(x) \, dx = \int_{\B} \psi(x) \Delta \widetilde{H}_{k,j}(x)  \, dx = 0 \qquad \forall k \in \mathbb{N} \, , \ \ \forall j \in \{1,\ldots,N(k,n)\} \, ,
$$
we deduce from \eqref{zerointegral3} that
\begin{equation} \label{zerointegral4}
	\int_{\B} \phi(x) \widetilde{H}_{k,j}(x) \, dx = 0 \qquad \forall k \in \mathbb{N} \, , \ \ \forall j \in \{1,\ldots,N(k,n)\} \, .
\end{equation}
As $\phi \in A(\B, \mathbb{R})^{\perp}$, \eqref{zerointegral4} implies $\phi \equiv 0$, which leads to $q = \Delta \psi$ in $\B$. Since $q \in H_{0}^{m}(\B, \mathbb{R})$, by elliptic regularity, this ensures that
$\psi \in H^{m+2}(\B, \mathbb{R}) \cap H_{0}^{2}(\B)$. Letting $\Lambda_{n}$ denote the Laplace-Beltrami operator on the unit sphere $\mathbb{S}^{n-1}$ (which only involves tangential derivatives on $\partial \B$), it holds
$$
0 =q= \Delta \psi = \dfrac{\partial^{2} \psi}{\partial r^{2}} + (n-1) \dfrac{\partial \psi}{\partial r} + \Lambda_{n}(\psi) = \dfrac{\partial^{2} \psi}{\partial r^{2}} \quad \text{on} \ \ \partial \B \, ,
$$
implying $\psi \in H_{0}^{3}(\B, \mathbb{R})$. But since
$$
\Delta \left(\nabla \psi \right) = \nabla (\Delta \psi) = \nabla q\quad \text{in} \ \ \B \, ,
$$
we deduce again, provided that $m\ge 2$
$$
0 = \Delta \left(\nabla \psi  \right)  = \dfrac{\partial^{2} }{\partial r^{2}}\nabla \psi  + (n-1) \dfrac{\partial }{\partial r} \nabla \psi + \Lambda_{n}\left(\nabla \psi  \right) = \dfrac{\partial^{2} }{\partial r^{2}}\nabla \psi \quad \Longrightarrow\quad \dfrac{\partial^{3} \psi}{\partial r^{3}}=0  \quad \text{on} \ \ \partial \B \, ,
$$
that is, $\psi \in H_{0}^{4}(\B, \mathbb{R})$. Iterating this procedure, we infer that
$$
\dfrac{\partial^{\ell} \psi}{\partial r^{\ell}} = 0 \quad \text{on} \ \ \partial \B \, , \ \ \forall \ell \in \{0,\ldots,m+1\}
$$
or, equivalently, $\psi \in H_{0}^{m+2}(\B, \mathbb{R})$. Consequently, as $q$ equals the Laplacian of a function in $H_{0}^{m+2}(\B, \mathbb{R})$, we deduce $q \in A_{m}(\B, \mathbb{R}) \cap A_{m}(\B, \mathbb{R})^{\perp}$, that is, $q \equiv 0$. This concludes the proof of \eqref{basisho}.\par

\bigskip
We turn now to proving \eqref{finalresultn}.
As in the proof of \eqref{basisho}, fix any pair of integers $k \geq 1$ and $j \in \{1,\ldots,N(k,n)\}$. We put
$$
Q_{k,j}(x) \doteq (1 - |x|^{2})^{m} \widetilde{H}_{k,j}(x) = \sum_{\ell=0}^{m} {m \choose \ell} (-1)^{\ell} r^{2\ell} \widetilde{H}_{k,j}(x) \qquad \forall x \in \B \, ,
$$
with $r \doteq |x|$ for every $x \in \B$. Recall, from the proof of \eqref{basisho}, that $Q_{k,j} \in A_{m}(\B, \mathbb{R})^{\perp} \cap \mathcal{C}^{\infty}(\overline{\B}, \mathbb{R})$. Moreover, the identity \eqref{mlaplacianformulalem2}$_1$ implies
\begin{equation} \label{gradlap1}
\nabla ((- \Delta)^{m} Q_{k,j})(x) = \dfrac{(2m)!! \, (2m+n+2k-2)!!}{(n+2k-2)!!} \nabla \widetilde{H}_{k,j}(x) \qquad \forall x \in \B \, .
\end{equation}

Define now the vector field $V_{k,j} \in \mathcal{C}^{\infty}(\overline{\B}, \mathbb{R}^{n})$ by the formula
	$$
	V_{k,j}(x) \doteq -\dfrac{(1 - |x|^{2})^{m+1} }{(2m+2)(n+2k-2)} \nabla \widetilde{H}_{k,j}(x) \qquad \forall x \in \B \, ,
	$$
	so that $V_{k,j} \in H^{m+1}_{0}(\B, \mathbb{R}^{n})$. Recalling \eqref{euler}$_3$, a straightforward computation shows that
	\begin{equation} \label{binomial0}
		\begin{aligned}	
		& (\nabla \cdot V_{k,j})(x) = -\dfrac{1}{(2m+2)(n+2k-2)} \nabla ((1 - |x|^{2})^{m+1}) \cdot \nabla \widetilde{H}_{k,j}(x) \\[6pt]
		& = \dfrac{1}{n+2k-2} (1 - |x|^{2})^{m} (x \cdot \nabla \widetilde{H}_{k,j}(x)) = \dfrac{k}{n+2k-2} Q_{k,j}(x) \qquad \forall x \in \B \, .
		\end{aligned}
	\end{equation}
On the other hand, given any index $i \in \{1,\ldots,n\}$, we have
$$
V_{k,j}(x) \cdot \widehat{e}_{i} = -\dfrac{1}{(2m+2)(n+2k-2)} \sum_{\ell=0}^{m+1} {m+1 \choose \ell} (-1)^{\ell} r^{2\ell} \dfrac{\partial \widetilde{H}_{k,j}}{\partial x_{i}}(x) \qquad \forall x \in \B \, .
$$
Firstly, suppose that $k=1$, so that $\widetilde{H}_{1,j}$ is simply a linear function in $\mathbb{R}^{n}$ and, therefore, $\dfrac{\partial \widetilde{H}_{1,j}}{\partial x_{i}}$ is a constant function. Owing to \eqref{gradlap1}, it holds
\begin{equation} \label{binomial1}
\begin{aligned}
& (- \Delta)^{m+1} V_{1,j}(x) \cdot \widehat{e}_{i} = -\dfrac{1}{n(2m+2)} \dfrac{\partial \widetilde{H}_{1,j}}{\partial x_{i}}(x) (\Delta^{m+1} r^{2m+2}) \\[6pt]
& = -\dfrac{1}{n(2m+2)} \dfrac{(2m+2)!! \, (2m+n)!!}{(n-2)!!} \dfrac{\partial \widetilde{H}_{1,j}}{\partial x_{i}}(x) \\[6pt]
& = - \dfrac{(2m)!! \, (2m+n)!!}{n!!} \dfrac{\partial \widetilde{H}_{1,j}}{\partial x_{i}}(x)  = - \nabla ((- \Delta)^{m} Q_{1,j})(x) \cdot \widehat{e}_{i} \qquad \forall x \in \B \, .
\end{aligned}
\end{equation}
Secondly, if $k \geq 2$, it follows that $\tfrac{\partial \widetilde{H}_{k,j}}{\partial x_{i}} \in \widetilde{\mathcal M}_{k-1}$
and, hence, the computations that led to \eqref{mlaplacianformulalem2} yield
\begin{equation} \label{binomial2}
	\begin{aligned}
		& (- \Delta)^{m+1} V_{k,j}(x) \cdot \widehat{e}_{i} = -\dfrac{1}{(2m+2)(n+2k-2)} \dfrac{(2m+2)!! \, (2m+n+2k-2)!!}{(n+2k-4)!!} \dfrac{\partial \widetilde{H}_{k,j}}{\partial x_{i}}(x) \\[6pt]
		& = - \dfrac{(2m)!! \, (2m+n+2k-2)!!}{(n+2k-2)!!} \dfrac{\partial \widetilde{H}_{k,j}}{\partial x_{i}}(x) = - \nabla ((- \Delta)^{m} Q_{k,j})(x) \cdot \widehat{e}_{i} \qquad \forall x \in \B \, .
	\end{aligned}
\end{equation}
Summarising, from \eqref{binomial0}-\eqref{binomial1}-\eqref{binomial2} we deduce, for every $k \geq 1$ and $j \in \{1,\ldots,N(k,n)\}$,
\begin{equation} \label{binomial4}
(-\Delta)^{m+1} V_{k,j} + \nabla ((- \Delta)^{m} Q_{k,j}) = 0 \, , \quad \nabla \cdot V_{k,j} = \dfrac{k}{n+2k-2} Q_{k,j} \qquad \text{in} \ \ \B
\end{equation}
or, equivalently,
\begin{equation} \label{binomial5}
\mathcal{S}_{m}(Q_{k,j}) = \dfrac{k}{n+2k-2} Q_{k,j} \qquad \forall k \geq 1 \, , \ \ \forall j \in \{1,\ldots,N(k,n)\} \, .
\end{equation}

By \eqref{basisho}, for every $q \in A_{m}(\B, \mathbb{R})^{\perp} \setminus \{0\}$ we have
\begin{equation} \label{binomial7}
q = \sum_{k=1}^{+\infty} \sum_{j=1}^{N(k,n)} \dfrac{(q,Q_{k,j})_{m,\B}}{\| Q_{k,j} \|_{m,\B}^{2}} Q_{k,j} \quad \text{in} \ \ \B \, ; \qquad \| q \|_{m,\B}^{2} = \sum_{k=1}^{+\infty} \sum_{j=1}^{N(k,n)} \dfrac{(q,Q_{k,j})^{2}_{m,\B}}{\| Q_{k,j} \|_{m,\B}^{2}} \, .
\end{equation}
Identities \eqref{binomial5}-\eqref{binomial7}$_1$ then entail
$$
\mathcal{S}_{m}(q) = \sum_{k=1}^{+\infty} \sum_{j=1}^{N(k,n)} \dfrac{k}{n+2k-2} \dfrac{(q,Q_{k,j})_{m,\B}}{\| Q_{k,j} \|_{m,\B}^{2}} Q_{k,j} \, ,
$$
thereby yielding, as a consequence of \eqref{binomial7}$_2$,
\begin{equation} \label{binomial8}
(\mathcal{S}_{m}(q),q)_{m, \B} = \sum_{k=1}^{+\infty} \sum_{j=1}^{N(k,n)} \dfrac{k}{n+2k-2} \dfrac{(q,Q_{k,j})^{2}_{m,\B}}{\| Q_{k,j} \|_{m,\B}^{2}} \geq \dfrac{1}{n} \| q \|_{m,\B}^{2} \qquad \forall q \in A_{m}(\B, \mathbb{R})^{\perp} \setminus \{0\} \, ,
\end{equation}
which is also valid for $q \equiv 0$. Given any $p \in X_{m}(\B, \mathbb{R}) \setminus \{0\}$, it can be decomposed as $p = h + q$ in $\B$, for a unique pair $(h,q) \in A_{m}(\B, \mathbb{R}) \times A_{m}(\B, \mathbb{R})^{\perp}$. In the light of \eqref{sondecompho}-\eqref{binomial8}, we infer
$$
\dfrac{(\mathcal{S}_{m}(p),p)_{m, \B}}{\|p\|_{m,\B}^2} = \dfrac{\|h\|_{m,\B}^2 + (\mathcal{S}_{m}(q),q)_{m, \B}}{\|h\|_{m,\B}^2 + \|q\|_{m,\B}^2} \geq \dfrac{1}{n} \qquad \forall p \in X_{m}(\B, \mathbb{R}) \setminus \{0\} \, ,
$$
so that
$$
\dfrac{1}{C^{(m)}_{B}(\B)} = \inf_{p \in X_{m}(\B, \mathbb{R}) \setminus \{0\}} \dfrac{(\mathcal{S}_{m}(p),p)_{m, \B}}{\|p\|_{m,\Omega}^2} \geq \dfrac{1}{n} \, .
$$
Since we also have $C^{(m)}_{B}(\B) \geq n$ (see Theorem \ref{lowerboundHO}), the proof is concluded.
\end{proof}

\begin{appendix}
	
\section{Examples}\label{sec:examples}

\subsection{How to find $u_f$ for some $f\in L^2_0(\B,\R)$?}\label{howto}

We recall from the beginning of the proof of Proposition \ref{lem:comm_1.6} that
$$
\mbox{for all $f\in L^2_0 (\Omega,\R)$, a solution $u\in H^1_{0}(\Omega,\R^n)$ to \eqref{Bogdom} coincides with $u_f$ if and only if }
u\in H^1_{0,\sigma}(\Omega,\R^n)^\perp.
$$
Or in other words, in order to find $u_f$ one may first try to find some $u\in H^1_{0}(\Omega,\R^n)$ solving \eqref{Bogdom} and then to decompose it into
$$
u=u_0+u_f \quad \mbox{with}\quad  u_0\in H^1_{0,\sigma}(\Omega,\R^n) \quad \mbox{and}\quad  u_f \in H^1_{0,\sigma}(\Omega,\R^n)^\perp.
$$
This means that one has to find the orthogonal projection of $u$ onto $H^1_{0,\sigma}(\Omega,\R^n)^\perp$.
Projection operators are highly nonlocal and often act counterintuitively. The following examples are intended to give a feeling of how this projection operator acts.

Let us first consider the simplest nontrivial example $f\in L^2_0(\B,\R)$, $f(x)=x_i$ for any $i\in \{1,\ldots,n\}$. Then it is an obvious guess to consider $u(x)=(|x|^2-1)e_i $ as a solution to \eqref{Bogdom}. From Proposition~\ref{thm:comm_1.4} one takes then that one even has $u=u_f$ and that moreover, this pair $(f,u_f)$ maximises $C_B(\B)$.

Next we take $f(x)=|x|^2 x_1$ so that an obvious guess for a  solution $u\in H^1_0(\B,\R^n)$ to \eqref{Bogdom} is
$$
u(x)= \frac14 (|x|^4-1)e_1.
$$
In spite of its simple form it {\em does not} coincide with $u_f$, which is given in the following statement.
In particular, symmetry properties of $f$ are not fully inherited by the corresponding solution $u_f$.
\begin{proposition}
For $f(x)=|x|^2 x_1$ the privileged solution $u_f\in H^1_0(\B,\R^n)$ of \eqref{Bogdom} is given by
\begin{equation}\label{nablauff}
u_f(x)=
\frac{|x|^2-1}{n+4}\begin{pmatrix}
		\frac14(|x|^2+1)+x_1^2+\frac{n+1}{2}\\
		x_1x_2\\
		\vdots\\
		x_1 x_n
	\end{pmatrix}
\quad\mbox{and}\quad	
\int_\B |\nabla u_f|^2\, dx= \left(n-4\frac{n-1}{(n+4)^2}\right) \int_\B f^2\, dx.
\end{equation}
\end{proposition}
\noindent
\begin{proof} Regularity and attainment of the boundary conditions are obvious. The fact that $u_f$ solves \eqref{Bogdom} follows from
	\begin{align*}
		\nabla u_f =&\frac{1}{n+4}
		\begin{pmatrix}
			\left( |x|^2+2x_1^2+n+1\right) x^T + 2x_1 (|x|^2-1)e_1^T \\
			2x_1x_2 x^T +	x_2 (|x|^2-1)e_1^T + x_1 (|x|^2-1)e_2^T \\
			\vdots \\
			2x_1x_n x^T +	x_n (|x|^2-1)e_1^T + x_1 (|x|^2-1)e_n^T
		\end{pmatrix}.
	\end{align*}
	
	In order to verify 	the necessary and sufficient condition \eqref{ELeq} we calculate further
	\begin{align*}
		\Delta u_f(x) =&\frac{1}{n+4}
		\begin{pmatrix}
			n\left( |x|^2+2x_1^2+n+1\right)+2|x|^2+4x_1^2  + 2 (|x|^2-1) +4x_1^2\\
			2(n+2)x_1x_2  +	2x_1 x_2   +	2x_1 x_2 \\
			\vdots \\
			2(n+2)x_1x_n  +	2x_1 x_n  +	2x_1 x_n
		\end{pmatrix}\\
		=&\frac{1}{n+4}
		\begin{pmatrix}
			(n+4) |x|^2+  2(n+4) x_1^2+ n(n+1)-2\\
			2(n+4)x_1x_2 \\
			\vdots \\
			2(n+4)x_1x_n
		\end{pmatrix}\\
		=&
		\begin{pmatrix}
			|x|^2+  2x_1^2+ \frac{(n+2)(n-1)}{n+4}\\
			2x_1x_2 \\
			\vdots \\
			2x_1x_n
		\end{pmatrix}
		=\nabla  \left(  |x|^2x_1 +  \frac{(n+2)(n-1)}{n+4}x_1\right)^T
		=(\nabla \psi(x))^T,
	\end{align*}
	where $\psi(x)=|x|^2x_1+\frac{(n+2)(n-1)}{n+4}x_1$. With this we find
	$$
	\int_\B \nabla u_f : \nabla v\, dx=-\int_\B\Delta u_f \cdot v\, dx
	=-\int_\B\nabla \psi  \cdot v\, dx=\int_\B\psi \underbrace{(\nabla \cdot v)}_{\equiv 0}\, dx=0\qquad
	\forall v\in H^1_{0,\sigma}(\B,\R^n).
	$$
	In view of the necessary and sufficient condition \eqref{ELeq} we see that $u_f$ is indeed the ``privileged'' solution, as claimed.\par
	The proof of \eqref{nablauff}$_2$ is obtained by direct computations.\end{proof}

The idea to obtain $u_f$ in this example can be explained as follows. The necessary and sufficient condition \eqref{ELeq} requires
$\Delta u_f$  to be a gradient field with a potential, say $\psi\in H^2(\B,\R)$, see \eqref{gstokes} in Theorem~\ref{minBog}. If $(\psi-f)$ had
homogeneous boundary conditions (which is not the case) from $\Delta \psi =\Delta f$ one would have
$\psi=f$. So, a correction term in $\psi$ is needed and the main issue was to find the right factor in front of $x_1$.
In two dimensions this could be determined explicitly by means of a stream function-ansatz, to be added to $u$. In general dimensions a series of ``educated guesses'' eventually gave the right factor.

\subsection{On the failure of extension arguments}
\label{sec:nonmonotonicity}

In spite of the fact that \eqref{bogo} involves only first order derivatives, e.g. Lemma \ref{invariance} and \eqref{balln}, {combined e.g.\
with Theorem~\ref{thm:comm_4.1}} show that there is no monotonicity with respect to domain inclusions.
The reason for this is that arguing by trivial extension is not admissible for determining the Bogovskii constant \eqref{bogo}.
The following example shows how, for a pair $(f,u_f)$ maximising the Bogovskii constant, the properties being $u_f$ break down as well as
the maximality of $f$. Again, this is an effect of the nonlocality of the projection operator mentioned in the previous subsection.

Since the calculations are quite involved, we restrict ourselves to the planar case $n=2$. Again, according to Proposition
\ref{thm:comm_1.4} we consider the maximising pair $(f,u_f)$ in $\B=:\Omega_1$ given by $f(x)=x_1$ and
$$
u_f(x_1,x_2)= \frac12 (|x|^2-1)\begin{pmatrix}1\\0 \end{pmatrix},
\qquad \nabla u_f(x_1,x_2)= \begin{pmatrix} x_1&x_2\\ 0&0 \end{pmatrix},
$$
and extend $f$ trivially by $0$ to $\Omega_2=\B_2$. For
\begin{equation}\label{againf}
	f \in L^2_0(\B_2), \qquad  f(x)=\chi_\B (x) x_1\, ,
\end{equation}
we  claim that
\begin{equation}\label{ansatz}
	u_f(x)=\begin{pmatrix}\frac{1}{128} |x|^2 +\frac{1}{64}x_2^2 +\frac{3}{16}\\-\frac{1}{64} x_1 x_2\end{pmatrix}
	+\chi_{\B} (x) \begin{pmatrix}\frac{3}{8}|x|^2 -\frac{1}{4}x_2^2 -\frac{1}{2}\\ \frac{1}{4} x_1 x_2\end{pmatrix}
	+\chi_{\B_2\setminus \B} (x)\begin{pmatrix}-\frac14+\frac{1}{8|x|^2}-\frac{x_2^2}{4|x|^4}\\ \frac{x_1 x_2}{4 |x|^4}	\end{pmatrix}.
\end{equation}
To see this, we need to verify the three properties:
$$
u_f\in H^1_0(\B_2),\qquad \nabla \cdot u_f(x)=\chi_{\B} (x)x_2=f(x)\mbox{ a.e. in }\B_2,\qquad\mbox{\eqref{gstokes}}.
$$
The first two properties can be verified through direct computations:
$$
\nabla u_f(x)=\frac{1}{64}\begin{pmatrix} x_1 & 3x_2\\ -x_2 &-x_1 \end{pmatrix}
+\frac14\chi_{\B} (x) \begin{pmatrix} 3x_1 & x_2\\ x_2 & x_1 \end{pmatrix}
+\chi_{\B_2\setminus \B} (x)\begin{pmatrix}-\frac{x_1}{4|x|^4}+\frac{x_1x_2^2}{|x|^6} & -\frac{3x_2}{4|x|^4}+\frac{x_2^3}{|x|^6} \\
	\frac{x_2}{4|x|^4}-\frac{x_1^2x_2}{|x|^6}  & \frac{x_1}{4|x|^4}-\frac{x_1x_2^2}{|x|^6}\end{pmatrix}.
$$
To prove \eqref{gstokes} we introduce the function
$$
\tilde\psi (x)\doteq \chi_{\B_2\setminus \B} (x)\left(\frac{1}{8|x|^2}-\frac14\right) x_2 -\frac18 \chi_{\B} (x) |x|^2 x_2
$$
that satisfies $\tilde\psi\in W^{2,\infty}(\B_2)$ and
\begin{equation}\label{Deltapsi}
	\Delta \tilde\psi (x)= -\chi_{\B} (x) x_2\quad\mbox{a.e.\ in }\B_2.
\end{equation}
Since $\tilde\psi$ does not have double trace vanishing of $\partial\B_2$,
we add a biharmonic correction, that is, we take
$$
\psi(x)=\tilde\psi (x)+\frac{1}{128}|x|^2 x_2 +\frac{3}{16}x_2=
\frac{1}{128}|x|^2x_2+\frac{3}{16}x_2+\chi_{\B_2\setminus \B}(x)\left(\frac{1}{8|x|^2}-\frac14\right)x_2-\frac18 \chi_{\B}(x)|x|^2 x_2
$$
and we obtain $\psi\in H^2_0(\B_2)$. From \eqref{ansatz} we see that $u_f=u+\nabla^\perp\psi=u+(\psi_{x_2},-\psi_{x_1})$  and, hence, \eqref{gstokes} follows if $\psi$ is
a distributional   solution to
\begin{equation}\label{distsol}
	\Delta^2 \psi =\Delta \big(u_{2,x_1}-u_{1,x_2}\big)=-\Delta u_{1,x_2}\qquad\mbox{in }\mathcal{D}'(\B_2).
\end{equation}
So, let us check the validity of \eqref{distsol}. First, we notice that
$$
\langle\Delta u_{1,x_2},\varphi\rangle =\int_{\B_2}u_{1,x_2}\Delta\varphi=\int_{\B}x_2\Delta\varphi\qquad\forall\varphi\in C^\infty_c(\B_2)\, .
$$
Then we compute
$$
\langle \Delta^2\psi,\varphi\rangle =\langle\Delta^2\tilde\psi,\varphi\rangle=
\int_{\B_2}\Delta\tilde\psi\Delta\varphi\stackrel{\eqref{Deltapsi}}{=}-\int_\B x_2\Delta\varphi=-\langle\Delta u_{1,x_2},\varphi\rangle\qquad\forall\varphi\in C^\infty_c(\B_2)
$$
and, hence, \eqref{distsol} holds. This completes the proof of \eqref{ansatz}.

Finally, we  compute $\|f\|^2_{L^2(\B_2)}$ and $\| \nabla u_f \|_{L^2(\B_2)}^2$. By making use of symmetry arguments we obtain
\begin{align*}
	\int_{\B_2} |f|^2\, dx=&\frac{\pi}{4}\, ,\qquad\int_{\B_2} |\nabla u|^2\, dx=\frac{\pi}{2}=2\int_{\B_2} |f|^2\, dx,\, ,\\
	\int_{\B_2} |\nabla u_f|^2\, dx =&\frac{1}{64^2}\int_\B \left( (49^2+15^2)x_1^2 +(15^2+19^2)x_2^2\right)\, dx\\
	&+ \frac{1}{2048}\int_{\B_2\setminus \B} \frac{1}{|x|^{10}} \Big(|x|^{12}+4x_2^2|x|^{10}+32|x|^6(x_2^2-x_1^2)+256|x|^4\Big)\, dx\\
	=&\frac{803}{2048}\int_\B |x|^2\, dx+ \frac{1}{2048}\int_{\B_2\setminus \B}  \Big(3|x|^{2}+256|x|^{-6}\Big)\, dx=
\frac{17}{64}\pi=\frac{17}{16}\int_{\B_2} |f|^2\, dx.
\end{align*}
This ratio illustrates how far the trivially extended $u$ is from the minimising $u_f$ and how far the trivial extension of $f$ becomes apart from being the maximising one for $C_B(\B_2)=2$, see \eqref{balln}.
	
\end{appendix}

\bigskip

\noindent
{\bf Data availability statement.} Data sharing not applicable to this article as no datasets were generated or analyzed during the current study.
\par\smallskip
\noindent
{\bf Conflict of interest statement}.  The Authors declare that they have no conflict of interest.
\par\smallskip
\noindent
{\bf Acknowledgements.} The second  and third author are grateful to the first author for his hospitality and to
the Department of Mathematics of the Politecnico di Milano for supporting their visits  by means of  the MUR grant {\em Dipartimento di Eccellenza 2023-27} (Italy).
Filippo Gazzola is also partially supported by INdAM. The research of Gianmarco Sperone is supported by the \textit{Chilean National Agency for Research and Development} (ANID) through the \textit{Fondecyt Iniciación} grant 11250322.\par

\phantomsection
\addcontentsline{toc}{section}{References}
\bibliographystyle{abbrv}
\bibliography{references}

\vspace{5mm}
\noindent
\begin{minipage}{53mm}
{\footnotesize	\textbf{Filippo Gazzola}\\
	Dipartimento di Matematica\\
	Politecnico di Milano\\
	Piazza Leonardo da Vinci 32\\
	20133 Milan - Italy\\
	E-mail: filippo.gazzola@polimi.it}
	\end{minipage}
\begin{minipage}{60mm}
	{\footnotesize	\textbf{Hans-Christoph Grunau}\\
		Fakult\"at f\"ur Matematik\\
		Otto-von-Guericke-Universit\"at\\
		Postfach 4120\\
		39016 Magdeburg - Germany\\
		E-mail: hans-christoph.grunau@ovgu.de}
\end{minipage}
\begin{minipage}{57mm}
{\footnotesize	\textbf{Gianmarco Sperone}\\
	Facultad de Matemáticas\\
	Pontificia Universidad Católica de Chile\\
	Avenida Vicuña Mackenna 4860\\
	7820436 Santiago - Chile\\
	E-mail: gianmarco.sperone@uc.cl}
\end{minipage}
\end{document}